%% file: mainpaper.tex
\documentclass[11pt]{article}

\usepackage{amsmath, amsfonts, amsthm, amssymb, mathtools}
\usepackage{authblk, color, bm, graphicx, epstopdf, url}
\usepackage[font = small, labelfont = bf, labelsep = period]{caption}
\usepackage{subcaption}
\usepackage{xspace}
\usepackage{enumitem}
\usepackage{accents}

\usepackage[colorlinks=true,linkcolor=red]{hyperref}

\usepackage[title]{appendix}

\usepackage{xpatch}
\xpatchcmd{\proof}{\itshape}{\normalfont\proofnamefont}{}{}

\newcommand{\proofnamefont}{\bfseries}

\usepackage[backend=biber,sortcites=true,doi=false,url=false,giveninits=true,hyperref]{biblatex}
\renewbibmacro{in:}{}
\let\oldbibliography\thebibliography
\renewcommand{\thebibliography}[1]{\oldbibliography{#1}\setlength{\itemsep}{0pt}}
\newcommand{\citep}[2][]{\cite[#1]{#2}}
\newcommand{\citet}[2][]{\cite[#1]{#2}}

\usepackage{tikz}
\usetikzlibrary{arrows.meta,positioning,patterns}

\usepackage{booktabs, tabularx, multirow}
\newcolumntype{C}{>{\centering\arraybackslash}X}
\newcolumntype{R}{>{\raggedleft\arraybackslash}X}
\newcolumntype{L}{>{\raggedright\arraybackslash}X}

\newtheorem{proposition}{Proposition}

\newtheorem{remark}{Remark}
\newtheorem{theorem}{Theorem}
\newtheorem{corollary}{Corollary}
\newtheorem{lemma}{Lemma}
\newtheorem{example}{Example}

\newcommand{\norm}[1]{\lVert#1\rVert}

\newcommand{\cvar}{\mathbb{P}\text{-CVaR}}
\newcommand{\abs}[1]{\lvert#1\rvert}

\newcommand{\inner}[2]{\langle#1,#2\rangle}
\newcommand{\mb}[1]{\bm{#1}}
\newcommand{\one}{\bm{\mathrm{e}}}
\newcommand{\eye}{\mathbf{I}}
\newcommand{\vecmat}[1]{\mathop{\text{vec}}(#1)}

\newcommand{\conv}[1]{\mathop{\mathrm{cl\,conv}}(#1)}
\renewcommand{\i}{\mathrm{i}}
\DeclareMathOperator{\E}{\mathbb{E}}
\renewcommand{\P}{\mathbb{P}}

\newcommand{\re}[1]{{#1}}
\graphicspath{{figures/}}
\usepackage{grffile}

\usepackage{fullpage}[2cm]
\title{Data-Driven Two-Stage Conic Optimization with Zero-One Uncertainties}
\author[1]{Anirudh Subramanyam}
\author[2]{Mohamed El Tonbari}
\author[1]{Kibaek Kim}
\affil[1]{\small Argonne National Laboratory, Lemont, IL}
\affil[2]{\small Georgia Institute of Technology, Atlanta, GA}
\date{\today}

\begin{document}

\maketitle

\begin{abstract}
    We address high-dimensional zero-one random parameters in two-stage convex conic optimization problems.
    Such parameters typically represent failures of network elements and constitute rare, high-impact random events in several applications.
    Given a sparse training dataset of the parameters, we motivate and study a distributionally robust formulation of the problem %
    using a Wasserstein ambiguity set
    centered at the empirical distribution.
    We present a simple, tractable, and conservative approximation of this problem that can be efficiently computed and iteratively improved.
    Our method relies on a reformulation that optimizes over the convex hull of a mixed-integer conic programming representable set, followed by an approximation of this convex hull using lift-and-project techniques.
    We illustrate the practical viability and strong out-of-sample performance of our method on nonlinear optimal power flow \re{and multi-commodity network design problems} that are affected by random contingencies, and we report improvements of up to 20\% over existing \re{sample average approximation and two-stage robust optimization} methods.
\end{abstract}

\input{intro.tex}

\input{reformulation.tex}

\input{liftandproject.tex}

\input{computational_experiments.tex}

\input{conclusions.tex}

\section*{Acknowledgments}
This material is based upon work supported in part by the U.S. Department of Energy, Office of Science and Office of Electricity Delivery \& Energy Reliability, Advanced Grid Research and Development (under contract number DE-AC02-06CH11357), and in part by the Office of Naval Research (under Award number N00014-18-1-2075).

\begingroup
\setlength\bibitemsep{0pt}
\printbibliography
\endgroup

\newpage

\begin{appendices}
\input{example1.tex}
\input{complexity.tex}

\input{benders.tex}

\input{risk_averse.tex}

\input{model.tex}

\input{mmcf_model.tex}

\input{proofs.tex}
\end{appendices}

\newpage

\noindent\fbox{\parbox{0.97\textwidth}{
The submitted manuscript has been created by UChicago Argonne, LLC, Operator of Argonne National Laboratory (``Argonne''). Argonne, a U.S. Department of Energy Office of Science laboratory, is operated under Contract No. DE-AC02-06CH11357. The U.S. Government retains for itself, and others acting on its behalf, a paid-up nonexclusive, irrevocable worldwide license in said article to reproduce, prepare derivative works, distribute copies to the public, and perform publicly and display publicly, by or on behalf of the Government. The Department of Energy will provide public access to these results of federally sponsored research in accordance with the DOE Public Access Plan (http://energy.gov/downloads/doe-public-access-plan).}
}

\end{document}

%% file: intro.tex
\section{Motivation}

This work is motivated by optimization problems arising in applications that are affected by an extremely large, yet finite, number of rare, high-impact random events.
In particular, we are motivated by applications in which the decision-relevant random events consist of high-dimensional binary outcomes.
Such applications are ubiquitous in network optimization, where the uncertain events amount to failures of the nodes or edges of the underlying network.
For example, in electric power networks, random node and edge failures have been used to model losses of physical components such as substations, transmission lines, generators, and transformers~\citep{BienstockVerma2010,nercStandard}. %
Similarly, in natural gas~\citep{praks2017monte}, wireless communication~\citep{doshi1999optical}, and transportation networks~\citep{berche2009resilience}, they can be used to model failures of compressors and gas pipelines, antennas, and road links, respectively.

The challenges in modeling and solving such uncertainty-affected optimization problems are threefold.
First, the number of possible failure states grows \emph{exponentially} as the size of the input increases.
For example, the number of failure states in a network with $n$ failure-prone nodes is $2^n$.
Second, network failures are extremely \emph{rare} but \emph{critical}, and historical records are often not rich enough to include observations for every possible failure state. %
Third, failures of individual network elements are unlikely to be independent of each other.
For example, transmission line failures in electric power systems often have a cascading effect that triggers the failure of other transmission lines.
Because of this \emph{high dimensionality}, the true distribution governing the random parameters is often unknown and difficult to estimate with a small number of historical observations.

In this high-dimensional context, we focus on \emph{two-stage optimization} problems under uncertainty.
Suppose that there are $M$ uncertain parameters $\tilde{\xi}_1, \ldots, \tilde{\xi}_M$, and that $\Xi \subseteq \{0, 1\}^M$ is the support of the underlying distribution of these parameters.
If the distribution
$\mathbb{P}$ is known,
then the two-stage problem takes the form
\begin{equation*}\label{eq:conceptual}
\mathop{\text{minimize}}_{\bm{x} \in \mathcal{X}} \;
c(\bm{x})
+
\mathbb{E}_{\mathbb{P}} \left[
\mathcal{Q}(\bm{x}, \tilde{\bm{\xi}})
\right],
\end{equation*}
where $\bm{x}$ represents the \emph{first-stage} decisions that must be made agnostically to the realization of the random parameters, $\mathcal{X} \subseteq \mathbb{R}^{N_1}$ is a convex compact set of feasible first-stage decisions, $c: \mathcal{X} \mapsto \mathbb{R}$ is a convex function representing
the deterministic cost associated with these decisions, and $\mathcal{Q}(\bm{x}, \bm{\xi})$ is the \emph{random loss} (or \emph{second-stage cost}) corresponding to decisions $\bm{x}$ and a fixed realization $\bm{\xi} \in \Xi$ of the random parameters.
We assume that the loss can be computed by solving a convex conic optimization problem of the form
\begin{equation}\label{eq:second_stage_value_function}
\mathcal{Q}(\bm{x}, \bm{\xi}) =
\inf_{\bm{y} \in \mathcal{Y}}
\left\{
\bm{q}(\bm{\xi})^\top \bm{y} : \bm{W}(\bm{\xi}) \bm{y} \geq 
\bm{T}(\bm{x}) \bm{\xi} + \bm{h}(\bm{x})
\right\},
\end{equation}
where $\bm{y}$ denotes the \textit{second-stage} decisions that can be made after the realization $\bm{\xi}$ is known; $\mathcal{Y} \subseteq \mathbb{R}^{N_2}$ is a \textit{proper} (closed, convex, pointed, and full-dimensional) cone; %
$\bm{q} : \Xi \mapsto \mathbb{R}^{N_2}$ and $\bm{W} : \Xi \mapsto \mathbb{R}^{L \times N_2}$ are vector- and matrix-valued affine functions respectively, while $\bm{h} : \mathcal{X} \mapsto \mathbb{R}^{L}$ and $\bm{T} : \mathcal{X} \mapsto \mathbb{R}^{L \times M}$ are componentwise closed, proper, convex vector- and matrix-valued functions, respectively.
We allow uncertainty to affect only the affine constraints of the problem and, similarly, the first- and second-stage decisions to interact only via the affine part.

Since the true underlying distribution $\mathbb{P}$ is unknown, this two-stage optimization formulation is ill-posed.
Nevertheless, $\mathbb{P}$ is typically observable through a finite amount of historical data.
We assume that we have access to $N$ such \re{independent and identically distributed} observations, which we denote by $\{\hat{\bm{\xi}}^{(1)}, \ldots, \hat{\bm{\xi}}^{(N)}\}$.
We also assume that generating additional data (e.g., via Monte Carlo computer simulations) is either costly or impossible. Thus, it is imperative to use the given data most efficiently.

A popular approach for solving the two-stage problem in such cases is to use \emph{sample average approximation}~\citep{shapiro2009lectures}.
This approach replaces the true distribution with the empirical distribution $\hat{\mathbb{P}}_N = \frac{1}{N}\sum_{i = 1}^N \delta_{\hat{\bm{\xi}}^{(i)}}$, where $\delta_{\hat{\bm{\xi}}^{(i)}}$ denotes the Dirac distribution at $\hat{\bm{\xi}}^{(i)}$.
In the context of rare events, however, obtaining accurate estimates of the true distribution and, hence, the optimal solution of the true two-stage problem may require unrealistically large amounts of data.
The following example illustrates this phenomenon even if there are only $M=3$ uncertain parameters.

\begin{figure}[!b]
    \centering
\begin{tikzpicture}[scale=0.8,%
                    every label/.append style={rectangle, font=\scriptsize},
                    ed/.style = {-Latex},%
                    cr/.style = {circle, draw, minimum size = 0.5},%
                    crX/.style = {cr,pattern={north east lines},pattern color=gray!80}]
\scriptsize
\node[cr]                   (1) at (0,  1) {\phantom{1}};
\node[cr]                   (2) at (0, -1) {\phantom{2}};
\node[crX,label={$0.25\%$}] (3) at (2,  2) {\phantom{3}};
\node[crX,label={$1\%$}]    (4) at (2,  0) {\phantom{4}};
\node[crX,label={$0.25\%$}] (5) at (2, -2) {\phantom{5}};
\node[cr]                   (6) at (4,  1) {\phantom{6}};
\node[cr]                   (7) at (4, -1) {\phantom{7}};
\node (o1) at (-1.5,  1) {100};
\node (o2) at (-1.5, -1) {100};
\node (d6) at ( 5.5,  1) {100};
\node (d7) at ( 5.5, -1) {100};
\draw[ed] (o1) -- (1);
\draw[ed] (o2) -- (2);
\draw[ed] (1) --node[above] {$1$} (3);
\draw[ed] (1) --node[above] {$1-\epsilon$} (4);
\draw[ed] (2) --node[above] {$1-\epsilon$} (4);
\draw[ed] (2) --node[above] {$1$} (5);
\draw[ed] (3) --node[above] {$1$} (6);
\draw[ed] (4) --node[above] {$1-\epsilon$} (6);
\draw[ed] (4) --node[above] {$1-\epsilon$} (7);
\draw[ed] (5) --node[above] {$1$} (7);
\draw[ed] (6) -- (d6);
\draw[ed] (7) -- (d7);
\end{tikzpicture}
\caption{Illustrative network flow instance.\label{figure:network_flow}}
\end{figure}
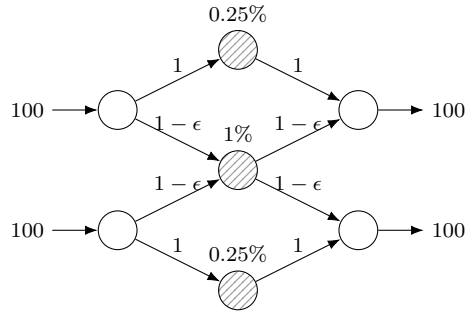

\begin{example}\label{ex:network}
    Consider the network shown in Figure~\ref{figure:network_flow}. Let $A$ denote the set of arcs. The goal is to decide arc capacities $\bm{x} \in \mathbb{R}^{\abs{A}}_{+}$ so as to route flow originating from the left layer of nodes to satisfy demand in the right layer of nodes. The per-unit cost $\bm{c} \in \mathbb{R}^{\abs{A}}$ of installing capacity on each arc is denoted above the arc, and the first-stage deterministic cost is given by $\bm{c}^\top \bm{x}$. The middle layer of nodes can fail independently of each other, and the numbers above the nodes denote their failure probabilities.
    If node~$i$ fails (indicated by ${\xi}_i = 1$), then all arcs incident to that node become unusable, and any resulting supply shortfall is penalized at a cost of $1,000$ per unit.
    For a given realization $\bm{\xi} \in \{0, 1\}^3$ of the node failures, the loss function $\mathcal{Q}(\bm{x}, \bm{\xi})$ is simply the total penalty cost under that realization and can be modeled as the optimal objective value of a linear program.

We can show that the optimal solution assigns a capacity of 100 units to every arc \re{(see Appendix \ref{sec:examaple_1})}.
In practice, however, the true failure probabilities are unknown.
We therefore estimate the distribution using a sample average approximation.
Figure~\ref{figure:network_flow_saa} shows the performance of the solutions obtained by replacing the true distribution $\mathbb{P}$ with the empirical distribution $\hat{\mathbb{P}}_N$ resulting from different sample sizes $N$.
In particular, the figure shows the difference between the total expected costs (computed under the true distribution) of the sample average solution and the true optimal solution.
We observe that despite the small dimensionality $(M = 3)$, $N \geq 1000$ samples are required for the sample average solution to estimate the true solution to an accuracy of 5\%.

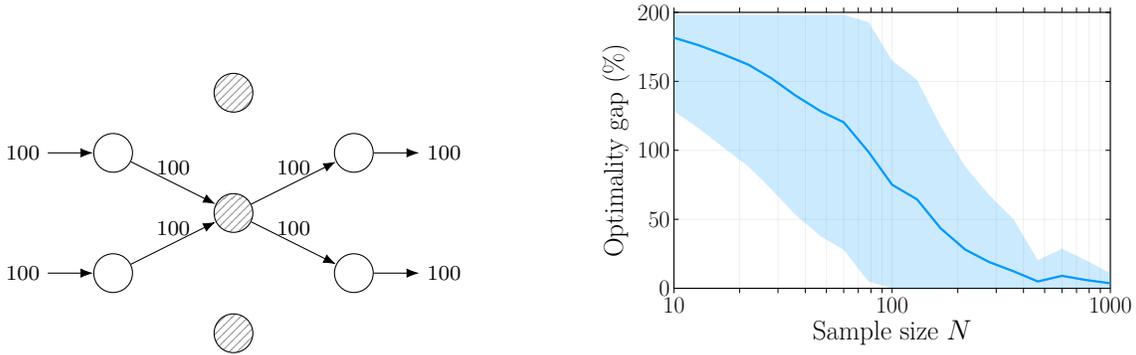
\begin{figure}[!htb]
    \begin{subfigure}[t]{0.45\textwidth}
        \centering
        \begin{tikzpicture}[scale=0.8,%
                            every label/.append style={rectangle, font=\scriptsize},
                            ed/.style = {-Latex},%
                            cr/.style = {circle, draw, minimum size = 0.5},%
                            crX/.style = {cr,pattern={north east lines},pattern color=gray!80}]
            \scriptsize
            \node[cr]  (1) at (0,  1) {\phantom{1}};
            \node[cr]  (2) at (0, -1) {\phantom{2}};
            \node[crX] (3) at (2,  2) {\phantom{3}};
            \node[crX] (4) at (2,  0) {\phantom{4}};
            \node[crX] (5) at (2, -2) {\phantom{5}};
            \node[cr]  (6) at (4,  1) {\phantom{6}};
            \node[cr]  (7) at (4, -1) {\phantom{7}};
            \node (o1) at (-1.5,  1) {100};
            \node (o2) at (-1.5, -1) {100};
            \node (d6) at ( 5.5,  1) {100};
            \node (d7) at ( 5.5, -1) {100};
            \draw[ed] (o1) -- (1);
            \draw[ed] (o2) -- (2);
            \draw[ed] (1) --node[above] {$100$} (4);
            \draw[ed] (2) --node[above] {$100$} (4);
            \draw[ed] (4) --node[above] {$100$} (6);
            \draw[ed] (4) --node[above] {$100$} (7);
            \draw[ed] (6) -- (d6);
            \draw[ed] (7) -- (d7);
        \end{tikzpicture}
        \caption{\footnotesize The optimal arc capacities determined by the sample average approximation when the empirical distribution $\hat{\mathbb{P}}_N$ puts all weight on the realization $\bm{\xi} = \bm{0}$, where no nodes fail \re{(see Appendix \ref{sec:examaple_1})}.}
        \label{figure:network_flow_saa:a}
    \end{subfigure}%
    ~
    \begin{subfigure}[t]{0.55\textwidth}
        \centering
        \includegraphics[width=0.8\linewidth]{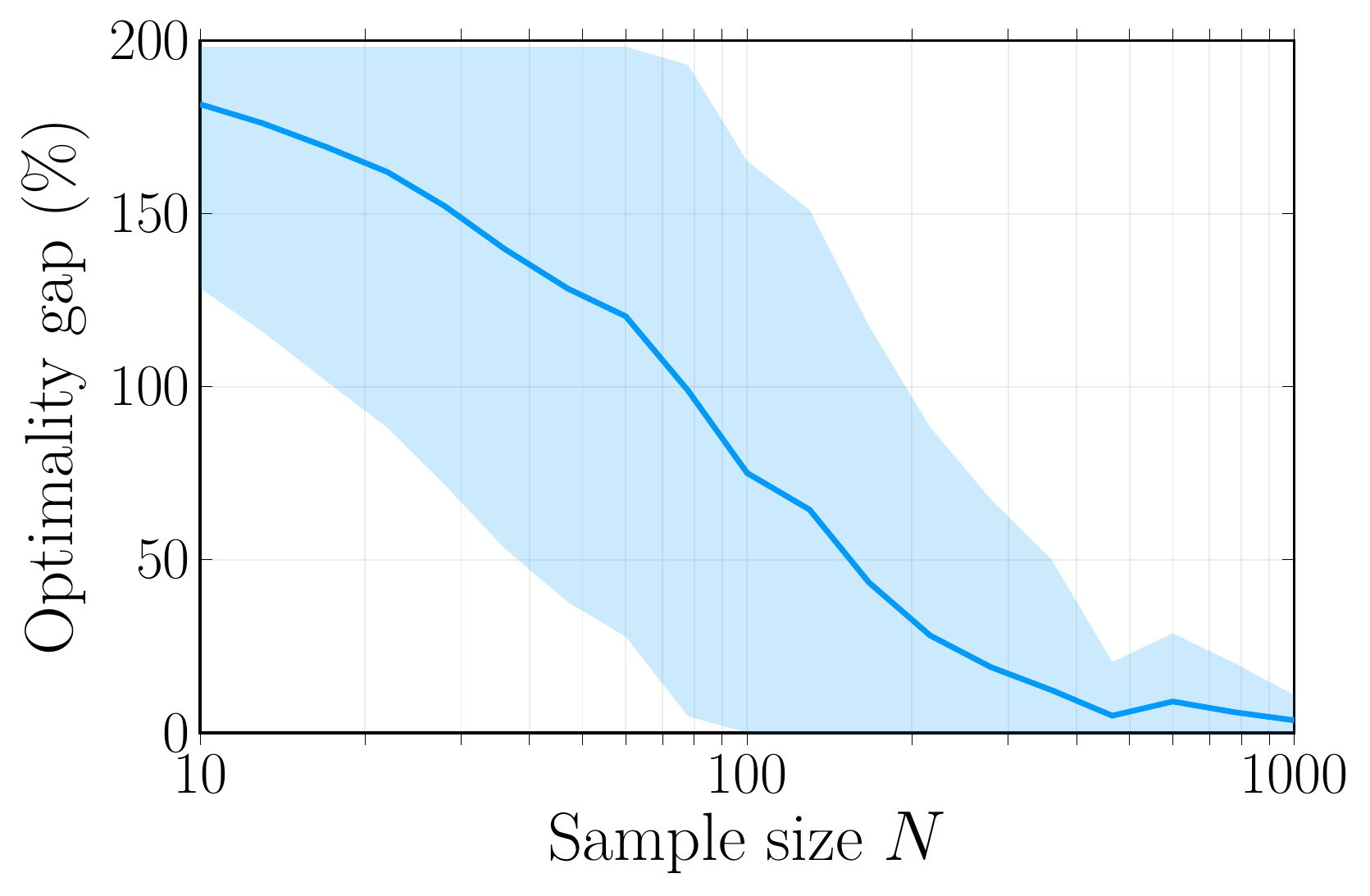}
        \caption{\footnotesize The optimality gap of the sample average approximation with respect to true optimal objective for increasing sample size. The mean (solid line) and standard deviation (shaded) are estimated by using 1,000 statistically independent sets of samples.}
        \label{figure:network_flow_saa:b}
    \end{subfigure}
    \caption{Performance of the sample average approximation on the network flow instance from Figure~\ref{figure:network_flow}.\label{figure:network_flow_saa}}
\end{figure}
\end{example}

\section{Distributionally Robust Approach for Discrete Rare Events}

The high dimensionality and rare occurrence of failure states render accurate estimation of the underlying distribution difficult.
To remedy this situation, we adopt a \emph{distributionally robust} approach, and construct an \emph{ambiguity set} $\mathcal{P}$ of possible distributions that are consistent with the observed data. We then minimize the \emph{worst-case} expected costs over all distributions in the ambiguity set.
Specifically, we consider distributionally robust two-stage conic optimization problems, of the form
\begin{equation}\label{eq:two_stage_dro}
\mathop{\text{minimize}}_{\mb{x} \in \mathcal{X}} \;
c(\mb{x})
+
\sup_{\mathbb{P} \in \mathcal{P}}
\mathbb{E}_{\mathbb{P}} \left[
\mathcal{Q}(\mb{x}, \tilde{\mb{\xi}})
\right].
\end{equation}
The ambiguity set $\mathcal{P}$ must be chosen such that it contains the true distribution with high confidence or, at the very least, distributions that assign nonzero probability to the rare events.
We focus on the Wasserstein ambiguity set, defined as the set of distributions that are close to the empirical distribution $\hat{\mathbb{P}}_N$ with respect to the Wasserstein distance $d_W$:
\begin{equation}\label{eq:ambiguity_set}
\mathcal{P} =
\left\{
\mathbb{Q} \in \mathcal{M}(\Xi) : d_W (\mathbb{Q}, \hat{\mathbb{P}}_N) \leq \varepsilon
\right\}.
\end{equation}
Here, $\mathcal{M}(\Xi)$ denotes the set of all distributions supported on $\Xi$, and $\varepsilon \geq 0$ is the radius of the Wasserstein ball.
Given any underlying metric $d(\cdot, \cdot)$ on the support set $\Xi$, the Wasserstein distance $d_W(\mathbb{P}, \mathbb{P}')$ between two distributions $\mathbb{P}, \mathbb{P}'$ can be defined as follows:
\begin{equation*}
d_W(\mathbb{P}, \mathbb{P}')
=
\min_{\Pi \in \mathcal{M}(\Xi \times \Xi)}
\left\{
\sum_{\mb{\xi} \in \Xi}\sum_{\mb{\xi}' \in \Xi} d(\mb{\xi}, \mb{\xi}') \Pi(\mb{\xi}, \mb{\xi}') :
\begin{array}{l}
\Pi \text{ is a coupling of $\mathbb{P}$ and $\mathbb{P}'$}
\end{array}
\right\},
\end{equation*}
We motivate the choice of Wasserstein ambiguity sets in Section~\ref{sec:advantages_wasserstein}.
The crucial role of the metric $d(\cdot, \cdot)$ and the radius $\varepsilon$ of the Wasserstein ball $\mathcal{P}$ is discussed in Section~\ref{sec:choice_of_metric_and_radius}

\subsection{Advantages of Wasserstein ambiguity sets for discrete rare events}\label{sec:advantages_wasserstein}
One can construct several ambiguity sets of distributions that are consistent with observed data.
Broadly, these are either sets of distributions that satisfy constraints on their moments~\citep{delage2010distributionally,wiesemann2014distributionally} or those defined as balls centered on a reference distribution with respect to a metric such as the $\phi$-divergence~\citep{ben2013robust,bayraksan2015data} or the Wasserstein distance~\citep{wozabal2012framework,esfahani2018data,gao2016distributionally,zhao2018data}.
Note that the high dimensionality and sparsity of the training data in the case of rare events prevent us from obtaining reliable estimates of moments, ruling out moment-based sets.
In contrast, metric-based sets have the attractive feature of tying directly with available data; indeed, the empirical distribution corresponding to the training data is a natural choice for the reference distribution.

Ambiguity sets based on $\phi$-divergence, especially Kullback-Liebler divergence, have certain shortcomings that are not shared by their Wasserstein counterparts.
First, the former can exclude the true distribution while including pathological distributions; in fact, they can fail to represent confidence sets for the unknown true distribution~\citep{gao2016distributionally}.
In contrast, the latter contains the true distribution with high confidence for an appropriate choice of $\varepsilon$, and hence the optimal value of the corresponding distributionally robust problem provides an upper confidence bound on the true out-of-sample cost~\citep{esfahani2018data}.
Second, Kullback-Liebler and other $\phi$-divergence based sets contain only those distributions that are absolutely continuous with respect to reference distribution; that is, these distributions can assign positive probability only to those realizations for which the empirical distribution also assigns positive mass~\citep{bayraksan2015data}.
This situation is undesirable for applications where uncertain events are rare, since the majority of possible uncertain states are unlikely to have been observed empirically.
In contrast, the Wasserstein ball of any positive radius $\varepsilon > 0$ and corresponding to any finite-valued metric $d(\cdot, \cdot)$ includes distributions that assign nonzero probability to any arbitrary realization.
Indeed, this property ensures that solutions of the distributionally robust problem~\eqref{eq:two_stage_dro} are robust to the occurrence of rare events.
We illustrate this via Example~\ref{ex:network}.

\setcounter{example}{0}
\begin{example}[continued]
Consider again the network in Figure~\ref{figure:network_flow}.
In addition to the sample average approximation, Figure~\ref{figure:network_flow_dro:b} now compares the performance of solutions computed by using the distributionally robust formulation~\eqref{eq:two_stage_dro} with a Wasserstein ambiguity set $\mathcal{P}$ of radius $\varepsilon = 10^{-3}$, centered around the empirical distribution $\hat{\mathbb{P}}_N$ resulting from different sample sizes $N$.
Here, we use the metric $d(\bm{\xi}, \bm{\xi}') = \norm{\bm{\xi} - \bm{\xi}'}_1$ induced by the 1-norm.
The figure shows the difference in total expected costs (computed under the true distribution) of the distributionally robust and sample average solutions with respect to the true optimal solution.
We observe that the distributionally robust solution strongly outperforms the sample average solution while being more stable to changes in the training data (i.e., its performance has a smaller variance for a fixed $N$).
\re{The performance of the solutions computed using a $\phi$-divergence ambiguity set are presented in the next subsection.}

\begin{figure}[!htb]
    \begin{subfigure}[t]{0.45\textwidth}
        \centering
        \begin{tikzpicture}[scale=0.8,%
                            every label/.append style={rectangle, font=\scriptsize},
                            ed/.style = {-Latex},%
                            cr/.style = {circle, draw, minimum size = 0.5},%
                            crX/.style = {cr,pattern={north east lines},pattern color=gray!80}]
        \scriptsize
        \node[cr]  (1) at (0,  1) {\phantom{1}};
        \node[cr]  (2) at (0, -1) {\phantom{2}};
        \node[crX] (3) at (2,  2) {\phantom{3}};
        \node[crX] (4) at (2,  0) {\phantom{4}};
        \node[crX] (5) at (2, -2) {\phantom{5}};
        \node[cr]  (6) at (4,  1) {\phantom{6}};
        \node[cr]  (7) at (4, -1) {\phantom{7}};
        \node (o1) at (-1.5,  1) {100};
        \node (o2) at (-1.5, -1) {100};
        \node (d6) at ( 5.5,  1) {100};
        \node (d7) at ( 5.5, -1) {100};
        \draw[ed] (o1) -- (1);
        \draw[ed] (o2) -- (2);
        \draw[ed] (1) --node[above] {$\frac23 100$} (3);
        \draw[ed] (1) --node[above] {$\frac13 100$} (4);
        \draw[ed] (2) --node[above] {$\frac13 100$} (4);
        \draw[ed] (2) --node[above] {$\frac23 100$} (5);
        \draw[ed] (3) --node[above] {$\frac23 100$} (6);
        \draw[ed] (4) --node[above] {$\frac13 100$} (6);
        \draw[ed] (4) --node[above] {$\frac13 100$} (7);
        \draw[ed] (5) --node[above] {$\frac23 100$} (7);
        \draw[ed] (6) -- (d6);
        \draw[ed] (7) -- (d7);
        \end{tikzpicture}
        \caption{\footnotesize The optimal arc capacities determined by the distributionally robust formulation when the empirical distribution $\hat{\mathbb{P}}_N$ puts all its weight on $\bm{\xi} = \bm{0}$, where none of the nodes fail \re{(see Appendix \ref{sec:examaple_1})}.}
        \label{figure:network_flow_dro:a}
    \end{subfigure}%
    ~
    \begin{subfigure}[t]{0.55\textwidth}
        \centering
        \includegraphics[width=0.8\linewidth]{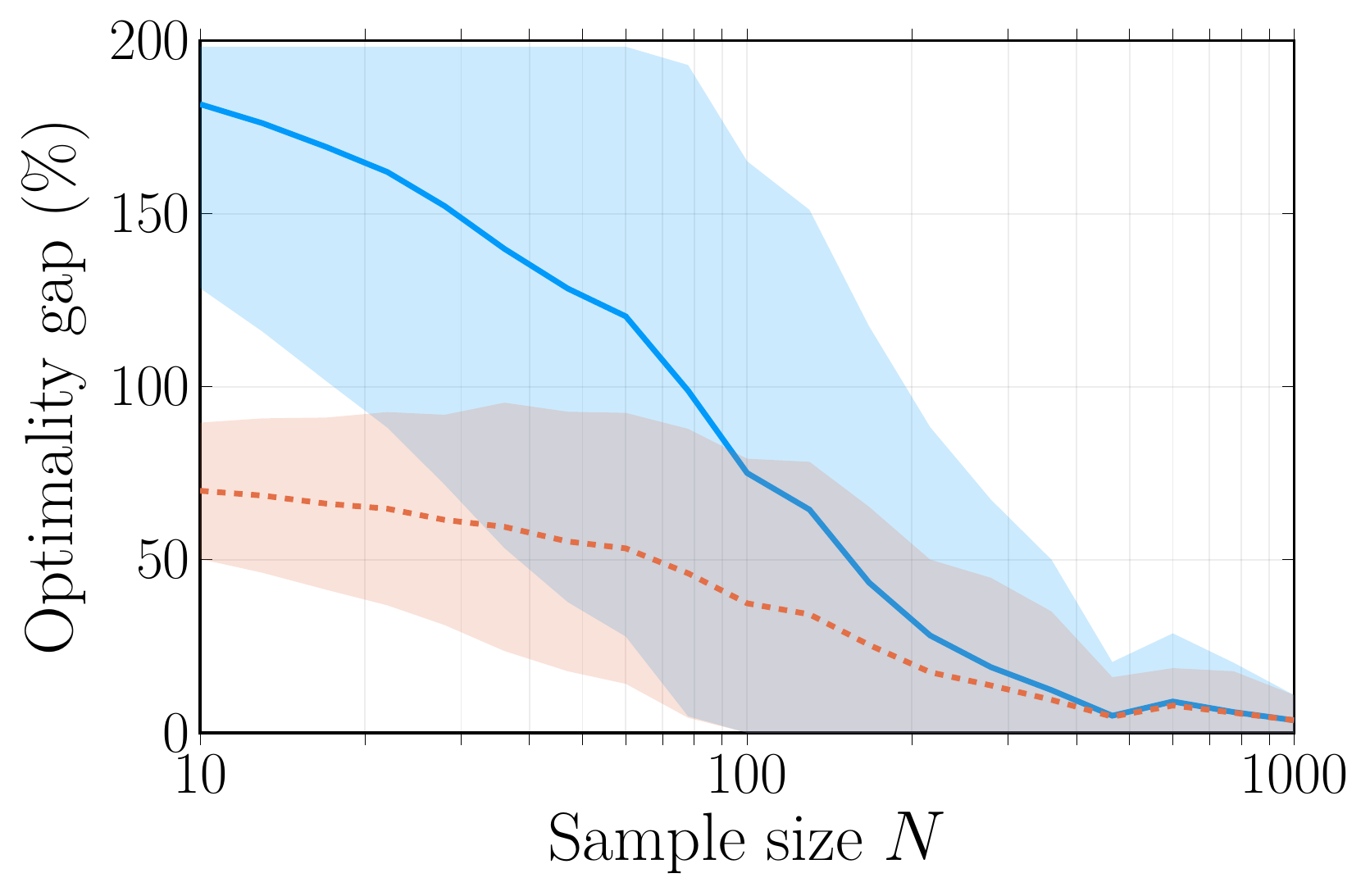}
        \caption{\footnotesize Comparison of the optimality gap of the distributionally robust formulation (dashed orange) with the sample average approximation (solid blue) for increasing sample size.}
        \label{figure:network_flow_dro:b}
    \end{subfigure}
    \caption{Performance of the distributionally robust formulation with a Wasserstein ambiguity set of radius $\varepsilon = 10^{-3}$ on the network flow instance from Figure~\ref{figure:network_flow}.\label{figure:network_flow_dro}}
\end{figure}
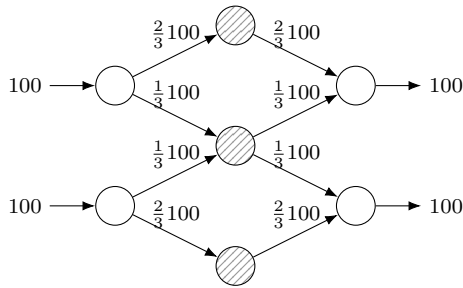
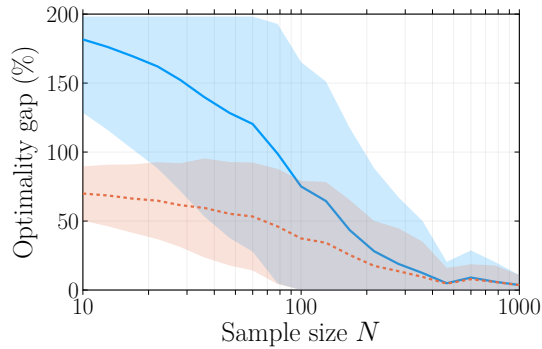
\end{example}

\subsection{Choice of the underlying metric and radius of the Wasserstein ball}\label{sec:choice_of_metric_and_radius}
The underlying metric $d(\cdot, \cdot)$ should ideally have the following properties: \textit{(i)} if $d(\bm{\xi}', \bm{\xi}'')$ is small, then the probabilities of occurrence of the two realizations $\bm{\xi}', \bm{\xi}'' \in \Xi$ should be similar; and \textit{(ii)} if $\bm{\xi}''$ is rarer than $\bm{\xi}'$, then for some fixed (say nominal) realization $\bm{\xi} \in \Xi$ (e.g., in a network, this could be the realization where none of the elements fail),  we must have $d(\bm{\xi}, \bm{\xi}') \leq d(\bm{\xi}, \bm{\xi}'')$.
These properties can be satisfied whenever $d(\bm{\xi}, \bm{\xi'}) = \norm{\bm{\xi} - \bm{\xi'}}$ is induced by a norm on $\mathbb{R}^{M}$ and by using an appropriate bit representation of the sample space.
Throughout the paper, we will therefore assume that the metric $d(\cdot, \cdot)$ is induced by an arbitrary norm, although our results also apply when $d(\cdot, \cdot)$ is any mixed-integer conic-programming-representable metric.

A noteworthy example of a metric that \textit{does not} satisfy the above requirement is the \emph{discrete metric}, defined as $d(\bm{\xi}', \bm{\xi}'') = 1$ whenever $\bm{\xi}' \neq \bm{\xi}''$ and $0$ otherwise.
In this case, the Wasserstein distance is equivalent to the \emph{total variation distance}, and the distributions that solve the inner supremum in~\eqref{eq:two_stage_dro} can assign positive probability only to the training samples and the worst-case realization~\citep{jiang2018,rahimian2019identifying}.
In a network with rare failures, this means that only past observed realizations and the realization corresponding to complete network failure are taken into account, resulting in poor out-of-sample performance.
\setcounter{example}{0}
\begin{example}[continued]
    Suppose that we define the metric $d(\cdot, \cdot)$ to be the discrete metric. Then, the performance of solutions of the distributionally robust problem~\eqref{eq:two_stage_dro} with a total variation ambiguity set $\mathcal{P}$ of \emph{any} radius $\varepsilon \geq 0$ is equal to or worse than that of the sample average solution.
\end{example}

For a given choice of the underlying metric, the radius $\varepsilon$ of the ambiguity set allows us to control the level of conservatism of solutions of~\eqref{eq:two_stage_dro}.
Specifically, given a confidence level $\beta \in (0, 1)$, one can choose the radius as a function of $\beta$ and the number of observations $N$ such that the true distribution $\mathbb{P}$ is contained in the ambiguity set with high probability:
\begin{equation}\label{eq:epsilon_choice}
\mathbb{P}^N\left[d_W(\mathbb{P}, \hat{\mathbb{P}}_N) \leq \varepsilon_N(\beta) \right] \geq 1 - \beta.
\end{equation}
Here, $\mathbb{P}^N$ is the product distribution that governs the observed data $\{\hat{\bm{\xi}}^{(1)}, \ldots, \hat{\bm{\xi}}^{(N)}\}$.
It was shown in~\citet{esfahani2018data} that~\eqref{eq:epsilon_choice} holds if we select $\varepsilon_N(\beta) = c_0 \left({N^{-1}\log \beta^{-1}}\right)^{1/M}$, where $c_0$ is a problem-dependent constant.
\re{
    Since this choice can lead to unnecessarily large values for the radius, the authors suggest solving the two-stage problem~\eqref{eq:two_stage_dro} for several \textit{a priori} fixed choices of $\varepsilon$ and then using cross-validation (e.g., $k$-fold cross-validation) to pick a good value.
    In the following, we show that Sanov's theorem can be used to obtain stronger finite sample guarantees; that is, less conservative values for $\varepsilon_N(\beta)$, by exploiting the finiteness of the support $\Xi$.
    \begin{theorem}[Finite sample guarantee]\label{thm:finite_sample_guarantee}
        For every fixed sample size $N > 0$, and confidence level $\beta \in (0, 1)$, the probabilistic guarantee~\eqref{eq:epsilon_choice} holds whenever
        \begin{equation}\label{eq:epsilon_finite_sample}
        \varepsilon_N(\beta) \geq D \sqrt{(2N)^{-1}\left( \abs{\Xi} \log(N+1) + \log \beta^{-1} \right)},
        \end{equation}
        where $D \coloneqq \max_{\bm{\xi}, \bm{\xi}' \in \Xi} d(\bm{\xi}, \bm{\xi}')$ is the diameter of $\Xi$ with respect to the metric $d$.
    \end{theorem}
    \begin{proof}
        See Appendix~\ref{sec:proofs}.
    \end{proof}
    The right-hand side of~\eqref{eq:epsilon_finite_sample} indicates that for a fixed choice of the support $\Xi$ and confidence level $\beta$, $\varepsilon_N(\beta)$ is roughly proportional to $\sqrt{N^{-1}\log (N+1)}$.
    For such choices, the optimal value of the corresponding distributionally robust two-stage problem~\eqref{eq:two_stage_dro} can be expected to provide an upper bound on the true (unknown) out-of-sample cost.
    We empirically verify this upper bound in Section~\ref{sec:computational_experiments}, when we vary $\varepsilon = \nu \sqrt{N^{-1}\log(N+1)}$ as a function of a scalar $\nu$.
}
Theorem~\ref{thm:finite_sample_guarantee} also indicates that $\varepsilon_N(\beta) \to 0$ as the sample size becomes large, $N \to \infty$, and that $\varepsilon_N(\beta) \to \infty$ if we are overly conservative, $\beta \to 0$.
This observation also elucidates that the distributionally robust formulation~\eqref{eq:two_stage_dro} generalizes both classical stochastic and robust optimization.
\begin{remark}[Reduction to two-stage stochastic and robust optimization]\label{rem:reduction_to_ro_and_sp}
    The distributionally robust two-stage problem~\eqref{eq:two_stage_dro} reduces to the classical sample average approximation whenever the radius of the ambiguity set $\varepsilon = 0$, since $\mathcal{P} = \{\hat{\mathbb{P}}_N\}$ reduces to a singleton in this case.
    Similarly, it reduces to a classical two-stage robust optimization problem whenever $\varepsilon \geq \max_{\bm{\xi}, \bm{\xi}' \in \Xi} d(\bm{\xi}, \bm{\xi}')$ since $\Xi$ is compact and $\mathcal{P}$ contains all Dirac distributions, $\delta_{\bm{\xi}}$, $\bm{\xi} \in \Xi$, in this case. Therefore, the worst-case expectation in~\eqref{eq:two_stage_dro} reduces to $\max_{\bm{\xi} \in \Xi} \mathcal{Q}(\bm{x}, \bm{\xi})$. %
\end{remark}

\subsection{Contributions}
This paper addresses the relatively unexplored topic of rare high-impact uncertainties through the lens of data-driven distributionally robust optimization.
Existing methods for addressing rare high-impact uncertainties~\citep{barrera2016chance,budhiraja2019minimization} are few, and they are all based on variants of Monte Carlo methods (e.g., importance sampling), which require the existence of a probability distribution that can be sampled to generate additional observations.

In contrast, our techniques fall within the scope of distributionally robust optimization~\citep{shapiro2018tutorial,rahimian2019distributionally}.
A surge in the popularity of Wasserstein ambiguity sets has coccurred in this area because of recent results~\citep{gao2016distributionally,esfahani2018data,bertsimas2018robustSAA,blanchet2019robust,kuhn2019wasserstein} that have established \textit{(i)} strong finite sample and asymptotic guarantees of their formulations, \textit{(ii)} connections with function regularization in machine learning, and \textit{(iii)} tractable reformulations for various classes of loss functions $\mathcal{Q}(\bm{x}, \bm{\xi})$ and metric spaces $(\Xi, d)$.
Most of the tractability results are for one-stage problems, however, with piecewise linear loss functions and continuous random parameters.
Existing tractability results for two-stage problems are limited to the case where $\mathcal{Q}(\bm{x}, \bm{\xi})$ is the optimal value of a linear program, $\Xi$ is a polyhedron, and $d$ is induced by the $1$-norm, see~\citet{esfahani2018data,hanasusanto2018conic}.
\re{Sufficient conditions for zero-one supports $\Xi$ and type-$\infty$ Wasserstein ambiguity sets have been established in~\citet{xie2019}.}
In the absence of sufficient conditions that ensure tractability, one resorts either to iterative global optimization methods (e.g., see~\citet{wozabal2012framework,zhao2014data,luo2017decomposition}) or tractable approximations.
The latter commonly include discretization schemes (of which sample average approximation is a special case)~\citep{luedtke2008sample,calafiore2005uncertain} and decision rule methods~\citep{ben2004adjustable,Georghiou2018,bertsimas2019SampleRobustOptimization}.

In this paper, we extend the state of the art in data-driven optimization by studying two-stage conic programs %
with a particular focus on high-dimensional zero-one uncertainties.
This is crucial because existing reformulations and algorithms for distributionally robust optimization with finitely supported distributions~\citep{postek2016computationally,ben2013robust,bertsimas2018robustSAA,bansal2018decomposition} scale with the size of the support set $\abs{\Xi}$, which can grow exponentially large in such cases.
We circumvent this exponential growth by utilizing tractable conservative approximations inspired by lift-and-project convexification techniques in global optimization~\citep{lovasz1991cones,sherali1990hierarchy,lasserre2001global}.

Closest in spirit to our work are the papers of~\citet{hanasusanto2018conic,Xu2018,ardestani2017linearized} who consider the case where the second-stage loss $\mathcal{Q}(\bm{x}, \bm{\xi})$ is the optimal value of a linear program with uncertain right-hand sides and the support set $\Xi$ is a polytope.
In this setting, \citet{hanasusanto2018conic,Xu2018} reformulate~\eqref{eq:two_stage_dro} as a copositive cone program and then approximate this using semidefinite programming, whereas \citet{ardestani2017linearized} provide approximations by leveraging reformulation-linearization techniques from bilinear programming.
\re{Although some extensions of these approaches to the case of zero-one support sets $\Xi$ have been made~\citep{mittal2020robust,jiang2019data}, problems where $\mathcal{Q}(\bm{x}, \bm{\xi})$ is the optimal value of a conic program have not been addressed.
This is partly because such extensions lead to so-called \emph{generalized copsitive programs} or \emph{set-semidefinite programs} (e.g., see \citet{burer2012representing,eichfelder2008set}), and relatively little is known about their tractable approximations.
In contrast, generalizations of lift-and-project techniques to mixed zero-one conic problems are fairly well known (e.g., see \cite{Stubbs1999ABM,Cezik2005}), and we exploit these to derive tractable approximations for distributionally robust optimization.
The relationship between convexification hierarchies based on copositive programming and lift-and-project techniques (for specific problem classes) has been explored in~\cite{povh2009copositive,burer2012non}.
}

We highlight the following main contributions:
\begin{enumerate}
    \item %
    By exploiting ideas from penalty methods and bilinear programming, we develop reformulations of the Wasserstein distributionally robust two-stage problem~\eqref{eq:two_stage_dro} that reduce its solution to optimization problems over the convex hulls of mixed-integer conic representable sets. Extensions to conditional value-at-risk are also presented.
    
    \item We provide sufficient conditions for our convex hull reformulations and hence the distributionally robust two-stage problem~\eqref{eq:two_stage_dro} to be tractable. We also show that they are generically NP-hard, however, even if there are no first-stage decisions and the second-stage loss function is the optimal value of a two-dimensional linear program with uncertain objective coefficients.
    
    \item By using lift-and-project hierarchies to approximate the convex hull of the mixed-integer conic representable sets, we derive tractable conservative approximations of the distributionally robust two-stage problem~\eqref{eq:two_stage_dro}, and we provide practical guidelines to compute them efficiently. The approximations \re{are tractable irrespective of the aforementioned sufficient conditions, and they} become exact as the Wasserstein radius $\varepsilon$ shrinks to zero.
    
    \item We demonstrate the practical viability of our method and its out-of-sample performance on challenging nonlinear optimal power flow \re{and multi-commodity network design problems} that are affected by rare network contingencies, and we study its behavior as a function of the rarity and impact of these contingencies, \re{illustrating improvements over classical sample average and two-stage robust optimization formulations.}
\end{enumerate}

The rest of the paper is organized as follows.
Section~\ref{sec:reformulation} derives the mixed-integer conic programming representation of interest,
Section~\ref{sec:lift_and_project} derives their lift-and-project approximations, and Section~\ref{sec:computational_experiments} reports numerical results. %
For ease of exposition, the complexity analysis as well as proofs of all assertions is deferred to the appendix.

\paragraph{Notation.}
Vectors and matrices are printed in bold lowercase and bold uppercase letters, respectively, while scalars are printed in regular font. 
The set of non-negative integers and reals is denoted by $\mathbb{Z}_{+}$ and $\mathbb{R}_{+}$, respectively.
For any positive integer $N$, we define $[N]$ as the index set $\left\{1,\ldots, N\right\}$. 
We use $\one_k$ to denote the $k^\text{th}$ unit basis vector, $\one$ to denote the vector of ones, $\eye$ to denote the identity matrix, and $\bm{0}$ to denote the vector or matrix of zeros,
respectively; their dimensions will be clear from the context.
For a matrix $\bm{A}$, we use $\vecmat{\bm{A}}$ to denote the vector obtained by stacking the columns of $\bm{A}$ in order.
The inner product between two matrices $\bm{A}, \bm{B} \in \mathbb{R}^{m \times n}$ is denoted by $\inner{\bm{A}}{\bm{B}} \coloneqq \sum_{i\in[m]}\sum_{j \in [n]}A_{ij} B_{ij}$. %
We use $\mathcal{C}^{n} = \left\{(\bm{x}, t) \in \mathbb{R}^{n-1} \times \mathbb{R}: \norm{\bm{x}} \leq t \right\}$ to denote the norm cone associated with the norm $\norm{\cdot}$.
For a logical expression $\mathcal{E}$, we define $\mathbb{I}[\mathcal{E}]$ as the indicator function which takes a value of $1$ if $\mathcal{E}$ is true and $0$ otherwise.
\re{Throughout the paper, we refer to an optimization problem as \emph{tractable} if it can be solved in polynomial time in the size of its input data, and \emph{intractable} if it is NP-hard.}

%% file: reformulation.tex
\section{Mixed-Integer Conic Representations}\label{sec:reformulation}
Throughout the paper, we assume that the distributionally robust two-stage problem~\eqref{eq:two_stage_dro} satisfies the assumptions of \emph{complete} and \emph{sufficiently expensive recourse}.
\begin{itemize}
    \item[\textbf{(A1)}] For every realization $\mb{\xi} \in \Xi$, there exists $\bm{y}^{+} \in \mathrm{int}(\mathcal{Y})$ such that $\bm{W}(\mb{\xi})\bm{y}^{+} > 0$.
    \item[\textbf{(A2)}] For every first-stage decision $\bm{x} \in \mathcal{X}$ and every realization $\mb{\xi} \in \Xi$, the second-stage loss function $\mathcal{Q}(\bm{x}, \bm{\xi})$ is bounded.
\end{itemize}
A natural way to ensure these assumptions is to add slack variables in the formulation of the second-stage problem $\mathcal{Q}(\bm{x}, \bm{\xi})$ and penalize them in the objective function.
Whenever the assumptions are satisfied, they imply that \emph{(i)} $\mathcal{Q}(\mb{x}, \mb{\xi})$ is always strictly feasible and bounded, \emph{(ii)} the dual of $\mathcal{Q}(\mb{x}, \mb{\xi})$, given in the following, is always feasible, and \emph{(iii)} strong conic duality holds between the second-stage problem and its dual, $\mathcal{Q}(\mb{x}, \mb{\xi}) = \mathcal{Q}_d(\mb{x}, \mb{\xi})$, where
 \begin{equation}\label{eq:Q_dual}
\mathcal{Q}_d(\mb{x}, \mb{\xi}) :=
\sup_{\bm{\lambda} \in \mathbb{R}^L_{+}}
\left\{
\left[\bm{T}(\mb{x})\mb{\xi} + \bm{h}(\mb{x}) \right]^\top \bm{\lambda} : 
\bm{q}(\mb{\xi}) - \bm{W}(\mb{\xi})^\top \bm{\lambda} \in \mathcal{Y}^*
\right\}.
\end{equation}
Here, $\mathcal{Y}^*$ denotes the dual cone of $\mathcal{Y}$.
We assume that the uncertain vectors and matrices in \eqref{eq:second_stage_value_function} are affine and can be represented as
$
\bm{q}(\mb{\xi}) = \bm{q}_0 + \bm{Q} \mb{\xi}
$
and
$
\bm{W}(\mb{\xi}) = \bm{W}_0 + \sum_{j \in [M]} \xi_j \bm{W}_j,
$
where $\bm{Q} \in \mathbb{R}^{N_2 \times M}$, %
and for each $j \in \{0, 1, \ldots, M\}$, we have $\bm{W}_j \in \mathbb{R}^{L \times N_2}$. %

A consequence of the above assumptions is the following lemma, which states that computing the worst-case expectation in the two-stage problem~\eqref{eq:two_stage_dro} is equivalent to averaging $N$ worst-case values of the loss function $\mathcal{Q}(\bm{x}, \bm{\xi})$ over $\bm{\xi} \in \Xi$, each regularized by one of the training samples.
We omit the proof since it follows directly from the compactness of $\Xi$ and the definition of the Wasserstein ambiguity set $\mathcal{P}$ in~\eqref{eq:ambiguity_set}; see~\citet{gao2016distributionally,blanchet2019quantifying} for proofs in much more general settings.
\begin{lemma}\label{prop:GK}
    The distributionally robust two-stage problem~\eqref{eq:two_stage_dro} admits the following reformulation:
    \begin{equation}\label{eq:GK}
    \displaystyle\mathop{\text{minimize}}_{\mb{x} \in \mathcal{X}, \alpha \geq 0} \; \displaystyle c(\mb{x}) + \alpha \varepsilon + \frac{1}{N}\sum_{i = 1}^N \max_{\mb{\xi} \in \Xi} \left\{ \mathcal{Q}(\mb{x}, \mb{\xi}) - \alpha d(\mb{\xi}, \hat{\mb{\xi}}^{(i)}) \right\}.
    \end{equation}
\end{lemma}

The remainder of this section establishes that the inner maximization in~\eqref{eq:GK} is equivalent to optimizing a linear function over the convex hull of the feasible region of a \emph{mixed-integer conic program} (MICP).
The following theorem is key to establishing this result.
\begin{theorem}[Convex hull reformulation]\label{thm:convex_reformulation}
    The distributionally robust two-stage problem~\eqref{eq:two_stage_dro} admits the following convex hull reformulation:
    \begin{equation}\label{eq:two_stage_dro_reform}
    \displaystyle\mathop{\text{minimize}}_{\bm{x} \in \mathcal{X}, \alpha \geq 0} \;\; \displaystyle c(\mb{x}) + \alpha \varepsilon + \frac{1}{N} \sum_{i = 1}^N Z_i(\mb{x}, \alpha),
    \end{equation}
    where, for each $i \in [N]$, we define the function $Z_i :\mathcal{X} \times \mathbb{R}_{+} \mapsto \mathbb{R}$ and the set $\mathcal{Z}_i$ as follows:
    \begin{subequations}
    \begin{align}
    &\displaystyle Z_i(\mb{x}, \alpha) = 
    \mathop{\text{maximize}}_{(\bm{\xi}, \bm{\lambda}, \bm{\Lambda}, \tau) \in \mathop{\mathrm{cl\,conv}}(\mathcal{Z}_i)}  \displaystyle \left\{\inner{\bm{T}(\bm{x})}{\bm{\Lambda}} + \bm{h}(\bm{x})^\top \bm{\lambda}
    - \alpha \tau \right\} \label{eq:Z_function_definition} \\
    &\displaystyle \mathcal{Z}_i = \left\{
    (\bm{\xi}, \bm{\lambda}, \bm{\Lambda}, \tau) \in \Xi \times \mathbb{R}^L_{+} \times \mathbb{R}^{L \times M} \times \mathbb{R}_{+}: 
    \begin{array}{l}
    \displaystyle \bm{\Lambda} = \bm{\lambda} \mb{\xi}^\top, \; (\mb{\xi} - \hat{\mb{\xi}}^{(i)}, \tau) \in \mathcal{C}^{M+1} \\
    \displaystyle \bm{q}_0 + \bm{Q}\mb{\xi} - \bm{W}_0^\top \bm{\lambda} - \sum_{j \in [M]} \bm{W}_j^\top \bm{\Lambda} \one_{j}  \in \mathcal{Y}^*
    \end{array}
    \right\}. \label{eq:Z_set_definition}
    \end{align}
    \end{subequations}
\end{theorem}

\begin{proof}
    See Appendix~\ref{sec:proofs}.
\end{proof}

The inner optimization problem~\eqref{eq:Z_function_definition} is over the closed convex hull of the set $\mathcal{Z}_i$, which couples the binary uncertain parameters $\bm{\xi}$ with the continuous dual variables $\bm{\lambda}$ via the bilinear equation $\bm{\Lambda} = \bm{\lambda} \bm{\xi}^\top$.
This set is, therefore, not the feasible region of an MICP.
We propose two approaches to ensure MICP representability.
The first is to linearize the bilinear equation $\bm{\Lambda} = \bm{\lambda} \bm{\xi}^\top$ using McCormick inequalities, which requires \emph{a priori} upper bounds on the dual variables $\bm{\lambda}$; we briefly discuss this in Section~\ref{sec:linearization}.
The second is to reformulate the loss function $\mathcal{Q}(\bm{x}, \bm{\xi})$ using ideas from penalty methods %
that circumvents any bilinear terms;
this approach is the subject of Section~\ref{sec:penalty_based_formulation}.
We compare the two approaches and summarize their merits in Section~\ref{sec:micp_discussion}.
We note that our results also apply if the risk-neutral expectation in the objective function of the two-stage problem~\eqref{eq:two_stage_dro} is replaced with the conditional value-at-risk; for ease of exposition, we defer this analysis to the appendix.

\subsection{Linearized reformulation}\label{sec:linearization}
The decision variables $\bm{\lambda}$ of the dual problem $\mathcal{Q}_d(\bm{x}, \bm{\xi})$ must be necessarily bounded for any fixed value of $\bm{x} \in \mathcal{X}$ and $\bm{\xi} \in \Xi$.
Indeed, under assumptions (A1) and (A2), the value of $\mathcal{Q}_d(\bm{x}, \bm{\xi})$ is bounded for any fixed $\bm{x} \in \mathcal{X}$ and $\bm{\xi} \in \Xi$; and since $\mathcal{X}$ and $\Xi$ are compact sets, the variables $\bm{\lambda}$ must also be necessarily bounded. 
Suppose that $\bar{\bm{\lambda}} \in \mathbb{R}^L_{+}$ are \emph{a priori} known upper bounds (independent of $\bm{x}$ and $\bm{\xi}$) on these variables.
\re{Such bounds may be analytically known whenever we explicitly add slack variables to ensure feasibility of the second-stage problem or if the latter has some structure (e.g., see~\citet{gabrel2014robust}).
Whenever such bounds are known}, we can exactly linearize the bilinear equation $\bm{\Lambda} = \bm{\lambda} \bm{\xi}^\top$ using McCormick inequalities \re{since $\bm{\xi}$ is binary-valued (e.g., see~\cite{glover1975improved})}, and reformulate the set $\mathcal{Z}_i$ in~\eqref{eq:Z_set_definition} as the feasible region of an MICP:
\begin{equation}
\displaystyle \mathcal{Z}_i = \left\{
(\bm{\xi}, \bm{\lambda}, \bm{\Lambda}, \tau) \in \Xi \times \mathbb{R}^L_{+} \times \mathbb{R}^{L \times M}_{+} \times \mathbb{R}_{+}: 
\begin{array}{l}
\displaystyle (\mb{\xi} - \hat{\mb{\xi}}^{(i)}, \tau) \in \mathcal{C}^{M+1} \\
\displaystyle \bm{\Lambda} - \bm{\lambda} \one^\top + \bar{\bm{\lambda}} (\one - \bm{\xi})^\top \in \mathbb{R}^{L\times M}_{+} \\
\displaystyle \bm{\lambda} \one^\top - \bm{\Lambda} \in \mathbb{R}^{L\times M}_{+}, \; \bar{\bm{\lambda}} \bm{\xi}^\top - \bm{\Lambda} \in \mathbb{R}^{L\times M}_{+} \\
\displaystyle \bm{q}_0 + \bm{Q}\mb{\xi} - \bm{W}_0^\top \bm{\lambda} - \sum_{j \in [M]} \bm{W}_j^\top \bm{\Lambda}\one_{j}  \in \mathcal{Y}^*
\end{array}
\right\}. \label{eq:Z_set_definition_linearized}\tag{\ref*{eq:Z_set_definition}-$\ell$}
\end{equation}
The MICP representation~\eqref{eq:Z_set_definition_linearized} adds $O(ML)$ variables and constraints for each $\mathcal{Z}_i$, $i \in [N]$, \re{which can be prohibitively large.
Moreover, analytical upper bounds $\bar{\bm{\lambda}} \in \mathbb{R}^L_{+}$ on the dual variables, which are independent of (and valid for all) $\bm{x}$ and $\bm{\xi}$, may be unavailable.
The following section shows that we can ensure MICP representability without \emph{a priori} knowledge of these bounds}, and at the expense of adding far fewer variables and constraints.

\subsection{Penalty reformulation}\label{sec:penalty_based_formulation}
Our goal in this section will be to move the uncertainty to the objective function of the second-stage problem, $\mathcal{Q}(\bm{x}, \bm{\xi})$.
This approach is motivated by the following corollary to Theorem~\ref{thm:convex_reformulation} that shows that MICP-representability is guaranteed if only the objective function of $\mathcal{Q}(\bm{x}, \bm{\xi})$ is uncertain.
\begin{corollary}[Convex hull reformulation for objective uncertainty]\label{coro:two_stage_dro_reform_obj_uncertainty}
    Suppose that only the objective function of the second-stage problem $\mathcal{Q}(\bm{x}, \bm{\xi})$ is uncertain: $\bm{W}(\bm{\xi}) = \bm{W}_0$ and $\bm{T}(\bm{x}) = \bm{0}$.
    Then the distributionally robust two-stage problem~\eqref{eq:two_stage_dro} admits reformulation~\eqref{eq:two_stage_dro_reform} where, for each $i \in [N]$, we define the function $Z_i :\mathcal{X} \times \mathbb{R}_{+} \mapsto \mathbb{R}$ and the MICP-representable set $\mathcal{Z}_i$ as follows:
    \begin{subequations}
        \begin{align}
        &\displaystyle Z_i(\mb{x}, \alpha) = 
        \mathop{\text{maximize}}_{(\bm{\xi}, \bm{\lambda}, \tau) \in \mathop{\mathrm{cl\,conv}}(\mathcal{Z}_i)}  \displaystyle \left\{\bm{h}(\bm{x})^\top \bm{\lambda}
        - \alpha \tau \right\} \label{eq:Z_function_definition_obj_unc} \\ %
        &\displaystyle \mathcal{Z}_i = \left\{
        (\bm{\xi}, \bm{\lambda}, \tau) \in \Xi \times \mathbb{R}^L_{+} \times  \mathbb{R}_{+}: 
        \displaystyle (\mb{\xi} - \hat{\mb{\xi}}^{(i)}, \tau) \in \mathcal{C}^{M+1}, \;\;
        \displaystyle \bm{q}_0 + \bm{Q}\mb{\xi} - \bm{W}_0^\top \bm{\lambda} \in \mathcal{Y}^*
        \right\}. \label{eq:Z_set_definition_obj_unc} %
        \end{align}
    \end{subequations}
\end{corollary}

\re{For the remainder of the paper, we shall make the following additional assumption of \emph{fixed recourse} which will be key to achieving our goal.}
\begin{itemize}
    \item[\textbf{(A3)}] For every realization $\mb{\xi} \in \Xi$ and every first-stage decision $\bm{x} \in \mathcal{X}$, we have $\bm{W}(\bm{\xi}) = \bm{W}_0$ and $\bm{T}(\bm{x}) = \bm{T}_0$, respectively.
\end{itemize}
We note that since $\bm{\xi}$ is binary valued, this assumption is without loss of generality if the primal decision variables $\bm{x}$ and $\bm{y}$ are bounded.
Indeed, in such cases, we can exactly linearize any products of uncertain parameters and decisions in the second-stage problem $\mathcal{Q}(\bm{x}, \bm{\xi})$ using McCormick inequalities~\citep{glover1975improved}, to ensure that it satisfies the assumption of fixed recourse.
Under this additional assumption, the following theorem states that the second-stage problem $\mathcal{Q}(\bm{x}, \bm{\xi})$ with constraint uncertainty can be equivalently reformulated as one with objective uncertainty.
\begin{theorem}[Penalty reformulation of the loss function]\label{thm:two_stage_dro_obj}
    There exists a sufficiently large, yet \emph{finite}, penalty parameter $\rho > 0$ such that the second-stage loss function $\mathcal{Q}(\bm{x}, \bm{\xi})$ in~\eqref{eq:second_stage_value_function} is equivalent to
    \begin{equation}\label{eq:two_stage_dro_obj}
    \mathcal{Q}^{\rho}(\bm{x}, \bm{\xi}) :=
    \inf_{\bm{y} \in \mathcal{Y}, \bm{z} \in [0, 1]^M}
    \left\{
    \bm{q}(\bm{\xi})^\top \bm{y} + \rho \left((\bm{\mathrm{e}} - 2 \bm{\xi})^\top \bm{z} + \bm{\mathrm{e}}^\top \bm{\xi}\right): 
    \bm{W}_0 \bm{y} \geq \bm{T}_0 \bm{z} + \bm{h}(\bm{x})
    \right\}.
    \end{equation}
\end{theorem}

\begin{proof}
    See Appendix~\ref{sec:proofs}.
\end{proof}

In conjunction with Corollary~\ref{coro:two_stage_dro_reform_obj_uncertainty}, Theorem~\ref{thm:two_stage_dro_obj} implies that the distributionally robust two-stage problem~\eqref{eq:two_stage_dro} admits a convex hull reformulation of the form~\eqref{eq:two_stage_dro_reform}, 
where the function $Z_i :\mathcal{X} \times \mathbb{R}_{+} \mapsto \mathbb{R}$ and the MICP-representable set $\mathcal{Z}_i$ for each $i\in [N]$ are given as follows:
    \begin{align*}
    &\displaystyle Z_i(\mb{x}, \alpha) = 
    \mathop{\text{maximize}}_{(\bm{\xi}, \bm{\lambda}, \bm{\mu}, \tau) \in \mathop{\mathrm{cl\,conv}}(\mathcal{Z}_i)}  \displaystyle \left\{\bm{h}(\bm{x})^\top \bm{\lambda} + \rho(\bm{\mathrm{e}}^\top \bm{\xi}) - \bm{\mathrm{e}}^\top \bm{\mu}
    - \alpha \tau \right\} \label{eq:Z_function_definition_penalty} \tag{\ref*{eq:Z_function_definition}-$\rho$} \\
    &\displaystyle \mathcal{Z}_i = \left\{
    (\bm{\xi}, \bm{\lambda}, \bm{\mu}, \tau) \in \Xi \times \mathbb{R}^L_{+} \times \mathbb{R}^{M}_{+} \times \mathbb{R}_{+}: 
    \begin{array}{l}
    \displaystyle (\mb{\xi} - \hat{\mb{\xi}}^{(i)}, \tau) \in \mathcal{C}^{M+1} \\
    \displaystyle \bm{q}_0 + \bm{Q}\mb{\xi} - \bm{W}_0^\top \bm{\lambda} \in \mathcal{Y}^* \\
    \displaystyle \rho(\bm{\mathrm{e}} - 2\bm{\xi}) + \bm{T}_0^\top \bm{\lambda} + \bm{\mu} \in \mathbb{R}^M_{+}
    \end{array}
    \right\}. \label{eq:Z_set_definition_penalty} \tag{\ref*{eq:Z_set_definition}-$\rho$}
    \end{align*}
In contrast to the linearized reformulation~\eqref{eq:Z_set_definition_linearized}, the MICP representation~\eqref{eq:Z_set_definition_penalty} adds only $O(M+L)$ variables and constraints  for each $\mathcal{Z}_i$, $i \in [N]$.
\re{%
}

The reformulation~\eqref{eq:two_stage_dro_obj} is reminiscent of penalty methods in nonlinear programming.
However, in contrast to the latter, which suffer from numerical issues because the penalty parameter $\rho$ must be driven to $\infty$, a finite value for $\rho$ can be precomputed in our case.
This process only requires solving the classical robust optimization formulation over any support $\Xi^0 \supseteq \Xi$; for example, some choices are $\Xi^0 = \Xi$ or $\Xi^0 = \{0, 1\}^M$.
The procedure is as follows. We compute the classical robust solution $\bm{x}^r$ and a corresponding worst-case realization, as per~\eqref{eq:classical_ro} shown below, and then set $\rho^r$ to be an optimal Lagrange multiplier of the last inequality in~\eqref{eq:rho_computation}.%
\begin{subequations}
    \begin{gather}
    \bm{x}^r \in \mathop{\arg\min}_{\bm{x} \in \mathcal{X}} \;
    \left\{c(\bm{x})
    +
    \max_{\bm{\xi} \in \Xi^0}
    \mathcal{Q}(\bm{x}, \bm{\xi}) \right\}, \;
    \bm{\xi}^r \in \mathop{\arg\max}_{\bm{\xi} \in \Xi^0} \;
    \mathcal{Q}(\bm{x}^r, \bm{\xi}), \label{eq:classical_ro}\\
    \inf_{\bm{y} \in \mathcal{Y}, \bm{z} \in [0, 1]^M}
    \left\{
    \bm{q}(\bm{\xi}^r)^\top \bm{y}: 
    \bm{W}_0 \bm{y} \geq \bm{T}_0 \bm{z} + \bm{h}(\bm{x}^r), \; (\bm{\mathrm{e}} - 2 \bm{\xi}^r)^\top \bm{z} + \bm{\mathrm{e}}^\top \bm{\xi}^r \leq 0
    \right\}. \label{eq:rho_computation}
    \end{gather}
\end{subequations}
\re{%
}%
We next prove that our procedure is valid: the optimal Lagrange multiplier $\rho^r$ is indeed an \emph{exact value} for the penalty parameter.
\begin{theorem}[Exact penalty reformulation]\label{thm:rho_computation}
    The optimal value of the distributionally robust two-stage problem~\eqref{eq:two_stage_dro} remains unchanged if we replace the second-stage loss function $\mathcal{Q}(\bm{x}, \bm{\xi})$ with its penalty reformulation $\mathcal{Q}^{\rho^r}(\bm{x}, \bm{\xi})$ as defined in~\eqref{eq:two_stage_dro_obj} with penalty parameter $\rho^r$.
\end{theorem}

\begin{proof}
    See Appendix~\ref{sec:proofs}.
\end{proof}

The classical robust formulation~\eqref{eq:classical_ro} is tractable if we choose $\Xi^0 = \{0, 1\}^M$ and the loss function $\mathcal{Q}(\bm{x}, \bm{\xi})$ exhibits a down-monotone (or up-monotone) property with respect to the random parameters $\bm{\xi}$; that is, $\mathcal{Q}(\bm{x}, \bm{\xi}') \geq \mathcal{Q}(\bm{x}, \bm{\xi})$ whenever $\bm{\xi}' \geq \bm{\xi}$.
In particular, this is the case for network optimization problems where
removing network elements is never advantageous.
In such cases, the classical robust formulation reduces to a deterministic problem under the worst-case realization of the uncertain parameters.
Moreover, this worst-case realization is often independent of the optimal robust solution $\bm{x}^r$.
For example, suppose $\bm{\xi} \in \{0, 1\}^M$ is a random vector denoting which of $M$ arcs in a network have been ``disrupted'' and arc-flow variables $\bm{y}$ satisfy $0 \leq y_a \leq \xi_a \bar{y}_a$ (i.e., the flow on arc $a$ is zero whenever it is disrupted $\xi_a = 0$, and is bounded between $0$ and $\bar{y}_a$ otherwise).
The worst-case realization is to disrupt all arcs in the network, $\bm{\xi}^r = \bm{0}$, independent of the optimal robust solution $\bm{x}^r$.
Notably, this also implies that we can circumvent the computation of~\eqref{eq:classical_ro} when determining the value of $\rho^r$.
We generalize this observation in the following proposition.

\begin{proposition} \label{lem:disrupt_all_lines}
    Assume without loss of generality that $\mathcal{Y} \subseteq \mathbb{R}_{+}^{N_2}$.
    Suppose also that for all $j \in [M]$, we have either $\bm{T}_0\bm{\mathrm{e}}_j, \bm{Q}\bm{\mathrm{e}}_j \in \mathbb{R}^L_{+}$ or $-\bm{T}_0\bm{\mathrm{e}}_j, -\bm{Q}\bm{\mathrm{e}}_j \in \mathbb{R}^L_{+}$.
    Then, the classical robust optimization problem
    $\min_{\bm{x} \in \mathcal{X}} \;
    \Big\{c(\bm{x})
    +
    \max_{\bm{\xi} \in \{0, 1\}^M}
    \mathcal{Q}(\bm{x}, \bm{\xi}) \Big\}$
    reduces to a deterministic problem
    $\min_{\bm{x} \in \mathcal{X}} \;
    \left\{c(\bm{x})
    +
    \mathcal{Q}(\bm{x}, \bm{\xi}^r) \right\}$,
    where for each $j \in [M]$ we have $\xi^r_j = 1$ if $\bm{T}_0\bm{\mathrm{e}}_j, \bm{Q}\bm{\mathrm{e}}_j \in \mathbb{R}^L_{+}$ and $\xi^r_j = 0$ otherwise.
\end{proposition}

\begin{proof}
    See Appendix~\ref{sec:proofs}.
\end{proof}

\re{\subsubsection{Penalty reformulation for indicator constraints}
The penalty reformulation $\mathcal{Q}^{\rho}(\bm{x}, \bm{\xi})$ in~\eqref{eq:two_stage_dro_obj} introduces $M$ additional variables $\bm{z} \in [0, 1]^M$ in the second-stage loss function.
This can be avoided by exploiting a structure that is common in network optimization problems.
    In several such applications, an uncertain parameter $\xi_j \in \{0, 1\}$ may switch on and off a single constraint $f_j(\bm{y}) \geq 0$, leading to the following definition of the loss function:
        \begin{equation}\label{eq:second_stage_value_function_indicator}
            \mathcal{Q}_\mathrm{ind}(\bm{x}, \bm{\xi}) =
            \inf_{\bm{y} \in \mathcal{Y}}
            \left\{
            \bm{q}(\bm{\xi})^\top \bm{y} :
            \begin{array}{l}
                \displaystyle \bm{W}_0 \bm{y} \geq \bm{h}(\bm{x}) \\
                \displaystyle \xi_j = 1 \implies \left[f_j(\bm{y}) = 0\right], \; j \in [M] \\
                \displaystyle \xi_j = 0 \implies \left[f_j(\bm{y}) \geq 0\right], \; j \in [M]
            \end{array}
            \right\},
    \end{equation}
    where $f_j:\mathcal{Y}\mapsto \mathbb{R}$ is an affine function for each $j \in [M]$.
    For example, in a network, $\xi_j = 1$ may indicate that link~$j$ has failed and the corresponding flow variable $f_j(\bm{y}) = y_j$ must be set to $0$, whereas $\xi_j = 0$ may indicate that the flow variable $y_j$ can take nonzero values. Such constraints are generally written as $f_j(\bm{y}) \leq \bar{f}(1-\xi_j)$, where $\bar{f}$ is a big-M upper bound on $f_j(\bm{y})$.
    Obtaining tight estimates on $\bar{f}$ may be non-trivial and at the same time, an overestimation of $\bar{f}$ can lead to numerical issues.
    In such cases, we can avoid both \emph{(i)} estimating $\bar{f}$, and \emph{(ii)} introducing auxiliary variables $\bm{z}$ in the penalty reformulation~\eqref{eq:two_stage_dro_obj}, by simply retaining the constraint $f_j(\bm{y}) \geq 0$ and adding $+\rho \xi_j f(\bm{y}_j)$ to the objective, leading to a tighter penalty reformulation.
    \begin{corollary}[Penalty reformulation of the loss function with indicator constraints]\label{coro:two_stage_dro_obj_indicator}
        There exists a sufficiently large, yet \emph{finite}, penalty parameter $\rho > 0$ such that the second-stage loss function $\mathcal{Q}_\mathrm{ind}(\bm{x}, \bm{\xi})$ in~\eqref{eq:second_stage_value_function_indicator} is equivalent to
        \begin{equation}\label{eq:two_stage_dro_obj_indicator}
            \mathcal{Q}^\rho_\mathrm{ind}(\bm{x}, \bm{\xi}) =
            \inf_{\bm{y} \in \mathcal{Y}}
            \left\{
            \bm{q}(\bm{\xi})^\top \bm{y} + \rho \bm{\xi}^\top \bm{f}(\bm{y}) :
            \bm{W}_0 \bm{y} \geq \bm{h}(\bm{x}), \; \bm{f}(\bm{y}) \geq \bm{0}
            \right\}.
        \end{equation}
    \end{corollary}
    The penalty term $+\rho \xi_j f(\bm{y}_j)$ ensures that $f_j(\bm{y})$ is driven to $0$ whenever $\xi_j = 1$ for large values of $\rho$. Note that we no longer need to estimate the big-M upper bounds $\bar{f}$ since the constraint $f_j(\bm{y}) \leq \bar{f}(1 - \xi_j)$ is not required anymore.
    
    Our previous results continue to be valid as long as we replace each occurrence of the penalty term $(\bm{\mathrm{e}} - 2 \bm{\xi})^\top \bm{z} + \bm{\mathrm{e}}^\top \bm{\xi}$ that multiples $\rho$ in the objective function of \eqref{eq:two_stage_dro_obj} with this modification.
    For example, suppose that $\bm{f}(\bm{y}) = \bm{f}_0 + \bm{F}\bm{y}$.
    Then, Corollaries~\ref{coro:two_stage_dro_reform_obj_uncertainty} and~\ref{coro:two_stage_dro_obj_indicator} imply that the distributionally robust two-stage problem~\eqref{eq:two_stage_dro} admits a convex hull reformulation of the form~\eqref{eq:two_stage_dro_reform}, 
    where the function $Z_i :\mathcal{X} \times \mathbb{R}_{+} \mapsto \mathbb{R}$ and the MICP-representable set $\mathcal{Z}_i$ for each $i\in [N]$ are given as follows:
    \begin{align*}
        &\displaystyle Z_i(\mb{x}, \alpha) = 
        \mathop{\text{maximize}}_{(\bm{\xi}, \bm{\lambda}, \bm{\mu}, \tau) \in \mathop{\mathrm{cl\,conv}}(\mathcal{Z}_i)}  \displaystyle \left\{\bm{h}(\bm{x})^\top \bm{\lambda} + \rho \bm{f}_0^\top \bm{\xi} - \bm{f}_0^\top\bm{\mu}
        - \alpha \tau \right\}  \\
        &\displaystyle \mathcal{Z}_i = \left\{
        (\bm{\xi}, \bm{\lambda}, \bm{\mu}, \tau) \in \Xi \times \mathbb{R}^L_{+} \times \mathbb{R}^{M}_{+} \times \mathbb{R}_{+}: 
        \begin{array}{l}
            \displaystyle (\mb{\xi} - \hat{\mb{\xi}}^{(i)}, \tau) \in \mathcal{C}^{M+1} \\
            \displaystyle \bm{q}_0 + \bm{Q}\mb{\xi} + \rho \bm{F}^\top \bm{\xi} - \bm{W}_0^\top \bm{\lambda} - \bm{F}^\top \bm{\mu} \in \mathcal{Y}^*
        \end{array}
        \right\}.
    \end{align*}
Similarly, a value for the penalty parameter $\rho$ can be computed as follows.
First, we compute the classical robust solution $\bm{x}^r$ and a corresponding worst-case realization by solving the problems:
\[
\bm{x}^r \in \mathop{\arg\min}_{\bm{x} \in \mathcal{X}} \;
\left\{c(\bm{x})
+
\max_{\bm{\xi} \in \Xi^0}
\mathcal{Q}_\mathrm{ind}(\bm{x}, \bm{\xi}) \right\}, \;
\bm{\xi}^r \in \mathop{\arg\max}_{\bm{\xi} \in \Xi^0} \;
\mathcal{Q}_\mathrm{ind}(\bm{x}^r, \bm{\xi}).
\]
Next, we set $\rho^r$ to be an optimal Lagrange multiplier of the last inequality of the following problem:
\[
    \inf_{\bm{y} \in \mathcal{Y}}
    \left\{
    \bm{q}(\bm{\xi}^r)^\top \bm{y}: 
    \bm{W}_0 \bm{y} \geq \bm{h}(\bm{x}^r), \; \bm{f}(\bm{y}) \geq \bm{0}, \; (\bm{\xi}^r)^\top \bm{f}(\bm{y}) \leq 0
    \right\}
\]
The proof for the validity of $\mathcal{Q}_\mathrm{ind} = \mathcal{Q}^{\rho^r}_\mathrm{ind}$ is similar to that of Theorem~\ref{thm:rho_computation} and we omit it for the sake of brevity.
Finally, note that one can avoid the computation of $\bm{x}^r$ and $\bm{\xi}^r$ (as before) by examining the structure of $\bm{f}(\bm{y})$ and $\bm{q}(\bm{\xi})$, see Proposition~\ref{lem:disrupt_all_lines}.
}

\subsection{Summary and comparison}\label{sec:micp_discussion}
Table~\ref{table:micp_comparison} summarizes the main differences between the linearized and penalty-based MICP reformulations of the sets $\mathcal{Z}_i$, $i \in [N]$ appearing in the convex hull reformulation~\eqref{eq:Z_function_definition}--\eqref{eq:Z_set_definition}.
Notably, the penalty reformulation adds far fewer variables and constraints.
However, it also requires additional assumptions and computations.
In particular, it requires computing a value for the penalty parameter $\rho$, which may further entail the solution of a classical robust optimization problem; see~\eqref{eq:classical_ro}.
We do not expect this to be a limitation, however, because the latter will likely reduce to a deterministic optimization problem for most practical applications.

\begin{table}[!htbp]
    \caption{Summary of the MICP representations of $\mathcal{Z}_i$ based on the linearized and penalty reformulations.\label{table:micp_comparison}}
    \begin{tabularx}{\textwidth}{llllL}
        \toprule
        Reformulation & Size & $\bm{W}(\bm{\xi})$ & $\bm{T}(\bm{x})$ & Requirements \\ \midrule
        Linearized~\eqref{eq:Z_set_definition_linearized} & $O(ML)$ & affine & convex & \emph{a priori} bounds on $\bm{\lambda}$ in $\mathcal{Q}_d(\bm{x}, \bm{\xi})$ \\ 
        Penalty~\eqref{eq:Z_set_definition_penalty} & $O(M + L)^*$ & constant$^\dagger$ & constant$^\dagger$ & computation of worst-case realization over any superset of $\Xi^\sharp$ \\
        \bottomrule
        \multicolumn{5}{l}{\footnotesize $^*$Can be further reduced; see Corollary~\ref{coro:two_stage_dro_obj_indicator}} \\
        \multicolumn{5}{l}{\footnotesize $^\dagger$Can be relaxed; see discussion following assumption~(A3).} \\
        \multicolumn{5}{l}{\footnotesize $^\sharp$Can be done in closed form; see Lemma~\ref{lem:disrupt_all_lines} and preceding discussion.}
    \end{tabularx}
\end{table}

%% file: liftandproject.tex
\section{Lift-and-Project Approximations}\label{sec:lift_and_project}
The key challenge in solving the convex hull reformulation~\eqref{eq:two_stage_dro_reform} is the inner optimization~\eqref{eq:Z_function_definition} over the convex hull of the MICP-representable set $\mathcal{Z}_i$, $i \in [N]$.
Appendix~\ref{sec:complexity} shows that although one can tractably compute these convex hulls in several cases, the problem remains NP-hard even in benign settings.
Therefore, Appendix~\ref{sec:benders} presents a Benders scheme, similar to the ones proposed in~\citet{Thiele2009:Benders,zhao2014data}, to tackle the convex hull constraints.
This scheme iteratively refines an \emph{inner approximation} of the MICP representation of $\mathcal{Z}_i$.
\re{An alternative to solving the convex hull reformulation~\eqref{eq:two_stage_dro_reform} is direct solution of the original reformulation~\eqref{eq:GK} using a column-and-constraint generation scheme~\citep{zeng2013solving}.
In contrast to the Benders scheme, the latter models the second-stage loss function $\mathcal{Q}(\mb{x}, \mb{\xi})$ via explicit second-stage variables and constraints by implicitly enumerating $\bm{\xi} \in \Xi$.}
In both schemes, however, each iteration requires the solution of $N$ global MICP problems and therefore, they can become computationally prohibitive. %
Moreover, intermediate solutions obtained from early termination provide no guarantees whatsoever since they bound the optimal value of the distributionally robust problem~\eqref{eq:two_stage_dro} from below, which itself is an upper bound on the true (unknown) optimal value.

These observations motivate the development of \emph{tractable outer approximations} of the convex hulls of $\mathcal{Z}_i$, $i \in [N]$, that provide not only \textit{(i)} upper bounds on the optimal value of~\eqref{eq:two_stage_dro} but also \textit{(ii)} guarantees of polynomial time solvability.
Our approximations are based on the following key observations.
First, if we suppose that $\Xi = \{\bm{\xi} \in \mathbb{Z}^{M}_{+}: \bm{E}\bm{\xi} \leq \bm{f}\}$ has a given outer description, then Section~\ref{sec:reformulation} establishes that each of the sets $\mathcal{Z}_i$, $i \in [N]$, can be represented as the feasible region of an MICP as follows:
\begin{equation}\label{eq:micp_set}
\mathcal{Z} = \left\{\bm{z} \in \mathbb{R}^n: \bm{A} \bm{z} - \bm{b} \in \mathcal{K} , \;\; z_{j} \in \{0, 1\}\;  j \in [M] \right\},
\end{equation}
where, for ease of exposition, we have dropped the subscript $i$ and included the bounds, $\bm{z} \geq \bm{0}$ and $1 \geq z_{j} \coloneqq \xi_j$, $j \in [M]$ in $\bm{A} \bm{z} - \bm{b} \in \mathcal{K}$.
For example, in case of~\eqref{eq:Z_set_definition_linearized}, we have $\bm{z} = (\bm{\xi}, \bm{\lambda}, \vecmat{\bm{\Lambda}}, \tau)$, $n = M + L + ML + 1$, $\mathcal{K} = \tilde{\mathcal{K}} \times \mathbb{R}^n_{+}$, where $\tilde{\mathcal{K}} = \mathbb{R}_{+}^F \times_{i = 1}^3 \mathbb{R}^{LM}_{+} \times \mathcal{C}^{M+1} \times \mathcal{Y}^*$ and $F$ is the dimension of $\bm{f} \in \mathbb{R}^F$, and $\bm{A} = [\tilde{\bm{A}}^\top \, \eye]^\top$ and $\bm{b} = [\tilde{\bm{b}}^\top \, \bm{0}^\top]^\top$ model the right-hand half of~\eqref{eq:Z_set_definition_linearized}.
Similarly, in case of~\eqref{eq:Z_set_definition_penalty}, we have $\bm{z} = (\bm{\xi}, \bm{\lambda}, \bm{\mu}, \tau)$, $n = 2M + L + 1$, $\mathcal{K} = \tilde{\mathcal{K}} \times  \mathbb{R}^n_{+}$, where $\tilde{\mathcal{K}} = \mathbb{R}_{+}^F \times \mathbb{R}^{M}_{+} \times \mathcal{C}^{M+1} \times \mathcal{Y}^*$, and $\bm{A} = [\tilde{\bm{A}}^\top \, \eye]^\top$, $\bm{b} = [\tilde{\bm{b}}^\top \, \bm{0}^\top]^\top$ for suitable matrices $\tilde{\bm{A}}$ and $\tilde{\bm{b}}$.

Second, given an MICP representation such as the above, its convex hull can be approximated by a hierarchy of increasingly tight \emph{convex relaxations},
\[
\mathcal{Z}^{0} \supseteq \mathcal{Z}^{1} \supseteq \ldots \supseteq \mathcal{Z}^{M} = \conv{\mathcal{Z}},
\]
that converge to the convex hull in $M$ iterations.
Here, $\mathcal{Z}^0 \coloneqq \left\{\bm{z} \in \mathbb{R}^n :  \bm{A} \bm{z} - \bm{b} \in \mathcal{K} \right\}$ is the continuous relaxation of $\mathcal{Z}$.
Several such \emph{sequential convexification} hierarchies are known, the most popular ones being those of~\citet{lovasz1991cones,balas1993lift,sherali1990hierarchy,lasserre2001global}.
They are based on the concept of \emph{lift-and-project} and represent $\conv{\mathcal{Z}}$ as the \emph{projection} of another convex set lying in a higher-dimensional space.
These hierarchies were originally proposed for (pure or mixed-) integer linear sets and later extended to mixed-integer convex sets in~\citet{Stubbs1999ABM,Cezik2005}.
Our proposal is to use an intermediate relaxation $\mathcal{Z}^{t}$ of any such hierarchy to outer approximate $\conv{\mathcal{Z}}$, which results in an outer approximation of the convex hull reformulation~\eqref{eq:two_stage_dro_reform} and, hence, \re{a conservative approximation} of the distributionally robust two-stage problem~\eqref{eq:two_stage_dro}.
Notably, since we can optimize an objective function over the level-$t$ relaxation in time $n^{O(t)}$ (which is polynomial for fixed $t$), we can also obtain tight approximations of the original problem~\eqref{eq:two_stage_dro} in polynomial time.
The approximation can be refined, if desired, by using higher values of $t$.

Third, the approximation of $\conv{\mathcal{Z}}$ when used in the convex hull reformulation allows us to dualize the inner optimization in~\eqref{eq:Z_function_definition} using conic duality.
The result is a single-stage convex conic optimization model that can be solved with off-the-shelf solvers.
Notably, we can prove that the resulting approximations of the distributionally robust two-stage problem~\eqref{eq:two_stage_dro}, obtained by replacing $\conv{\mathcal{Z}}$ with any of the relaxations $\mathcal{Z}^0, \ldots, \mathcal{Z}^M$, become exact if the radius $\varepsilon$ of the Wasserstein ambiguity set $\mathcal{P}$ shrinks to zero \re{with increasing sample size $N$}.
\begin{theorem}[Lift-and-project approximation quality]\label{thm:exactness_for_zero_radius}
    \re{Suppose that $N\varepsilon_N \to 0$ as $N \to \infty$. Then, }
    the optimal value of the distributionally robust two-stage problem~\eqref{eq:two_stage_dro} coincides with that of the convex hull reformulation~\eqref{eq:two_stage_dro_reform} even if we approximate each $\conv{\mathcal{Z}_i}$, $i \in [N]$, with $\mathcal{Z}_i^0$ in~\eqref{eq:Z_function_definition}.
\end{theorem}

\begin{proof}
    See Appendix~\ref{sec:proofs}.
\end{proof}

We emphasize that our method is not tied to any particular convexification technique.
This feature is important because each technique has its advantages and disadvantages.
For example, in the linear case (i.e., $\mathcal{K} = \mathbb{R}^{n'}_{+}$), it is known~\citep{laurent2003comparison} that the approximations in order of decreasing tightness are those of~\citet{lasserre2001global}, \citet{sherali1990hierarchy}, \citet{lovasz1991cones}, and~\citet{balas1993lift}; however, this ranking is reversed when they are ordered based on increasing computational complexity.
For its simplicity and tradeoff between tightness and tractability, we focus on the Lov{\'a}sz-Schrijver approximation~\citep{lovasz1991cones} in the remainder of this section.
We show how it can be used to obtain a single-stage approximation of the distributionally robust two-stage problem~\eqref{eq:two_stage_dro}, and we provide practical guidelines for its efficient computation.

\subsection{Lov{\'a}sz-Schrijver approximation}\label{sec:lift_and_project.lov_sch}

The level-$1$ Lov{\'a}sz-Schrijver approximation $\mathcal{Z}^1$ is defined as a set-valued mapping, and the level-$t$ approximation $\mathcal{Z}^t$ is defined as an iterated application of this mapping. %
For any $u \in \mathbb{R}_{+}$ and any conic representable set such as the continuous relaxation  $\mathcal{Z}^0 = \left\{\bm{z} \in \mathbb{R}^n :  \bm{A} \bm{z} - \bm{b} \in \mathcal{K} \right\}$, we denote by $\mathcal{Z}^0(u) = \left\{\bm{z} \in \mathbb{R}^n :  \bm{A} \bm{z} - \bm{b}u \in \mathcal{K} \right\}$ to be the homogenization of $\mathcal{Z}^0$ with respect to $u$.
Next, we define the following lifted set:%
\begin{equation}\label{eq:lovasz_schrijver_lifted}
\mathcal{L}(\mathcal{Z}^0) = \left\{
\bm{z}, \{\bm{z}^{j0}\}_{j \in [M]}, \{\bm{z}^{j1}\}_{j \in [M]} \in \mathbb{R}^n:
\begin{array}{l@{\;}l}
\displaystyle \exists u^{j0}, u^{j1} \geq 0, \; u^{j0} + u^{j1} = 1, & j \in [M] \\
\displaystyle \bm{z}^{j0} \in \mathcal{Z}^0(u^{j0}), \; \bm{z}^{j1} \in \mathcal{Z}^0(u^{j1}), & j \in [M] \\
\displaystyle \bm{z} = \bm{z}^{j0} + \bm{z}^{j1}, & j \in [M] \\
\displaystyle z^{j0}_j = 0, \; z^{j1}_j = u^{j1}, & j \in [M] \\
\displaystyle z^{j1}_k = z^{k1}_j, \;\;k \in [M]: k>j, & j \in [M]
\end{array}\right\}.
\end{equation}
Consider now the following set-valued map, which is the projection of $\mathcal{L}(\mathcal{Z}^0)$ onto $\mathbb{R}^n$:
\begin{equation}
\mathcal{P}(\mathcal{Z}^0) = \left\{
\bm{z} \in \mathbb{R}^n : \exists \bm{z}^{j0}, \bm{z}^{j1}, \; j\in [M] \text{ such that }
\left(\bm{z}, \{\bm{z}^{j0}\}_{j \in [M]}, \{\bm{z}^{j1}\}_{j \in [M]} \right) \in \mathcal{L}(\mathcal{Z}^0)
\right\}. \label{eq:lovasz_schrijver_projected}
\end{equation}
One can easily verify that $\mathcal{P}(\mathcal{Z}^0)$ is a convex relaxation of $\conv{\mathcal{Z}}$.
In fact, we have the following relationship \re{\cite[Theorem~1]{Cezik2005}}:
$
\conv{\mathcal{Z}}  \subseteq \mathcal{P}(\mathcal{Z}^0) \subseteq \bigcap_{j \in [M]} \conv{\left\{\bm{z} \in \mathcal{Z}^0: z_j \in \{0, 1\} \right\}} \subseteq \mathcal{Z}^0.
$
The set $\mathcal{P}(\mathcal{Z}^0)$ corresponds to the level-1 relaxations of the Lov{\'a}sz-Schrijver hierarchy.
For any $t \geq 1$, the level-$t$ relaxation is given by $\mathcal{Z}^t = \mathcal{P}(\mathcal{Z}^{t-1})$, and one can show that $\mathcal{Z}^M = \conv{\mathcal{Z}}$.
This is known as the \emph{linear} Lov{\'a}sz-Schrijver hierarchy.
One can obtain stronger relaxations by imposing positive semidefiniteness on the submatrix of $\{\bm{z}^{j1}\}_{j \in [M]}$ corresponding to the binary variables.
\begin{remark}[Positive semidefinite Lov{\'a}sz-Schrijver hierarchy]
    Consider the following set:
    \begin{equation*}
    \mathcal{P}_{+}(\mathcal{Z}^0) = \left\{
    \bm{z} \in \mathbb{R}^n :\hspace{-0.5em} \begin{array}{l}
    \exists \bm{z}^{j0}, \bm{z}^{j1}, \; j\in [M] \text{ such that }
    \left(\bm{z}, \{\bm{z}^{j0}\}_{j \in [M]}, \{\bm{z}^{j1}\}_{j \in [M]} \right) \in \mathcal{L}(\mathcal{Z}^0) \\
    \text{and } \left[\bm{\xi}^{11}  \ldots \bm{\xi}^{M1}\right] \succeq \bm{\xi}\bm{\xi}^\top, \text{where } \bm{\xi} = [z_1 \ldots z_M]^\top, \bm{\xi}^{j1} = [z^{j1}_1 \ldots z^{j1}_M]^\top
    \end{array}
    \right\}.
    \end{equation*}
    This set corresponds to the level-1 relaxation of the \emph{positive semidefinite} Lov{\'a}sz-Schrijver hierarchy and is related to its linear counterpart as follows: $\conv{\mathcal{Z}} \subseteq \mathcal{P}_{+}(\mathcal{Z}^0) \subseteq \mathcal{P}(\mathcal{Z}^0)$.
    For any $t \geq 1$, the level-$t$ relaxation in this hierarchy is given by $\mathcal{Z}^t = \mathcal{P}_{+}(\mathcal{Z}^{t-1})$.
    As before, we have $\mathcal{Z}^M = \conv{\mathcal{Z}}$.
\end{remark}

For a given MICP representation of $\mathcal{Z}_i$ and any $t \geq 1$, we can use the level-$t$ Lov{\'a}sz-Schrijver relaxation to approximate $\conv{\mathcal{Z}_i}$ in~\eqref{eq:Z_function_definition}--\eqref{eq:Z_set_definition}.
We can then convert the inner maximization~\eqref{eq:Z_function_definition} to a minimization and embed it in the first-stage problem.
The following lemma illustrates this for $t = 1$; we omit the proof since it follows from a straightforward application of conic duality.
In this lemma, $\bm{\gamma}$ represents the objective function of the inner optimization~\eqref{eq:Z_function_definition}; in the linearized MICP representation~\eqref{eq:Z_set_definition_linearized} of $\mathcal{Z}$, it is given by $\bm{\gamma} = \left[\bm{0}^\top\; \bm{h}(\bm{x})^\top\, \vecmat{\bm{T}(\bm{x})}^\top\, -\alpha \right]^\top$, whereas in the penalty MICP representation~\eqref{eq:Z_set_definition_penalty}, we have $\bm{\gamma} = \left[\rho\one^\top\; \bm{h}(\bm{x})^\top\, -\one^\top\, -\alpha \right]^\top$.
\begin{lemma}\label{lem:dualization}
    Suppose that $\mathcal{Z}$ is defined as in~\eqref{eq:micp_set} with $\mathcal{K} = \tilde{\mathcal{K}} \times  \mathbb{R}^n_{+}$, $\bm{A} = [\tilde{\bm{A}}^\top \, \eye]^\top$ and $\bm{b} = [\tilde{\bm{b}}^\top \, \bm{0}^\top]^\top$ for suitable matrices $\tilde{\bm{A}}$ and $\tilde{\bm{b}}$. If $\mathcal{Z}^1$ denotes the level-1 Lov{\'a}sz-Schrijver relaxation of $\conv{\mathcal{Z}}$, then for any $\bm{\gamma} \in \mathbb{R}^n$, we have the following strong duality result:
    \begin{equation}\label{eq:lovasz_schrijver_lifted_dual}
        \begin{array}{r@{\;\;}l}
            \displaystyle\sup_{\bm{z} \in \mathcal{Z}^1} \,\bm{\gamma}^\top \bm{z} \;=\; 
            \mathop{\text{minimize}} & \displaystyle \sum_{j \in [M]} \max \left\{ -\tilde{\bm{b}}^\top \tilde{\bm{\zeta}}^{j0}, -\tilde{\bm{b}}^\top \tilde{\bm{\zeta}}^{j1} - \beta^{j1} \right\} \\
            \text{subject to} & \displaystyle
            \sum_{j \in [M]} \bm{\delta}^{j} + \bm{\gamma} \leq \bm{0}, \\
            & \displaystyle \begin{drcases}
            \tilde{\bm{A}}^\top \tilde{\bm{\zeta}}^{j\ell} + \beta^{j\ell} \one_j - \mathbb{I}[\ell=1]\bm{Y}\one_j \leq \bm{\delta}^{j} \\
            \bm{\delta}^{j} \in \mathbb{R}^n, \; %
            \tilde{\bm{\zeta}}^{j\ell} \in \tilde{\mathcal{K}}^*, \; \beta^{j\ell} \in \mathbb{R}
            \end{drcases} \; \forall \ell \in \{0, 1\}, j \in [M], \\
            & \bm{Y} \in \mathbb{R}^{M \times M}: \bm{Y} = -\bm{Y}^\top.
        \end{array}
    \end{equation}
\end{lemma}

\re{\begin{remark}[Relationship to approximations obtained by relaxing the support]\label{rem:continuous_support_sets}
        An alternative outer approximation of the distributionally robust two-stage problem~\eqref{eq:two_stage_dro} can be obtained by simply relaxing the zero-one constraints on the support in reformulation~\eqref{eq:GK}; \textit{i.e.}, by replacing $\Xi = \{\bm{\xi} \in \mathbb{Z}^{M}_{+}: \bm{E}\bm{\xi} \leq \bm{f}\}$ in~\eqref{eq:GK} with its continuous relaxation $\{\bm{\xi} \in \mathbb{R}^{M}_{+}: \bm{E}\bm{\xi} \leq \bm{f}\}$.
        The resulting approximation is intractable in general, unless the uncertainty $\bm{\xi}$ appears only in the objective of the loss function $\mathcal{Q}(\bm{x}, \bm{\xi})$ (see discussion in Section~\ref{sec:reformulation}).
        In the latter case, it can be easily seen that the resulting approximation coincides precisely with that obtained by replacing $\conv{\mathcal{Z}_i}$ in~\eqref{eq:Z_function_definition_obj_unc}--\eqref{eq:Z_set_definition_obj_unc} with the continuous relaxation $\mathcal{Z}_i^0$.
        Therefore, our lift-and-project approximation technique can be viewed as a generalization of this approach.
        Unlike the former, however, a crucial difference is that this approach provides no formal mechanism to improve the quality of the final approximation; we illustrate this empirically in Section~\ref{sec:computational_experiments}.
\end{remark}
}

\subsection{Numerical considerations}\label{sec:lift_and_project_numerical}
Several factors can impact the numerical solution of the approximations obtained by replacing $\conv{\mathcal{Z}_i}$, $i \in [N]$, in~\eqref{eq:Z_function_definition}, with their level-1 Lov{\'a}sz-Schrijver relaxations $\mathcal{Z}^1_i$.
First, if we use the linearized MICP representation~\eqref{eq:Z_set_definition_linearized} to reformulate $\mathcal{Z}_i$, then formulation~\eqref{eq:lovasz_schrijver_lifted_dual} has $O(M(F+N_2+LM))$ variables, $O(M^2L)$ linear constraints, and $O(M(M+N_2))$ conic constraints.
If we use the penalty representation~\eqref{eq:Z_set_definition_penalty}, then these are reduced to $O(M(F+N_2+L+M))$ variables, $O(M(L+M))$ linear constraints, and $O(M(M+N_2))$ conic constraints. 
In either case, the number of conic constraints can be reduced to $O(MN_2)$ whenever the metric $d(\bm{\xi}, \bm{\xi}') = \norm{\bm{\xi} - \bm{\xi}'}_1$ is induced by the 1-norm (see argument in proof of Proposition~\ref{obs:ideal} in Appendix~\ref{sec:complexity}).

Second, observe that setting $\bm{Y} = \bm{0}$ also reduces the number of variables in~\eqref{eq:lovasz_schrijver_lifted_dual} for only a minor loss in approximation quality; indeed, the resulting relaxation is still equal to $\bigcap_{j \in [M]} \conv{\{\bm{z} \in \mathcal{Z}^0_i: z_j \in \{0, 1\} \}}$.
In fact, it is equal to $\bigcap_{j \in \mathcal{J}_i} \conv{\{\bm{z} \in \mathcal{Z}^0_i: z_j \in \{0, 1\} \}}$, where $\mathcal{J}_i \subseteq [M]$ is the index set of binary parameters whose optimal values in the left-hand side of~\eqref{eq:lovasz_schrijver_lifted_dual} are fractional.
We expect $\abs{\mathcal{J}_i}$ to be small since the optimal value of $\bm{\xi}$ in any convex relaxation of $\mathcal{Z}_i$ is unlikely to be far from the binary-valued $\hat{\bm{\xi}}^{(i)}$ (see argument in proof of Theorem~\ref{thm:exactness_for_zero_radius}).
This motivates the following iterative heuristic to identify the index sets $\mathcal{J}_i$.
Note that this procedure is independent of the MICP representation used for each $\mathcal{Z}_i$.
\begin{enumerate}
    \item Select $\texttt{tol} \in (0, 0.5)$, $\texttt{niter} \in \mathbb{Z}_{+}$. Set $\texttt{iter} \gets 1$. For each $i \in [N]$, set $\tilde{\mathcal{Z}}_i^1 \gets \mathcal{Z}^0_i$ and $\mathcal{J}_i \gets \emptyset$.
    \item For each $i \in [N]$, replace $\conv{\mathcal{Z}_i}$ with its current approximation $\tilde{\mathcal{Z}}_i^1$ and dualize the corresponding problem~\eqref{eq:Z_function_definition}. Solve the resulting convex hull approximation~\eqref{eq:two_stage_dro_reform}.
    \item For each $i \in [N]$, let $\bar{\bm{\xi}}^{[i]}$ be the optimal value of $\bm{\xi}$ in~\eqref{eq:Z_function_definition}, recovered as scaled dual multipliers. For each $j \in [M]\setminus \mathcal{J}_i$, if $\bar{{\xi}}^{[i]}_j \in [\texttt{tol}, 1 - \texttt{tol}]$, update $\mathcal{J}_i \gets \mathcal{J}_i \cup \{j\}$ and $\tilde{\mathcal{Z}}_i^1$ as follows:
    \[
    \tilde{\mathcal{Z}}_i^1 \gets \left\{
    \bm{z} \in \mathbb{R}^n:
    \begin{array}{l@{\;}l}
    \displaystyle \exists u^{j0}, u^{j1} \geq 0, \; u^{j0} + u^{j1} = 1, & j \in \mathcal{J}_i \\
    \displaystyle \exists \bm{z}^{j0} \in \mathcal{Z}^0_i(u^{j0}), \; \bm{z}^{j1} \in \mathcal{Z}^0_i(u^{j1}), & j \in \mathcal{J}_i \\
    \displaystyle \bm{z} = \bm{z}^{j0} + \bm{z}^{j1}, & j \in \mathcal{J}_i \\
    \displaystyle z^{j0}_j = 0, \; z^{j1}_j = u^{j1}, & j \in \mathcal{J}_i 
    \end{array}\right\}.
    \]
    \item If none of the index sets $\mathcal{J}_1, \ldots, \mathcal{J}_N$ were updated or if $\texttt{iter} \geq \texttt{niter}$, stop. Otherwise, update $\texttt{iter} \gets \texttt{iter} + 1$ and go to Step~2.
\end{enumerate}
Note that the successive optimizations in Step~2 can benefit from an efficient initialization of their variables by using the optimal solution from the previous solve.
Moreover, the size of these problems can be controlled by using smaller values of $\texttt{niter}$ and larger values of $\texttt{tol}$, since they directly influence the size of $\mathcal{J}_i$ and $\tilde{\mathcal{Z}}_i^1$, albeit at the expense of coarser approximations.
In our implementation, we found that a setting of $\texttt{iterlim} = 5$ and $\texttt{tol} = 10^{-2}$ achieved a good tradeoff between approximation quality and computational effort.

%% file: computational_experiments.tex
\section{Computational Experiments}\label{sec:computational_experiments}

We illustrate the applicability of our method to operational problems in electric power systems in Section~\ref{sec:OPF}, and to \re{design problems in multi-commodity flow networks in Section~\ref{sec:MMCF}}.
Our goals are to:
\textit{(i)} study the lift-and-project approximations $\mathcal{Z}^0$ and $\mathcal{Z}^1$ in terms of their computational effort and ability to approximate $\conv{\mathcal{Z}}$;
\textit{(ii)} compare their out-of-sample performance with the standard sample average approximation \re{and with classical two-stage robust optimization};
and, \textit{(iii)} elucidate the effect of two key parameters on the relative benefits of the distributionally robust two-stage problem~\eqref{eq:two_stage_dro} over these classical formulations: the ``rareness'' of network failures and the relative magnitude of ``impact'' when failures occur.

Our code was implemented in Julia~1.5.3, using JuMP~0.21.4. We used Mosek~9.2 for solving our lift-and-project approximations, and Gurobi~9.1.1 as the solver for the Benders and \re{column-and-constraint generation schemes} (which we compare in Sections~\ref{sec:OPF} and~\ref{sec:MMCF} respectively), since the latter performed better than the former in solving the mixed-integer subproblems in those schemes; whereas Mosek performed better in solving the conic programming relaxations.
All runs were conducted on an Intel~Xeon~2.3 GHz computer, with a limit of four cores per run.

\subsection{Optimal power flow}\label{sec:OPF}

We use our method to address the security-constrained optimal power flow problem that is fundamental to the secure operation of electric power grids and solved every fifteen minutes or so by grid operators (e.g.,see~\citet{alsac1974optimal,chiang2015solving}).
The goal is to determine voltages and generation levels of available generators so as to satisfy power demand in the network, while adhering to various physical and engineering constraints.
For example, electric power between network nodes (also known as \emph{buses}) can flow only along capacitated edges or transmission lines.
As such, the latter are failure prone, and transmission line outages can lead to an unstable power network or even complete system failure, resulting in costly blackouts.
However, such high-impact failure events are rare.
For example, between the years 2000 and 2014, fewer than 1,500 power outages have occurred that affected 50,000 or more residents in the entire United States, which is fewer than 100 events per year~\citep{electricDisturbance}.
This rarity complicates the accurate estimation of their underlying distribution.

Because electric power is governed by complex physical laws, optimal power flow is a highly nonlinear optimization problem.
Nevertheless, the underlying physics can be approximated well by using second-order cone or semidefinite programming relaxations~\citep{low2014convex}.
Although our method generalizes to any convex cone relaxation, we focus on the standard second-order cone relaxation~\citep{kocuk2016strong}, where $\mathcal{X}$ is second-order cone representable and $\mathcal{Q}(\bm{x}, \bm{\xi})$ is the optimal value of a second-order cone program.

Our two-stage optimization model is inspired by~\citet{gocompetition} and presented in Appendix~\ref{appendix:model}.
Conceptually, the first-stage problem determines minimum cost power generation levels assuming no line outages.
Upon line failure, the second-stage model seeks to adjust the power generation levels subject to physical constraints where failed lines cannot be used, with a goal of minimizing the total penalty cost of violating power balances.
This model satisfies assumptions (A1), (A2), and (A3) and allows the use of the penalty reformulation~\eqref{eq:Z_set_definition_penalty}, which also has the advantage of using fewer variables and constraints compared with the linearized reformulation (see Section~\ref{sec:micp_discussion}).

The operational state of transmission lines is modeled as a random binary vector ${\bm{\xi}}$ with support $\Xi = \{0, 1\}^M$, where ${\xi}_{l} =  1$ indicates that line $l$ has failed.
In particular, since ${\bm{\xi}}$ represent on/off switches, we can use %
Corollary~\ref{coro:two_stage_dro_obj_indicator} to get not only a smaller MICP formulation but also tighter values of the penalty parameter $\rho = \rho^r$.
The latter is computed by using Theorem~\ref{thm:rho_computation}, where the classical robust counterpart reduces to a deterministic problem (see Lemma~\ref{lem:disrupt_all_lines}); indeed, the second-stage loss function trivially attains its worst-case value when each component of $\bm{\xi}$ is one, that is, when all transmission lines fail.

\subsubsection{Test instances}\label{sec:test_instances}
We conduct our experiments on the standard IEEE 14-, 30- and 118-bus test cases from the PGLib-OPF library~\citep{babaeinejadsarookolaee2019power}.
In each case, the second-stage per-unit penalty cost for violating the power balance equations is set to be $\phi$ times the maximum per-unit first-stage generation cost.
Note that this choice depends on the economic cost of failure to meet power demand.
Since loss of power and blackouts tend to be costly and the associated penalty costs much larger than the cost of generation, we set $\phi = 100$ in Sections~\ref{sec:results_approx_quality} and \ref{sec:results_out_of_sample} and analyze its sensitivity in Section~\ref{sec:results_sensitivity_analysis}.

We generate empirical data using a Bernoulli model.
Specifically, we model each component of $\tilde{\bm{\xi}}$ as independent and identically distributed Bernoulli random variables with parameter $\psi M^{-1}$, where $M$ denotes the number of transmission lines.
Note that this choice reflects the rare nature of line failures; in particular, it implies that only $\psi\cdot100\%$ of training samples record a line failure.
We set $\psi = 0.1$ in Sections~\ref{sec:results_approx_quality} and \ref{sec:results_out_of_sample} and analyze its impact in Section~\ref{sec:results_sensitivity_analysis}.

In all our experiments, for a fixed sample size $N$ and radius $\varepsilon$, we report average results using 100 statistically independent sets of training samples, and we estimate the variance by reporting the standard deviation over these 100 runs.
In Sections~\ref{sec:results_out_of_sample} and~\ref{sec:results_sensitivity_analysis}, the out-of-sample performances of a candidate solution are estimated by using 1,000 statistically independent sets of testing samples.

\subsubsection{Approximation quality and computational effort}\label{sec:results_approx_quality}
To study the quality of our lift-and-project approximations, we compute the following quantities for each sample size $N \in \{10, 100, 1000\}$ and radius \re{$\varepsilon = \nu \sqrt{N^{-1}\log(N+1)}$},
where $\nu \in \{0, 10^{-3}, 10^{-2}, 10^{-1}\}$: \emph{(i)} the optimal value $v^\star$ of the convex hull reformulation~\eqref{eq:two_stage_dro_reform} using the Benders scheme described in Appendix~\ref{sec:benders}, and \emph{(ii)} the optimal values $v^0$, $\tilde{v}^1$ and $v^1$ of formulation~\eqref{eq:two_stage_dro_reform} when the convex hulls $\conv{\mathcal{Z}_i}$ in~\eqref{eq:Z_function_definition}--\eqref{eq:Z_set_definition} are approximated by using the continuous relaxation $\mathcal{Z}^0$ and heuristically and exactly computed level-1 Lov{\'a}sz-Schrijver relaxations $\tilde{\mathcal{Z}}^1$ and ${\mathcal{Z}}^1$, respectively (see Section~\ref{sec:lift_and_project_numerical}).
\re{The choice of the radius $\varepsilon = \nu \sqrt{N^{-1}\log(N+1)}$ is motivated from Theorem~\ref{thm:finite_sample_guarantee}, and we elaborate on it further in the next subsection.}

Figure~\ref{fig:gaps} reports the average (line plot) and standard deviation (error bar) of the optimality gaps, defined as $\left(v - v^\star\right)/{v^\star} \times 100\%$, where $v \in \{v^0, \tilde{v}^1, v^1\}$, %
based on 100 statistically independent sets of training samples.
We make the following observations.
\begin{itemize}
    \item The exact level-1 relaxation ${\mathcal{Z}}^1$ is near optimal, with optimality gaps never exceeding $10\%$, whereas the continuous relaxation $\mathcal{Z}^0$ is less accurate, especially for larger radii (\emph{e.g.}, 50\% gap for $\nu = 0.1$). The heuristically computed level-1 relaxation $\tilde{\mathcal{Z}}^1$ is also near optimal for small and large radii but performs relatively poorly for intermediate values of $\nu$.
    
    \item For a fixed sample size $N$ and decreasing radius $\nu$, the optimality gaps of all approximations decrease to $0$. %
    For increasing $\nu$, the gaps of the level-1 relaxations increase far less rapidly than that of the continuous relaxation.
    
    \item For a fixed radius $\nu$ and increasing sample size $N$, the gaps of all approximations decrease.
\end{itemize}

\begin{figure}[!htb]
    \centering
    \begin{subfigure}[b]{0.32\linewidth}
        \includegraphics[width=\linewidth]{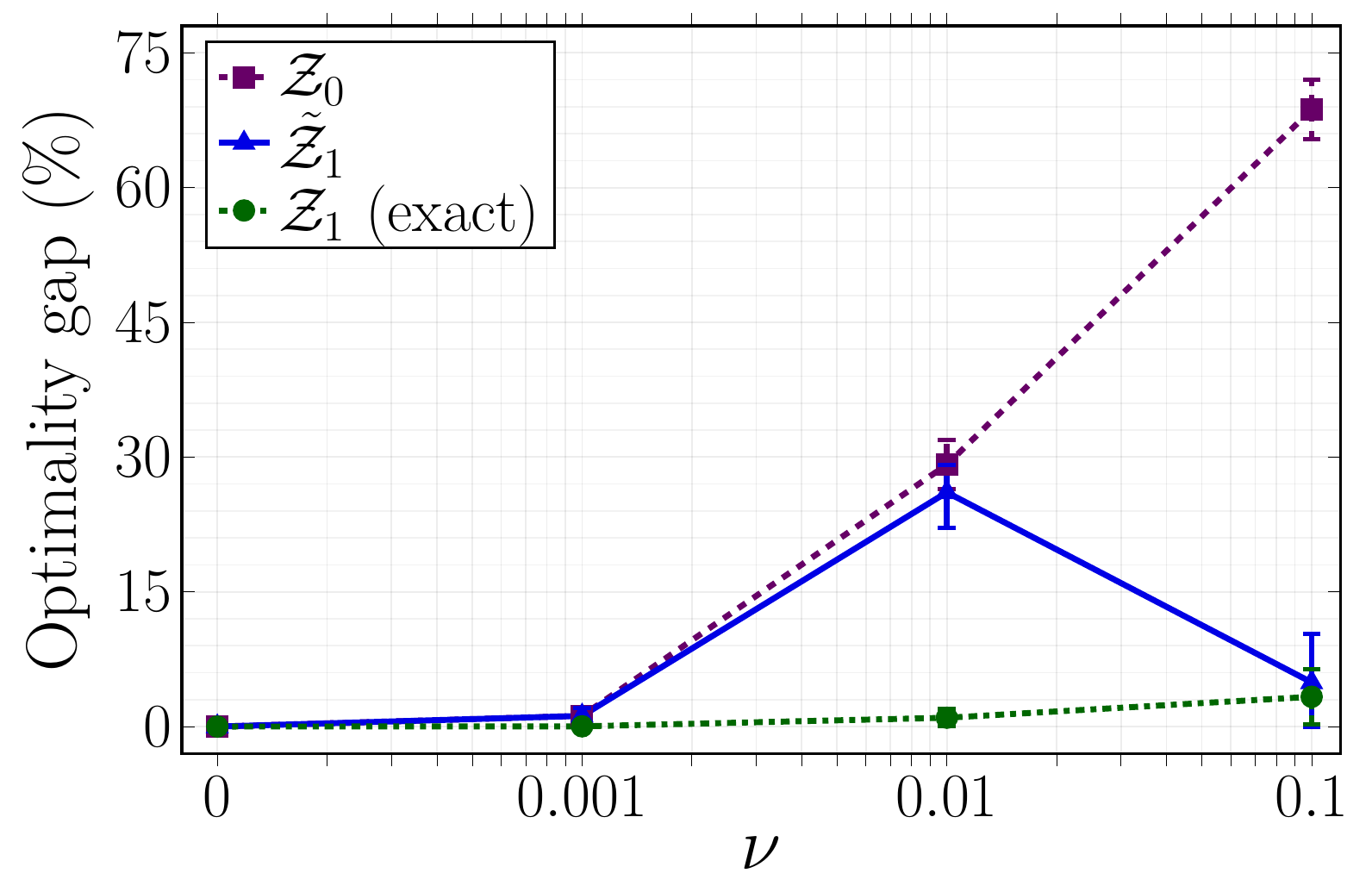}
        \caption{14-bus, $N = 10$}\label{fig:gaps_14_N10}
    \end{subfigure}\hfil
    \begin{subfigure}[b]{0.32\linewidth}
        \includegraphics[width=\linewidth]{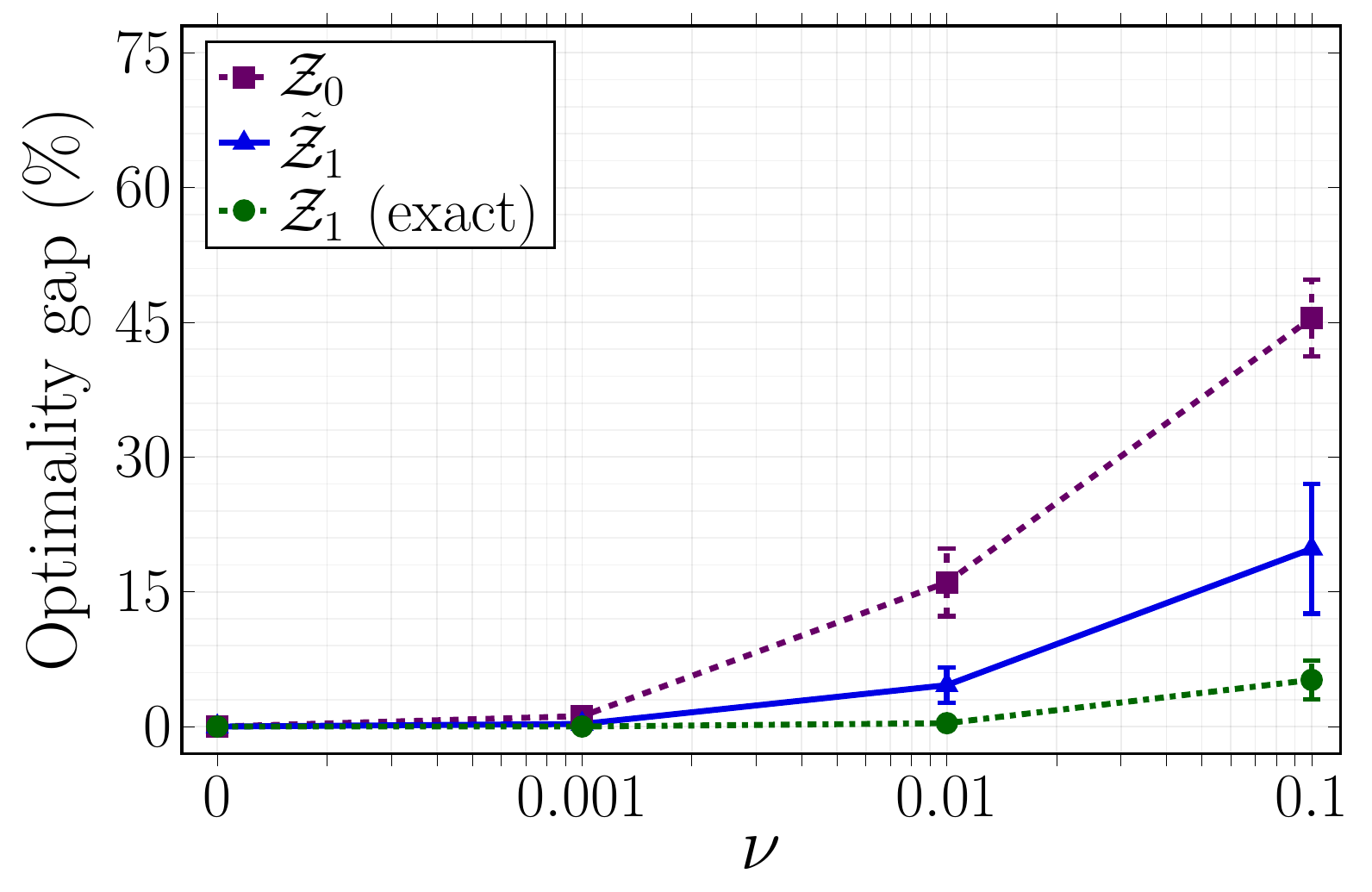}
        \caption{14-bus, $N = 100$}\label{fig:gaps_14_N100}
    \end{subfigure}\hfil
    \begin{subfigure}[b]{0.32\linewidth}
        \includegraphics[width=\linewidth]{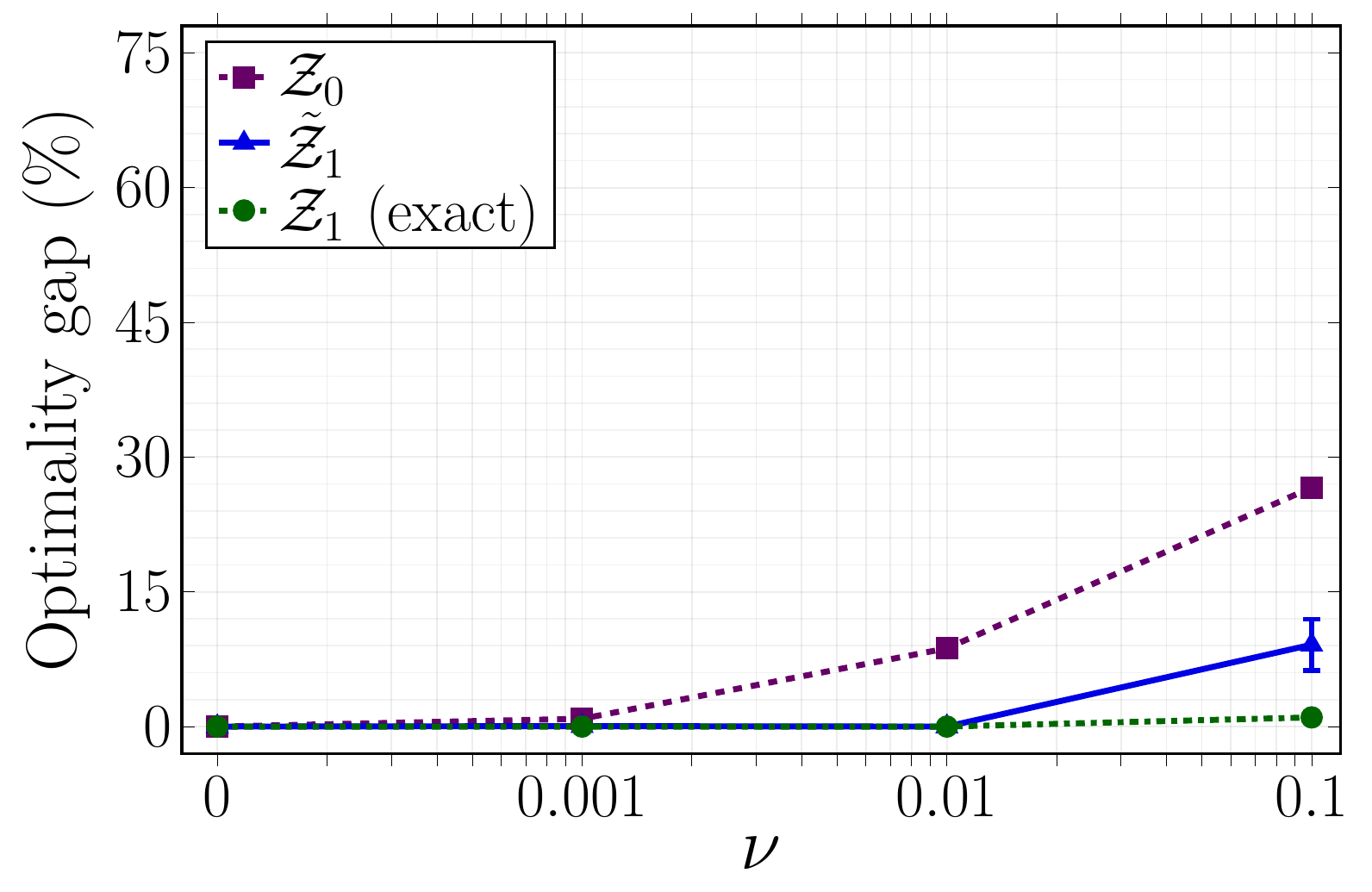}
        \caption{14-bus, $N = 1000$}\label{fig:gaps_14_N1000}
    \end{subfigure}
    \medskip
    \begin{subfigure}[b]{0.32\linewidth}
        \includegraphics[width=\linewidth]{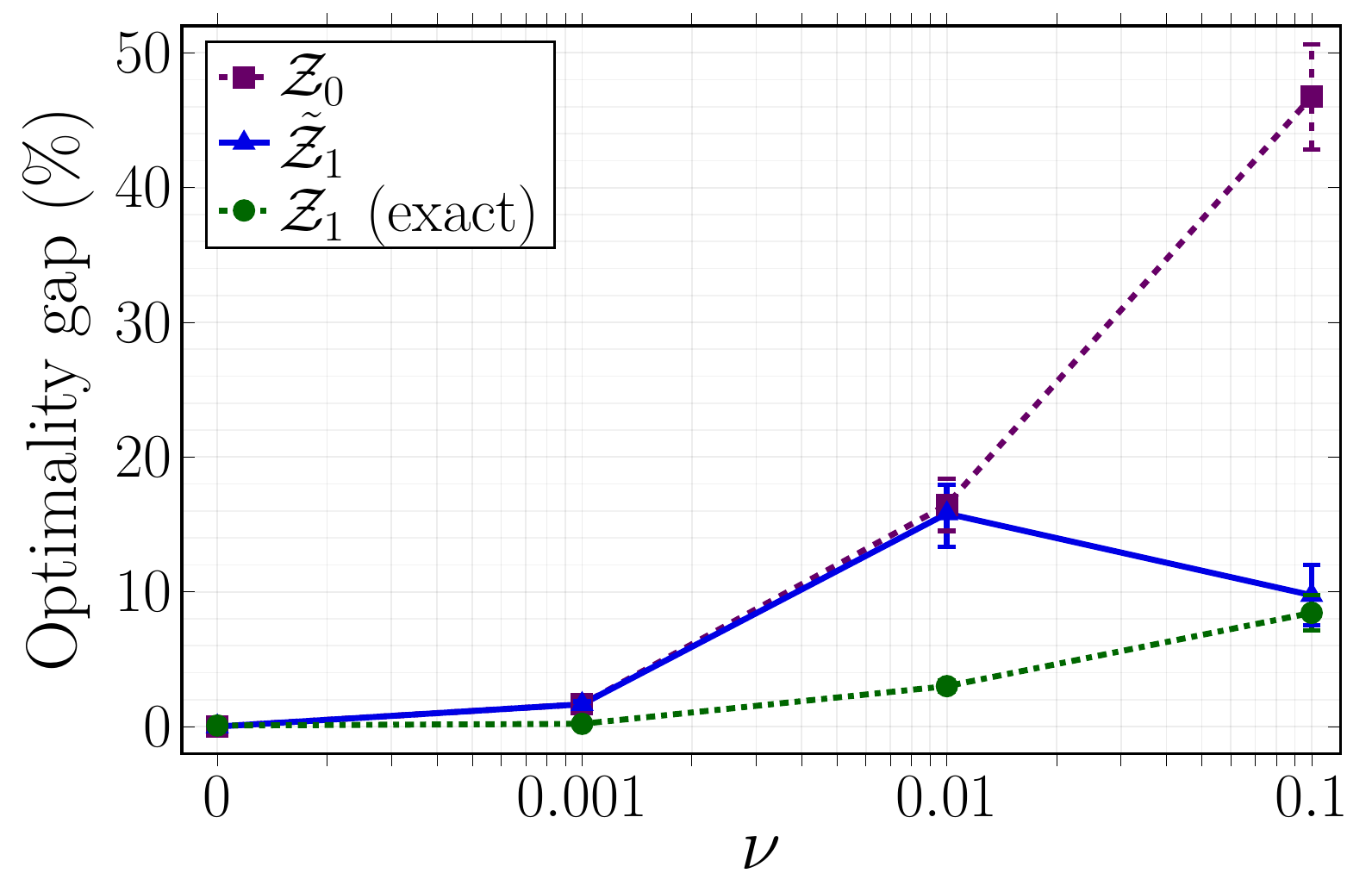}
        \caption{30-bus, $N = 10$}\label{fig:gaps_30_N10}
    \end{subfigure}\hfil
    \begin{subfigure}[b]{0.32\linewidth}
        \includegraphics[width=\linewidth]{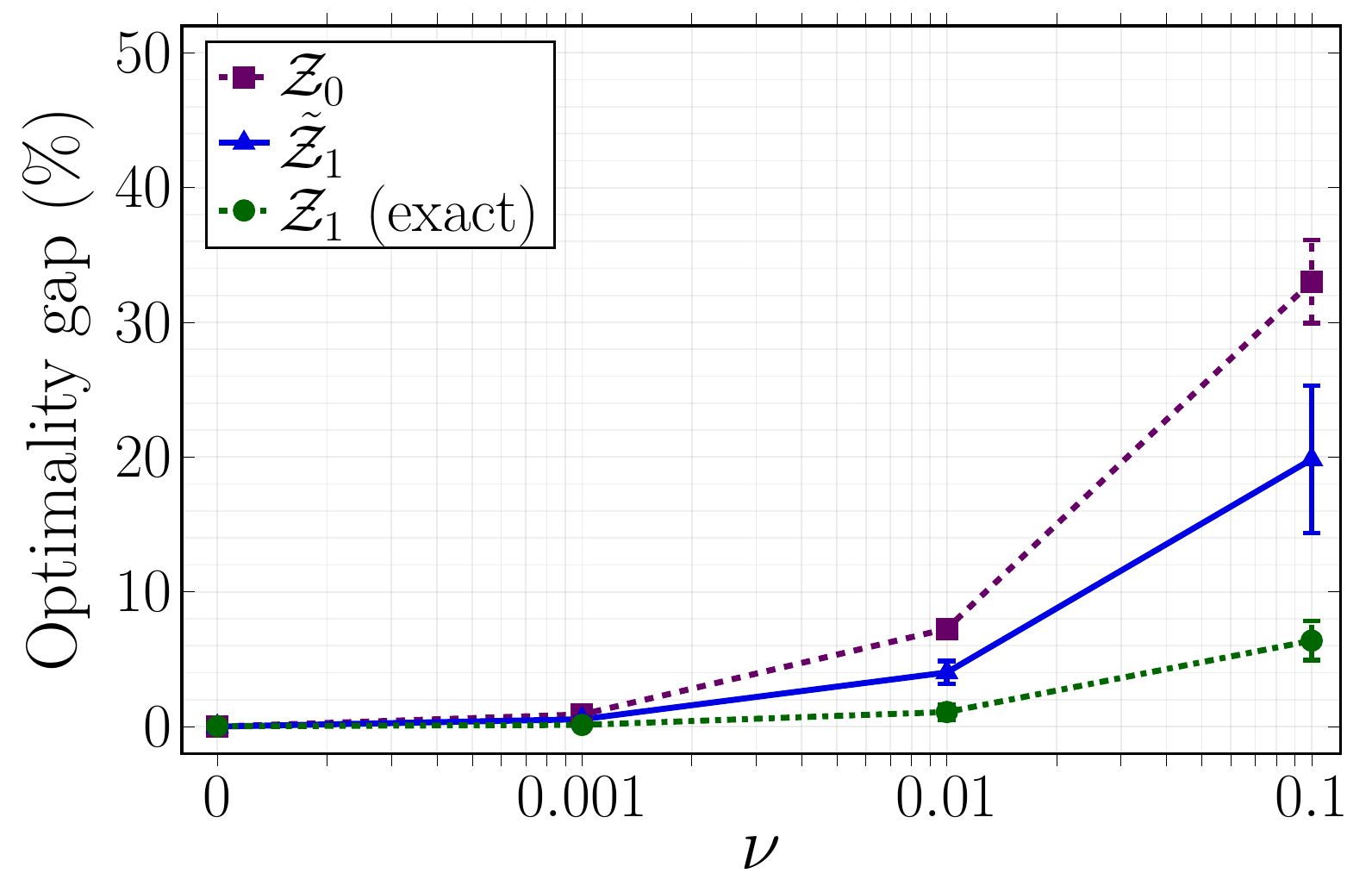}
        \caption{30-bus, $N = 100$}\label{fig:gaps_30_N100}
    \end{subfigure}\hfil
    \begin{subfigure}[b]{0.32\linewidth}
        \includegraphics[width=\linewidth]{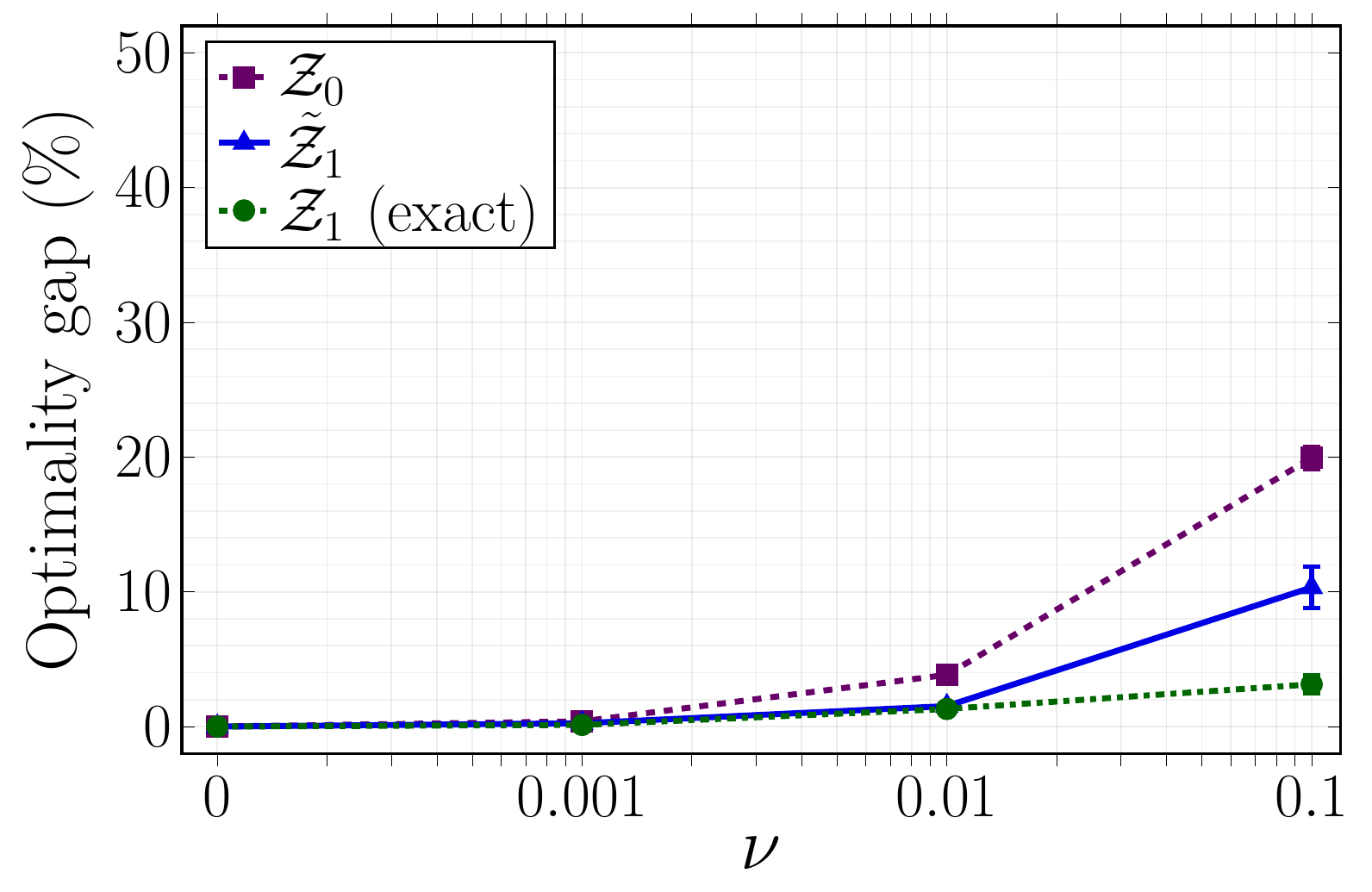}
        \caption{30-bus, $N = 1000$}\label{fig:gaps_30_N1000}
    \end{subfigure}
    \caption{Optimality gaps using the continuous relaxation $\mathcal{Z}^0$ and the heuristically and exactly computed level-1 Lov{\'a}sz-Schrijver relaxations $\tilde{\mathcal{Z}}^1$ and ${\mathcal{Z}}^1$ as a function of $\nu$ and $N$, where \re{$\varepsilon = \nu \sqrt{N^{-1}\log(N+1)}$}.\label{fig:gaps}}
\end{figure}

Figure~\ref{fig:time} reports the average computation time to solve the various approximations and compares them with that of the Benders decomposition scheme.
We offer the following comments.
\begin{itemize}
    \item The continuous relaxation $\mathcal{Z}^0$ and heuristically computed level-1 relaxations $\tilde{\mathcal{Z}}^1$ have the smallest computation times, with the former being faster for larger values of $N$ and $\nu$ and, in particular, for the larger 118-bus case where it is more than 10~times faster. When compared with the Benders scheme, the relative difference in their computation times is minor for small sample sizes $N$ but increases significantly for large sample sizes. For $N = 1000$, both approximations run 10~times faster than the Benders scheme for the 14-bus case, while $\tilde{\mathcal{Z}}^1$ runs 4~times faster and $\mathcal{Z}^0$ runs almost 100~times faster for the 30-bus case.

    \item The exact level-1 relaxation ${\mathcal{Z}}^1$ and the Benders scheme appear to be the most difficult to solve. Although not shown, for the 30-bus case the former took about 1, 2 and 10 minutes for $N = 10, 100$, and $1,000$, respectively, whereas for the larger 118-bus case neither scheme terminated within 10~minutes for $N = 10, 100$ or within 1~hour for $N = 1000$. Moreover, some of the MICP subproblems within the Benders scheme can cause slow convergence (e.g., due to search tree enumeration). This is evidenced by the fact that about 1\% of the Benders runs did not terminate within 10~minutes even for the smaller 14-bus and 30-bus cases, respectively.
    
    \item In conjunction with Figure~\ref{fig:time}, the heuristically computed level-1 relaxation $\tilde{\mathcal{Z}}^1$ appears to offer the best tradeoff in terms of approximation quality and computational effort.
    
    \item The run times of all approximations, and in particular $\tilde{\mathcal{Z}}^1$, can be significantly improved by using an efficient initialization of their variables (see Section~\ref{sec:lift_and_project_numerical}), which we did not implement.
\end{itemize}

\begin{figure}[!htb]
    \centering
    \begin{subfigure}[b]{0.32\linewidth}
        \includegraphics[width=\linewidth]{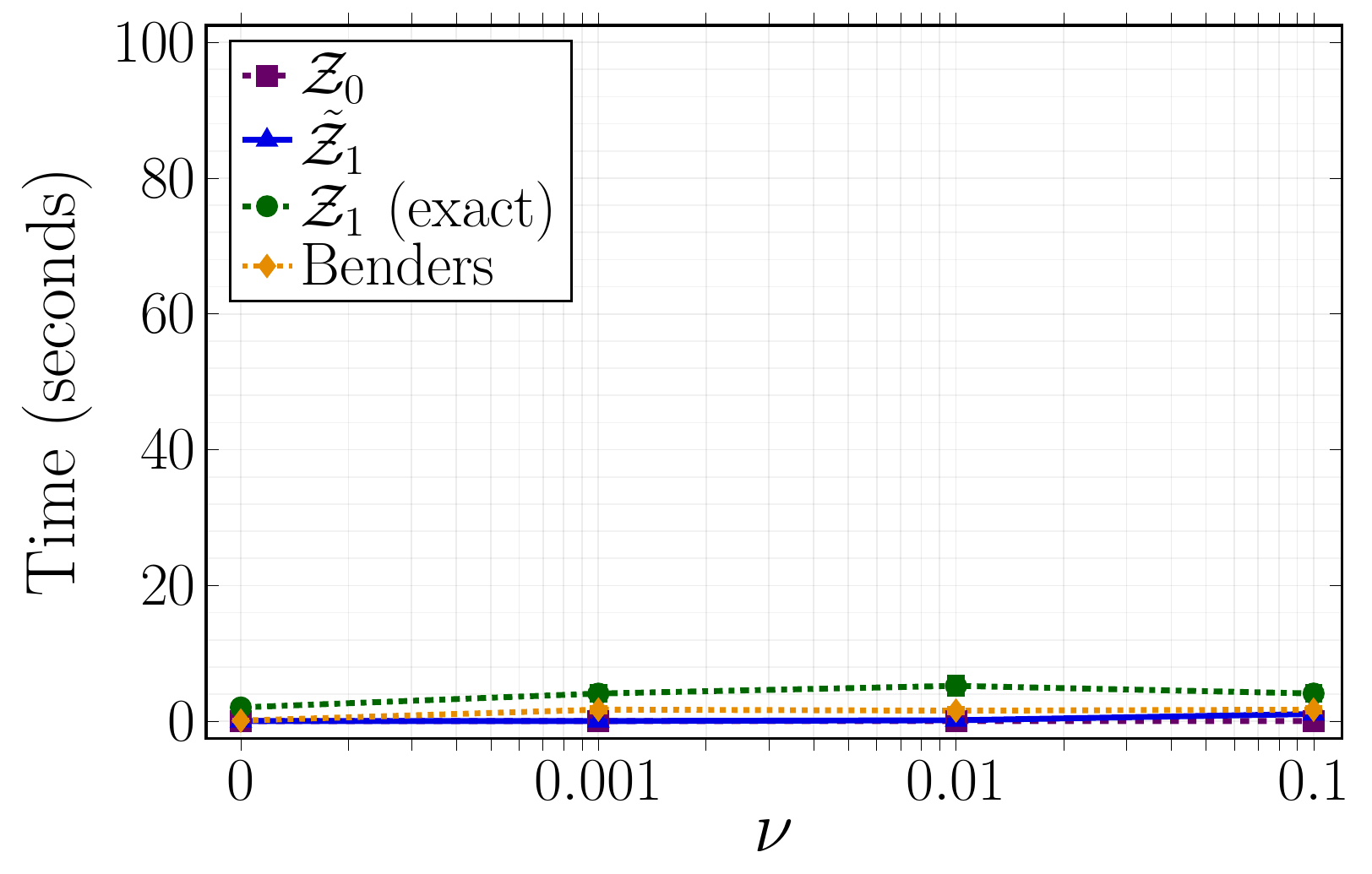}
        \caption{14-bus, $N = 10$}\label{fig:time_14_N10}
    \end{subfigure}\hfil
    \begin{subfigure}[b]{0.32\linewidth}
        \includegraphics[width=\linewidth]{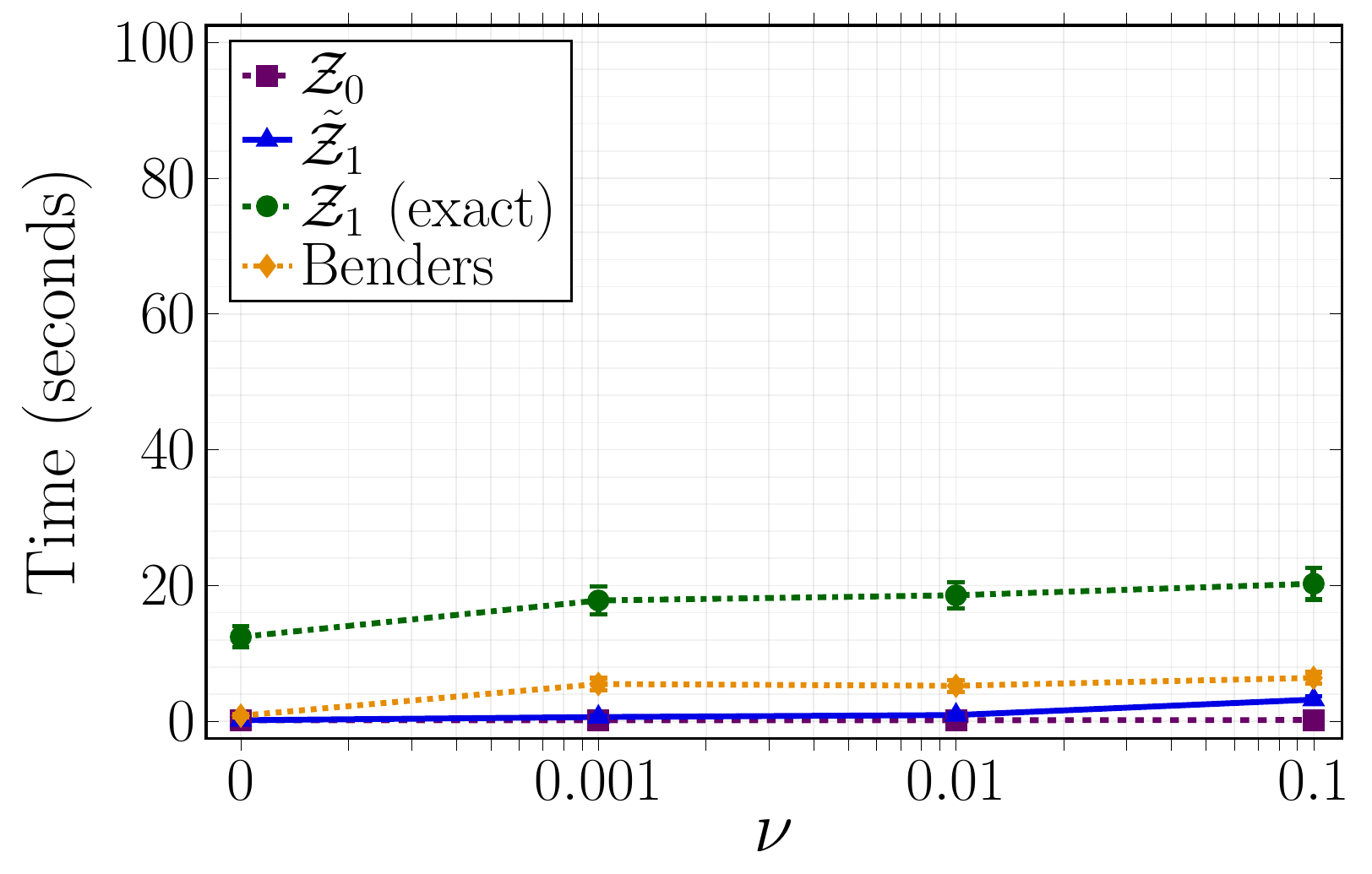}
        \caption{14-bus, $N = 100$}\label{fig:time_14_N100}
    \end{subfigure}\hfil
    \begin{subfigure}[b]{0.32\linewidth}
        \includegraphics[width=\linewidth]{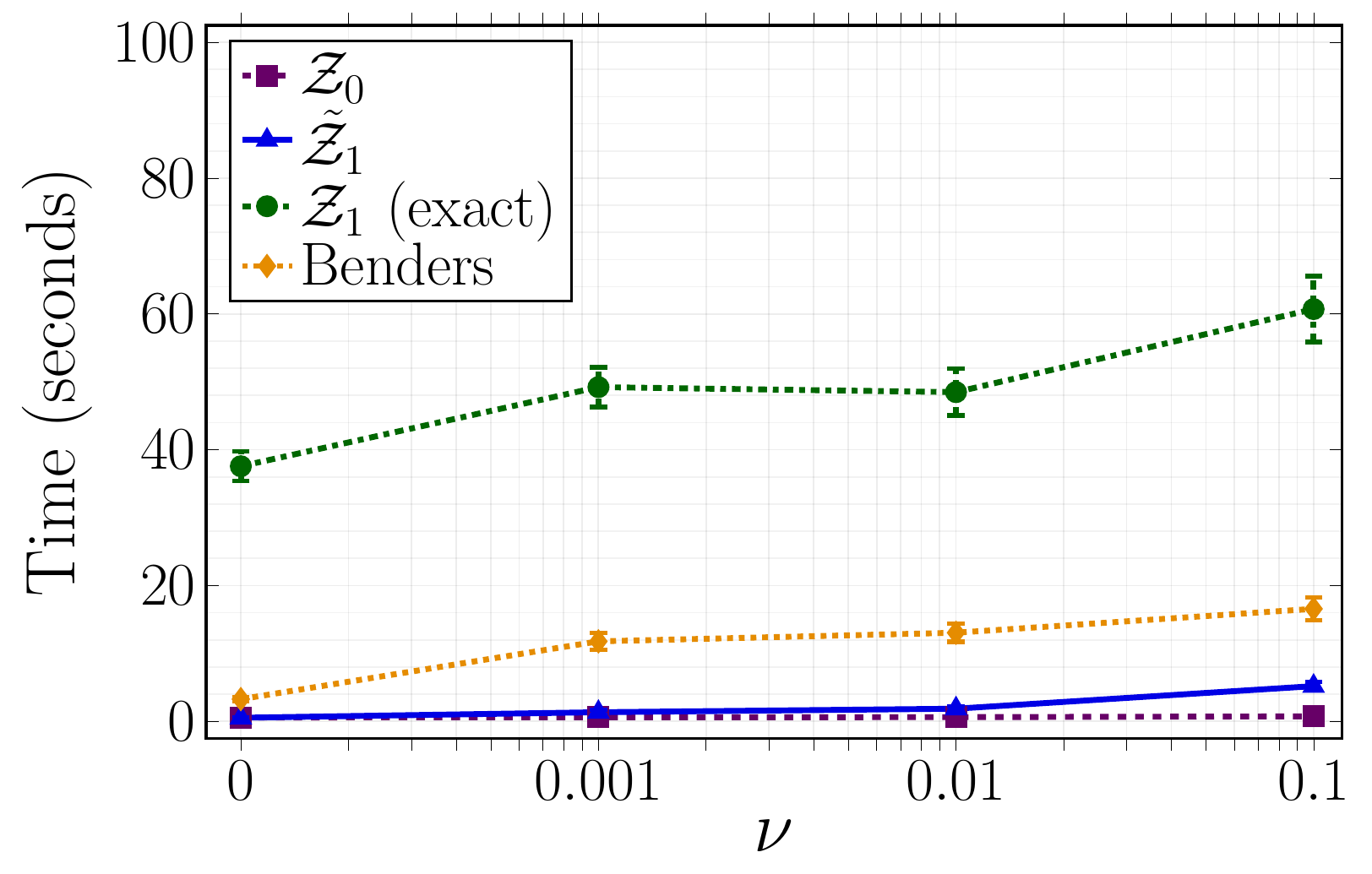}
        \caption{14-bus, $N = 1000$}\label{fig:time_14_N1000}
    \end{subfigure}
    \medskip
    \begin{subfigure}[b]{0.32\linewidth}
        \includegraphics[width=\linewidth]{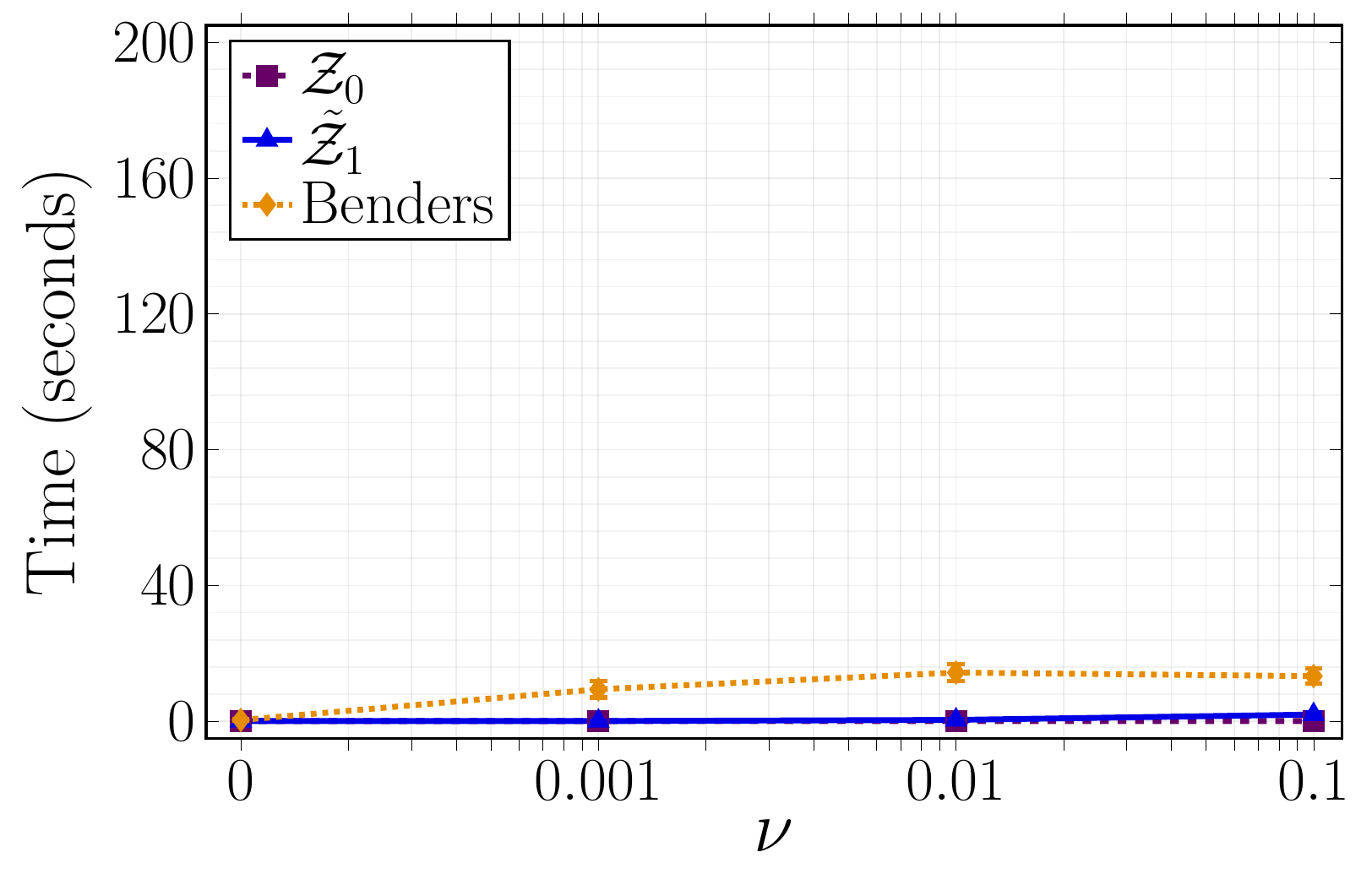}
        \caption{30-bus, $N = 10$}\label{fig:time_30_N10}
    \end{subfigure}\hfil
    \begin{subfigure}[b]{0.32\linewidth}
        \includegraphics[width=\linewidth]{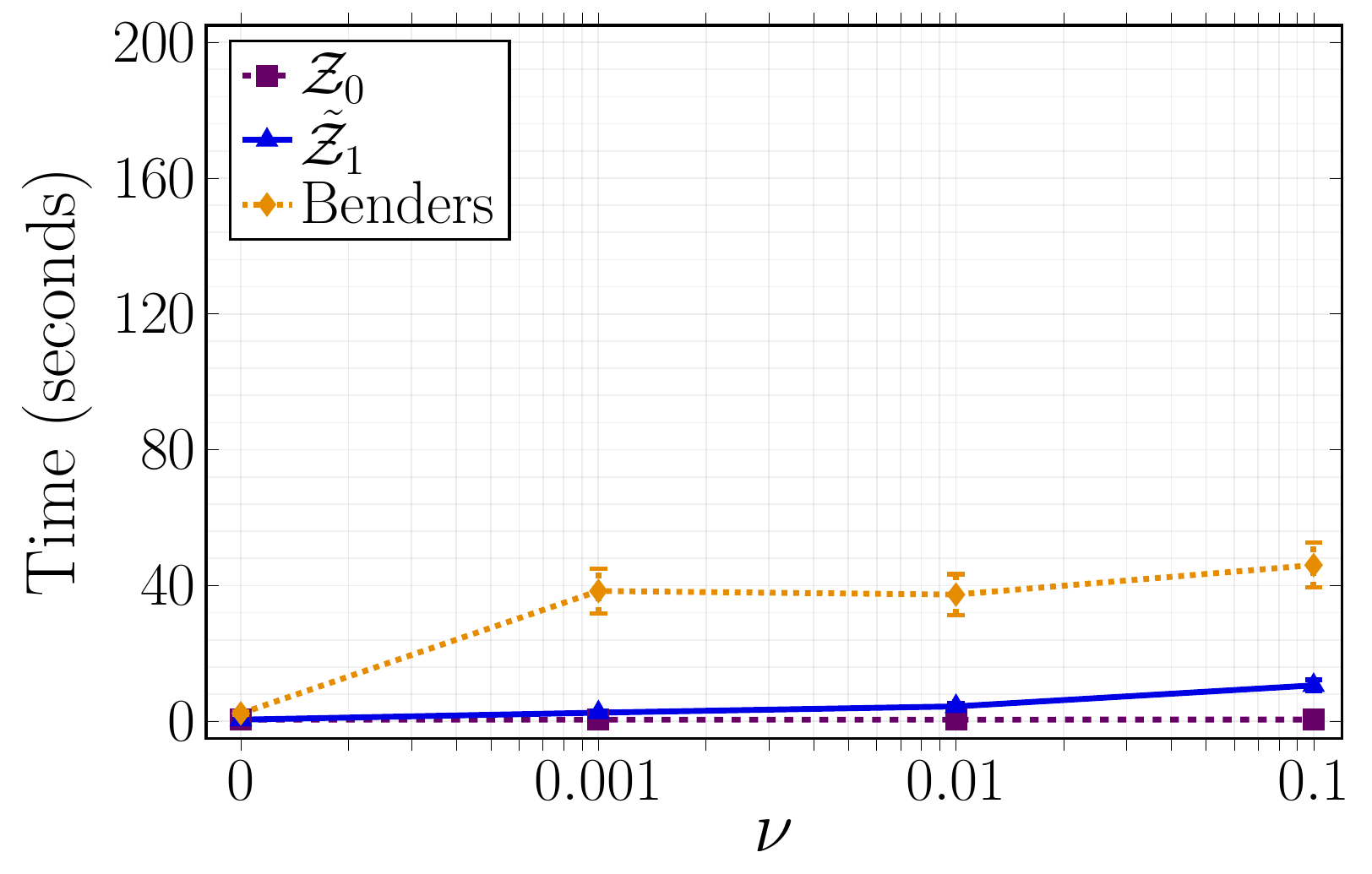}
        \caption{30-bus, $N = 100$}\label{fig:time_30_N100}
    \end{subfigure}\hfil
    \begin{subfigure}[b]{0.32\linewidth}
        \includegraphics[width=\linewidth]{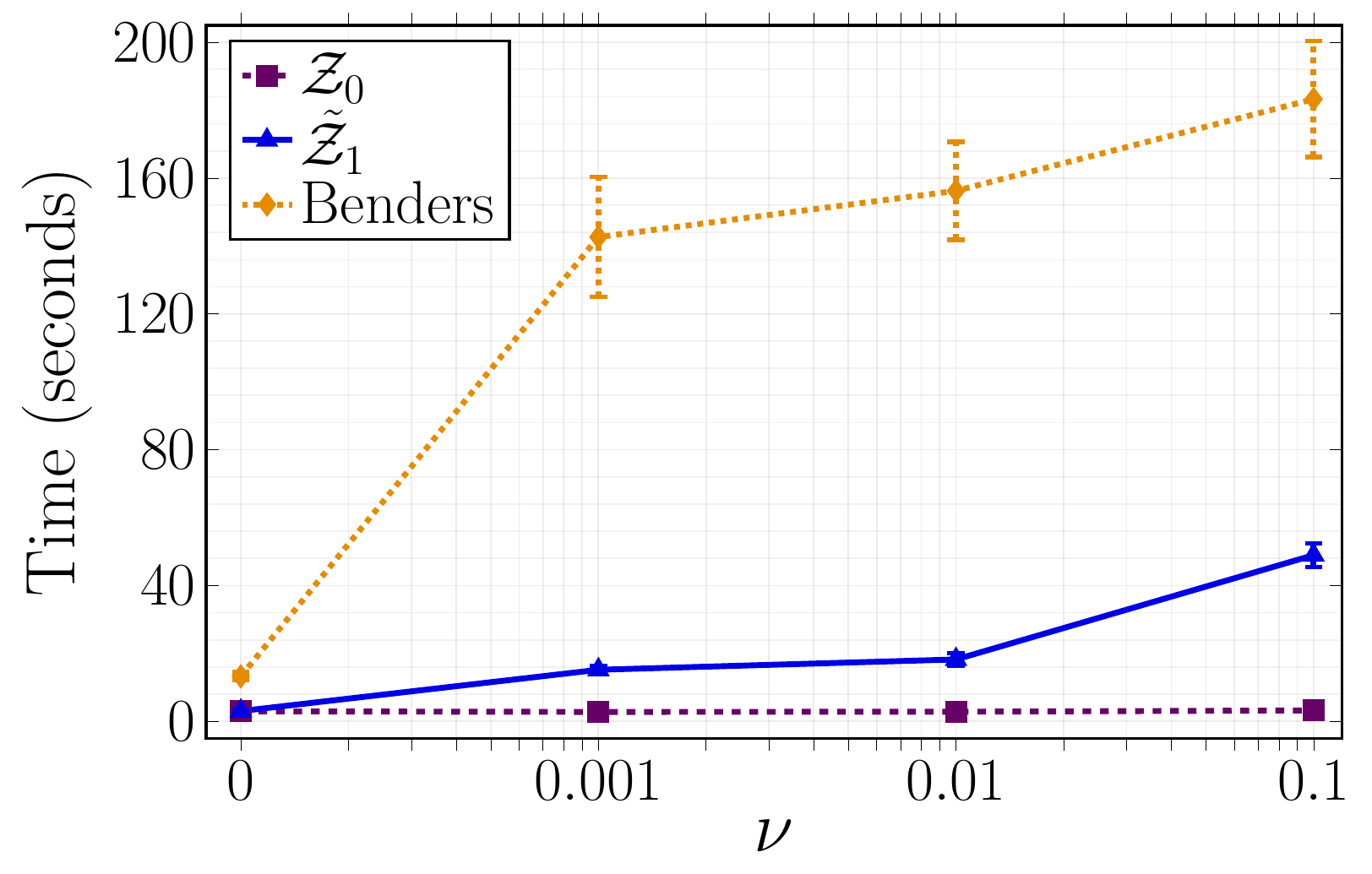}
        \caption{30-bus, $N = 1000$}\label{fig:time_30_N1000}
    \end{subfigure}
    \medskip
    \begin{subfigure}[b]{0.32\linewidth}
        \includegraphics[width=\linewidth]{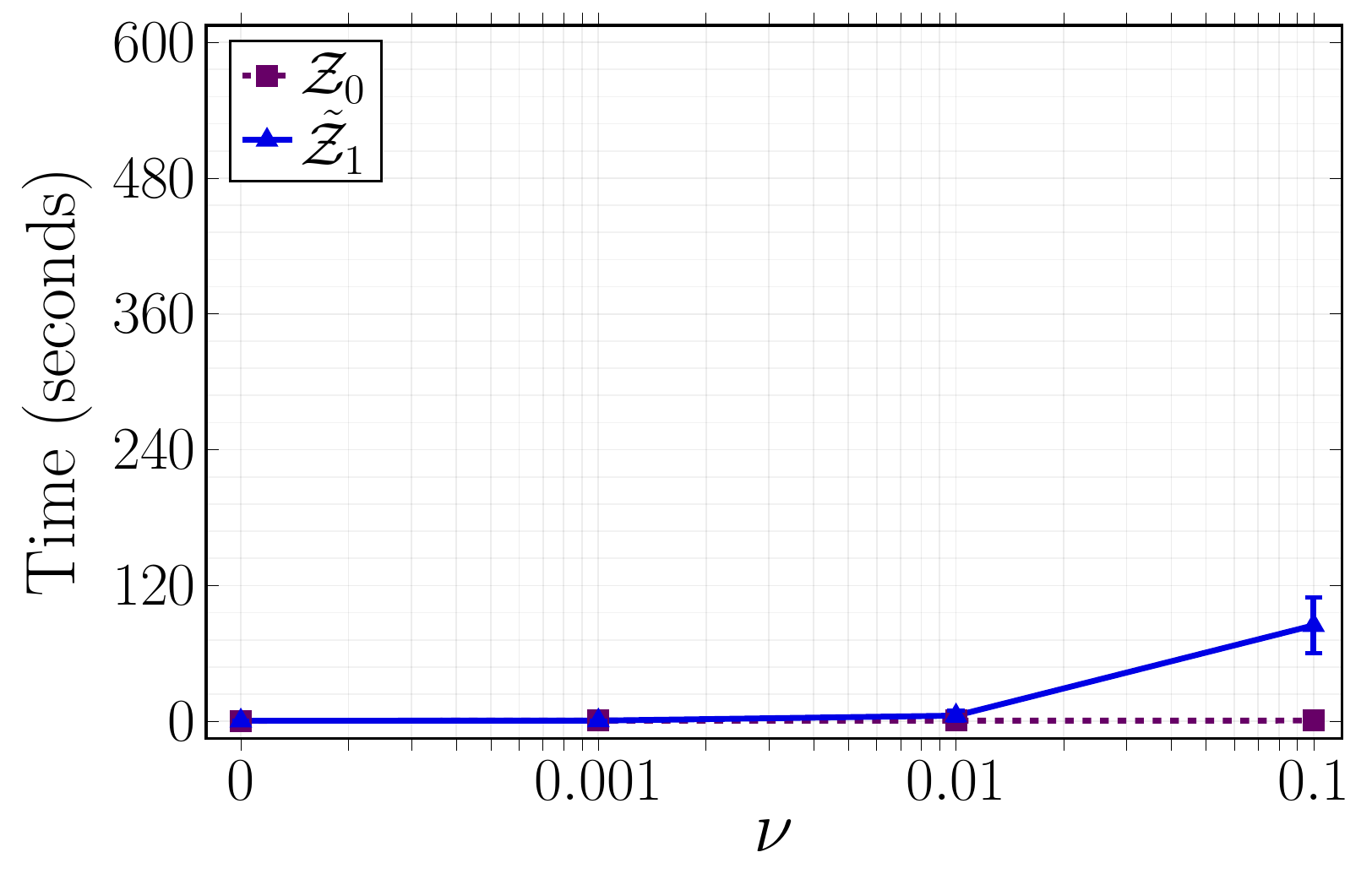}
        \caption{118-bus, $N = 10$}\label{fig:time_118_N10}
    \end{subfigure}\hfil
    \begin{subfigure}[b]{0.32\linewidth}
        \includegraphics[width=\linewidth]{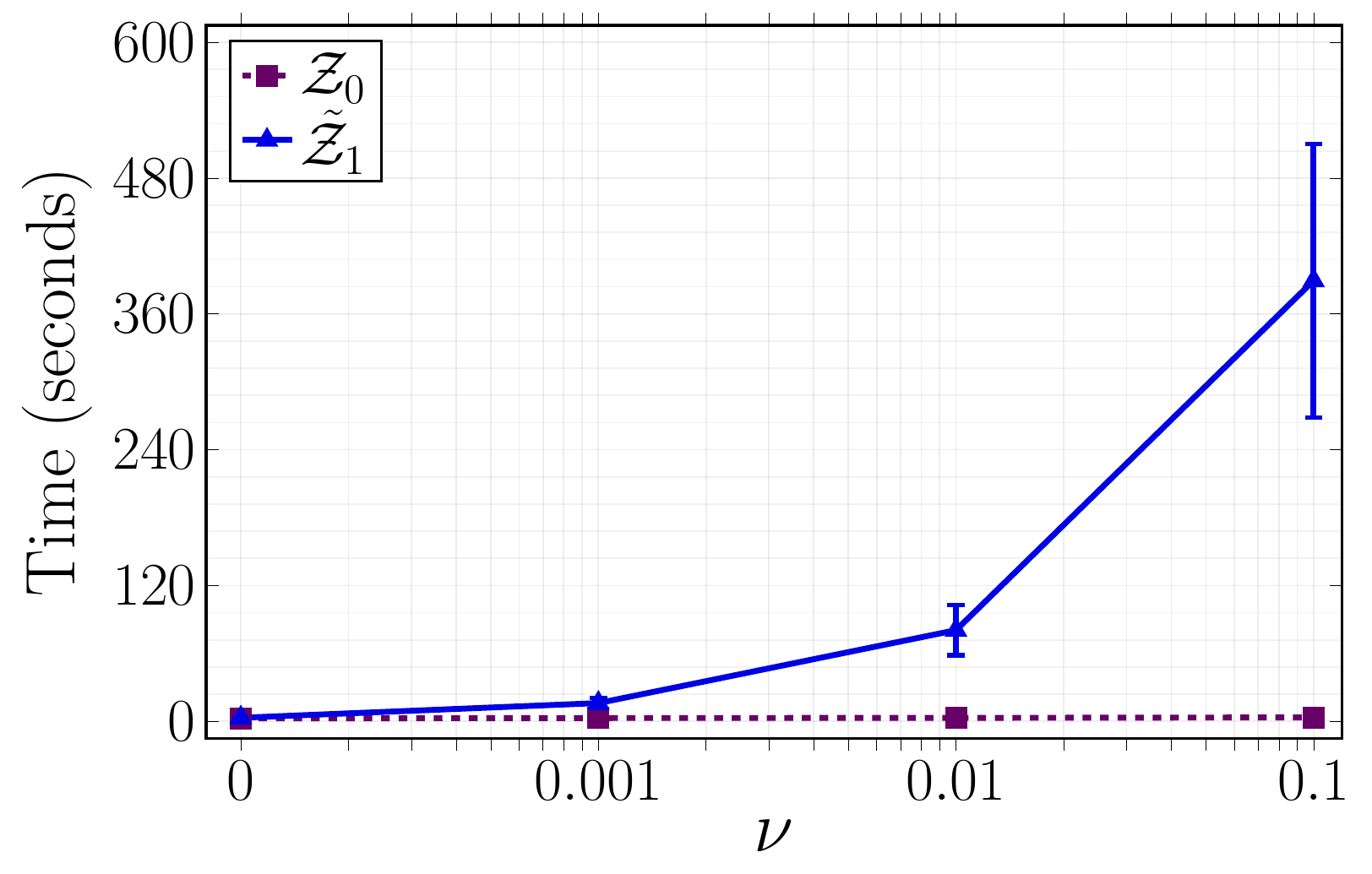}
        \caption{118-bus, $N = 100$}\label{fig:time_118_N100}
    \end{subfigure}\hfil
    \begin{subfigure}[b]{0.32\linewidth}
        \includegraphics[width=\linewidth]{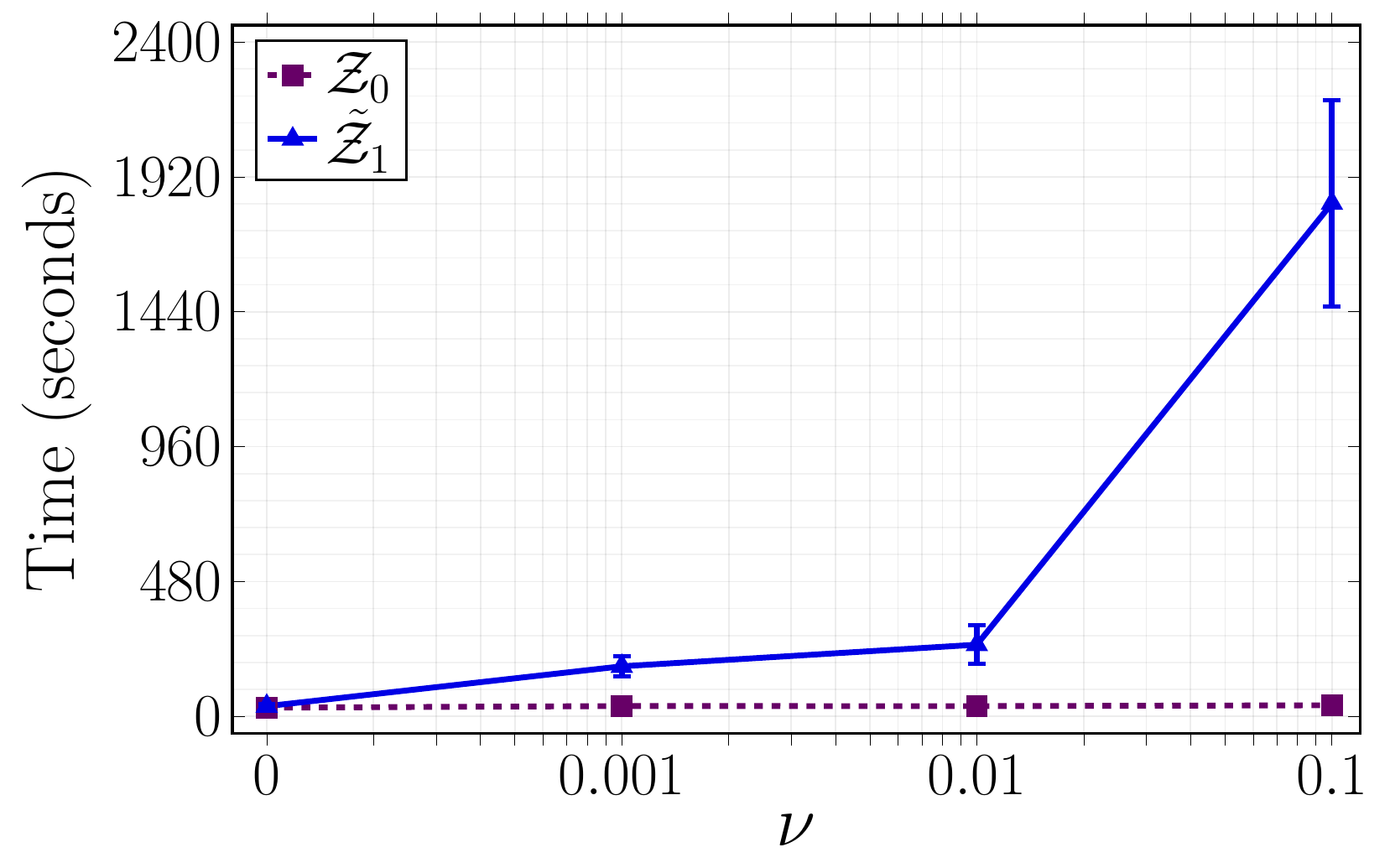}
        \caption{118-bus, $N = 1000$}\label{fig:time_118_N1000}
    \end{subfigure}
    \caption{Computation times for solving formulation~\eqref{eq:two_stage_dro_reform} using the continuous relaxation $\mathcal{Z}^0$, the heuristically computed and exact level-1 Lov{\'a}sz-Schrijver relaxations $\tilde{\mathcal{Z}}^1$ and ${\mathcal{Z}}^1$, and the Benders decomposition scheme, as a function of $\nu$ and $N$, where \re{$\varepsilon = \nu \sqrt{N^{-1}\log(N+1)}$}.\label{fig:time}}
\end{figure}

\subsubsection{Out-of-sample performance \re{and finite sample guarantee}}\label{sec:results_out_of_sample}

To understand the potential benefits of a distributionally robust approach, %
we evaluate its out-of-sample performance.
For a given training dataset of size $N$ and a given choice of $\nu$, we obtain a candidate first-stage solution $\bm{x}^\nu$ by solving formulation~\eqref{eq:two_stage_dro_reform} with the level-1 Lov{\'a}sz-Schrijver relaxation $\tilde{\mathcal{Z}}^1$ and Wasserstein radius \re{$\varepsilon = \nu \sqrt{N^{-1}\log(N+1)}$}.
We then estimate the out-of-sample performance of $\bm{x}^\nu$ by recording $z^\nu = \bm{c}(\bm{x}^\nu) + 1000^{-1} \sum_{i=1}^{1000} \mathcal{Q}(\bm{x}^\nu,\hat{\bm{\xi}}^{(i)})$, where $\hat{\bm{\xi}}^{(1)}, \ldots, \hat{\bm{\xi}}^{(1000)}$ are $1,000$ independently generated testing samples.
This entire process is repeated 100 times for statistically independent sets of $N$ training samples and 1,000 testing samples.

\re{We first justify the choice of the radius $\varepsilon = \nu \sqrt{N^{-1}\log(N+1)}$ as a function of the training sample size $N$.
This dependence is motivated by inequality~\eqref{eq:epsilon_finite_sample} in Theorem~\ref{thm:finite_sample_guarantee}.
However, the latter inequality can be loose, especially when accounting for problem-dependent constants such as the size and diameter of the support $\Xi$.
Therefore, we empirically verify the finite sample guarantee of the first-stage solutions $\bm{x}^\nu$ under the tighter parameterization $\varepsilon = \nu \sqrt{N^{-1}\log(N+1)}$ for several choices of the coefficient $\nu \in \{0, 10^{-3}, 10^{-2}, 10^{-1}\}$.
Figure~\ref{fig:reliability} reports the \emph{reliability} of $\bm{x}^\nu$, which we define as the empirical probability (over the 100 sets of training samples) that the optimal value $\tilde{v}^1$ of the level-1 Lov{\'a}sz-Schrijver relaxation $\tilde{\mathcal{Z}}^1$ is an upper bound on the out-of-sample cost $z^\nu$.
}

\re{Figure~\ref{fig:reliability} shows that, for fixed values of $N$, the reliability of $\bm{x}^\nu$ increases with increasing values of $\nu$, and this can be used to guide the choice of $\nu$.
For example, depending on their risk level, decision-makers can select the smallest value of $\nu$ with sufficiently high reliability.
In particular, for training datasets with small sample size $N$, observe that $\nu = 10^{-2}$ yields an upper bound on the out-of-sample cost with probability more than 0.5, whereas $\nu = 10^{-1}$ yields an upper bound with probability 1.0.
However, note that we cannot always access the true out-of-sample cost (and hence, the true reliability).
Nevertheless, one could estimate the out-of-sample cost by using cross validation or the holdout method (e.g, see~\cite{esfahani2018data,jiang2019data}), and then select a value for the coefficient $\nu$ that offers a good trade-off between low out-of-sample cost and high reliability.
}

\begin{figure}[!htbp]
    \centering
    \begin{subfigure}[b]{0.32\linewidth}
        \includegraphics[width=\linewidth]{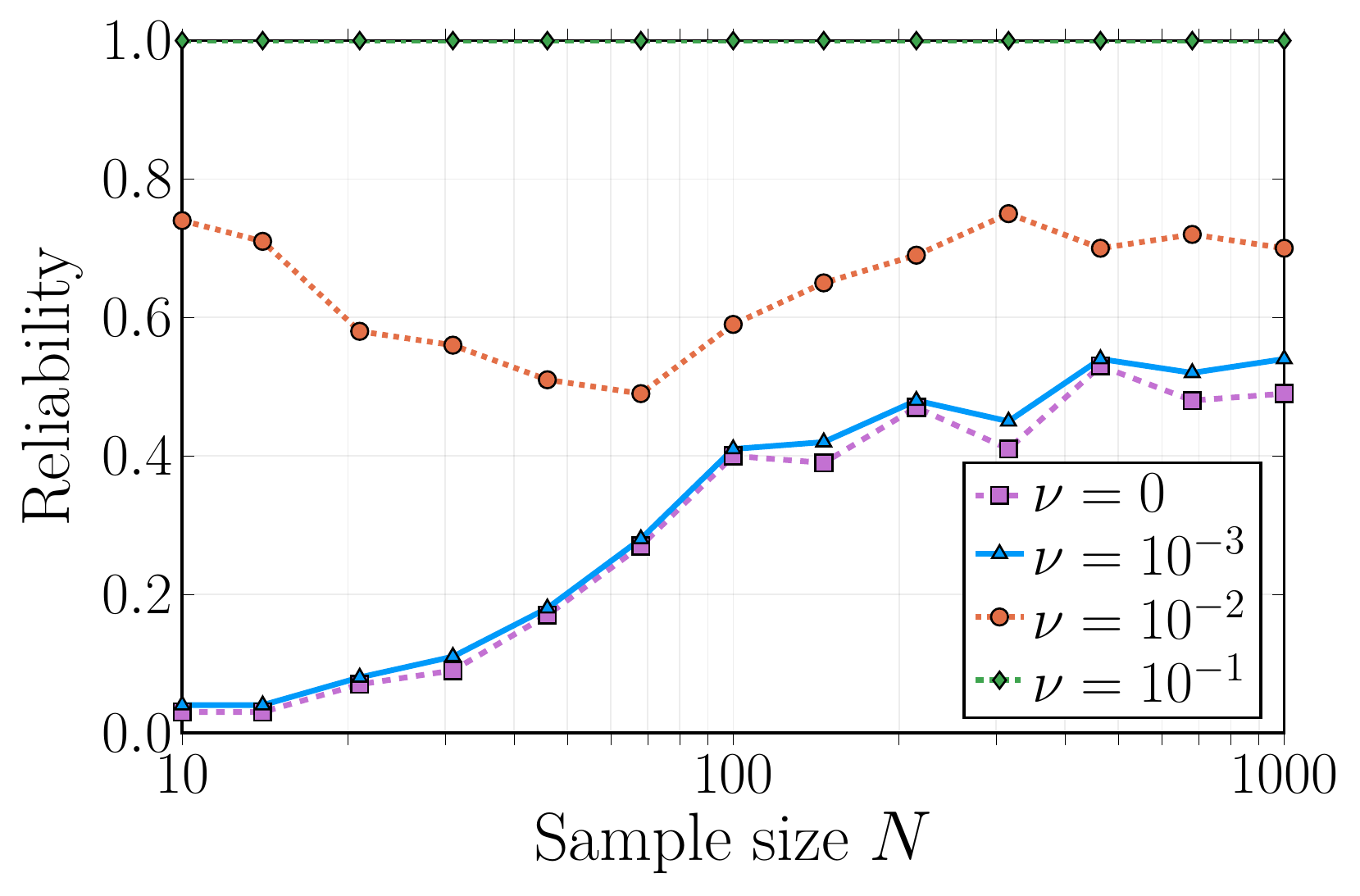}
        \caption{14-bus}\label{fig:reliability_14}
    \end{subfigure}\hfil
    \begin{subfigure}[b]{0.32\linewidth}
        \includegraphics[width=\linewidth]{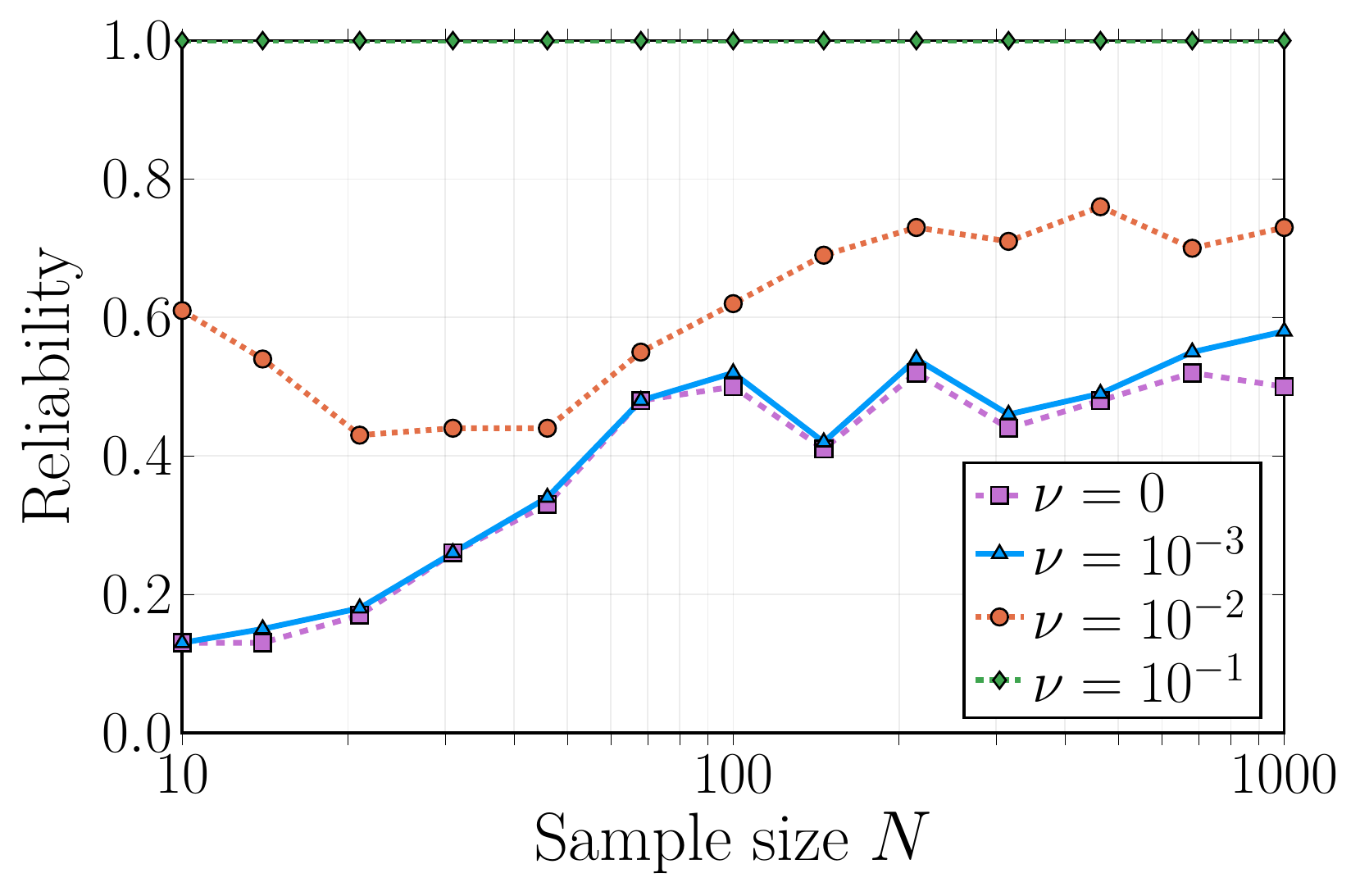}
        \caption{30-bus}\label{fig:reliability_30}
    \end{subfigure}\hfil
    \begin{subfigure}[b]{0.32\linewidth}
        \includegraphics[width=\linewidth]{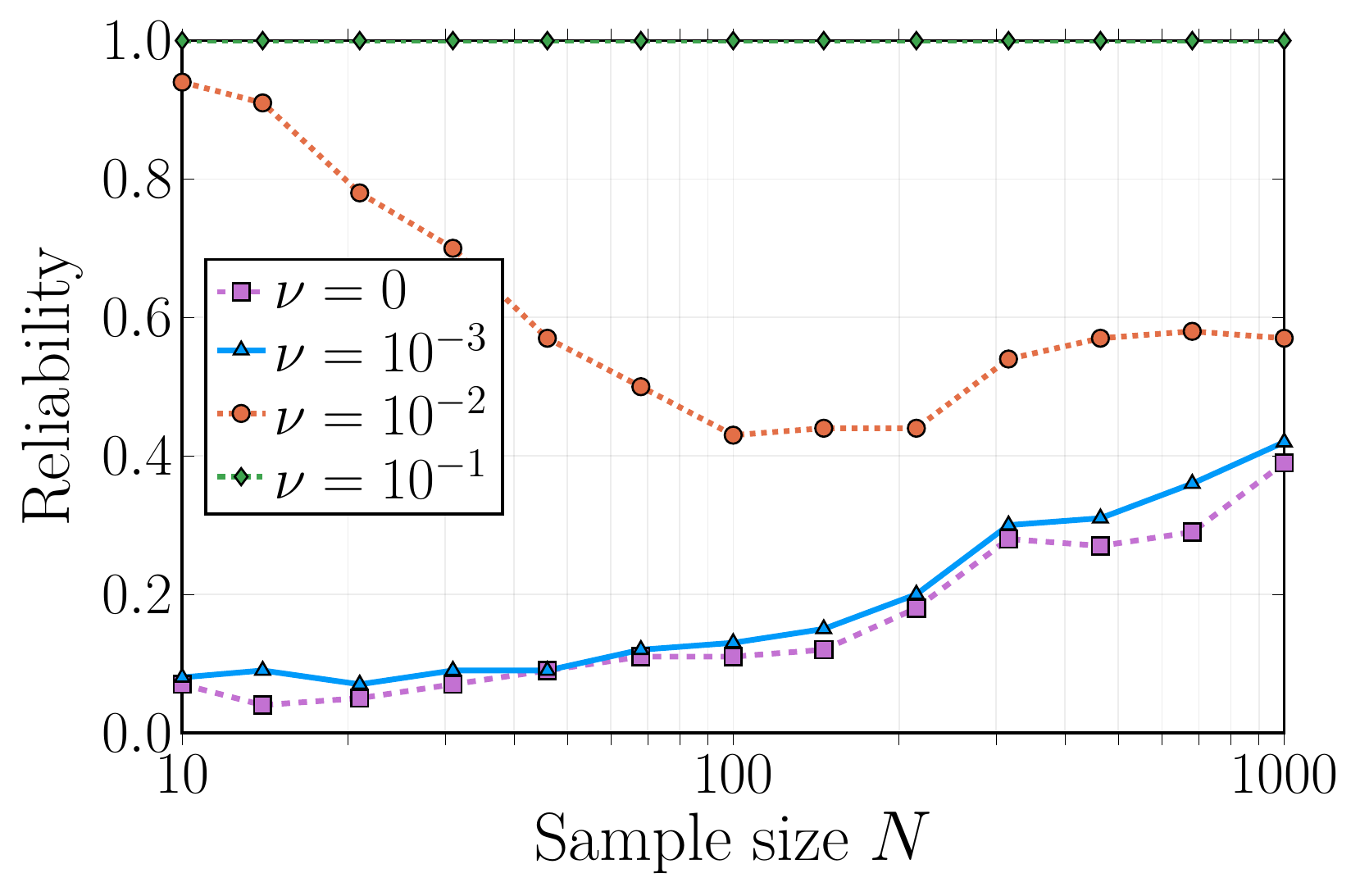}
        \caption{118-bus}\label{fig:reliability_118}
    \end{subfigure}
    \caption{\re{Reliability of the level-1 Lov{\'a}sz-Schrijver relaxation $\tilde{\mathcal{Z}}^1$, as a function of training sample size $N$.}\label{fig:reliability}}
\end{figure}

We now evaluate the benefits of our distributionally robust model over the sample average approximation, by computing the relative improvement in out-of-sample cost, which we define as $(z^0 - z^\nu)/z^0 \times 100\%$.
Figure~\ref{fig:outOfSample} reports the mean (solid line) and standard deviation (shaded ribbon) of the relative improvement over the 100 independent sets of training samples.
We make the following observations from Figure~\ref{fig:outOfSample}.
\begin{itemize}
    \item The distributionally robust model~\eqref{eq:two_stage_dro} consistently outperforms the sample average approximation, particularly for small sample sizes $N$. The magnitude of the relative improvement is instance dependent (roughly 15\%, 10\%, and 5\% for the 14-, 30-, and 118-bus cases, respectively) but consistently decreases for large values of $N$ as expected. The magnitude of the radius that leads to the best  possible improvement also is instance dependent. %
    
    \item The larger variances in improvement for smaller $N$ \re{and for larger instances} can be partially explained by the combinatorial growth in the number of truly distinct training datasets of size $N$ (i.e., those that lead to distinct first-stage solutions) that are possible under the rare event model of line outages.
    \re{The large variances for $\nu = 10^{-1}$ can also be similarly explained by the larger number of truly distinct first-stage solutions that can result from slight variations in the training dataset.}
\end{itemize}

\begin{figure}[!htbp]
    \centering
    \begin{subfigure}[b]{0.32\linewidth}
        \includegraphics[width=\linewidth]{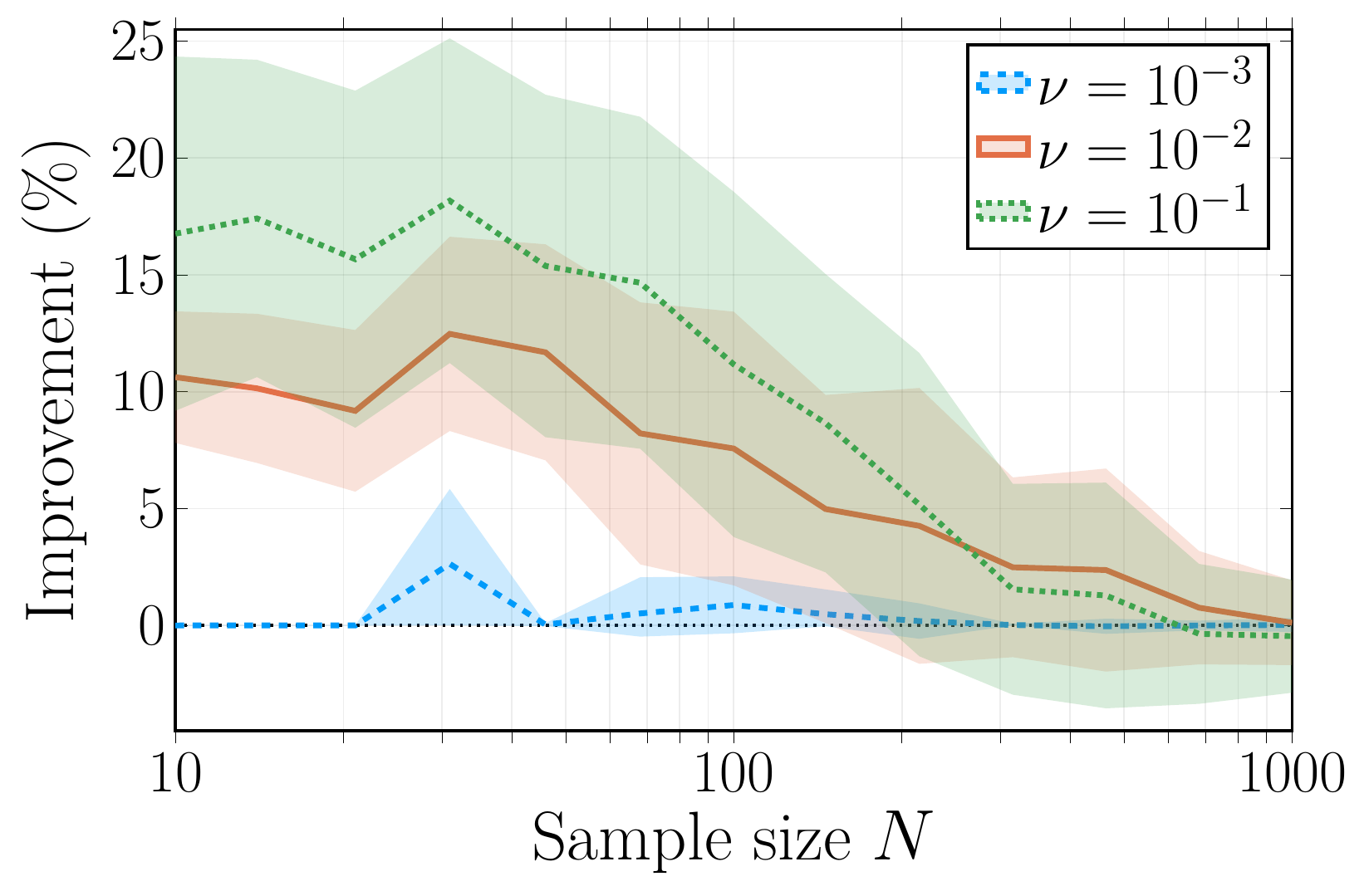}
        \caption{14-bus}\label{fig:outOfSample_14}
    \end{subfigure}\hfil
    \begin{subfigure}[b]{0.32\linewidth}
        \includegraphics[width=\linewidth]{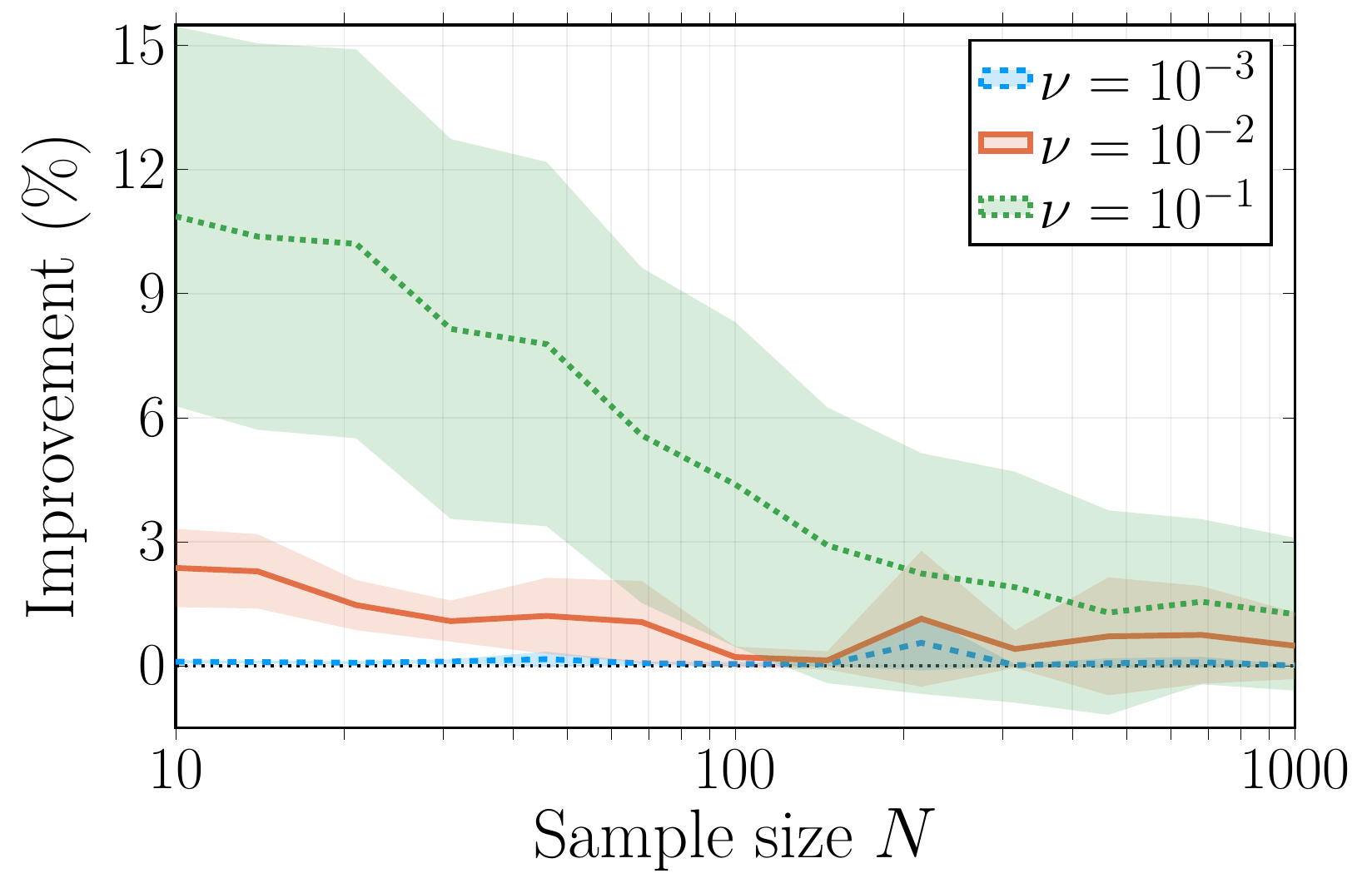}
        \caption{30-bus}\label{fig:outOfSample_30}
    \end{subfigure}\hfil
    \begin{subfigure}[b]{0.32\linewidth}
        \includegraphics[width=\linewidth]{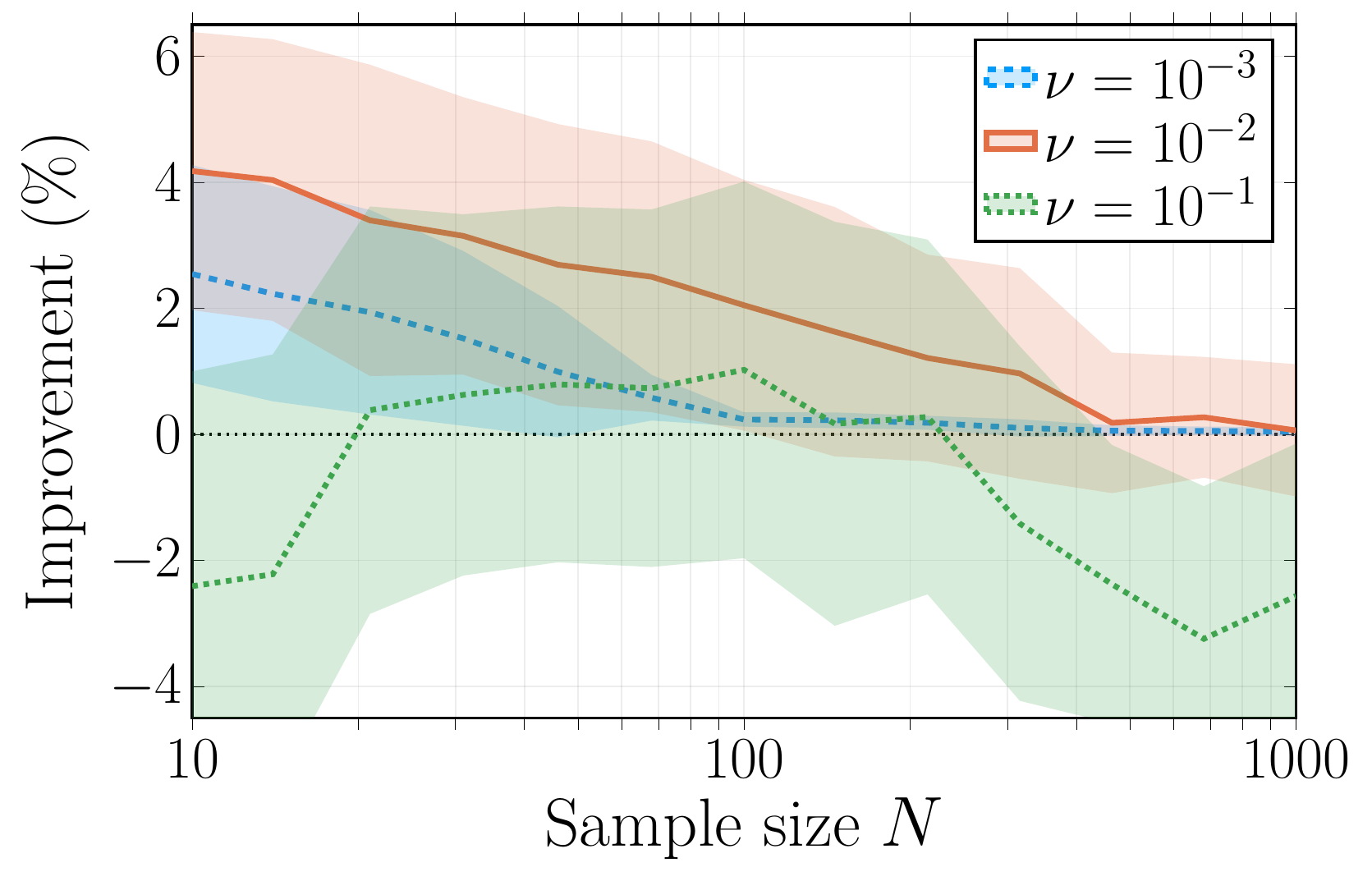}
        \caption{118-bus}\label{fig:outOfSample_118}
    \end{subfigure}
    \caption{Relative improvement in the out-of-sample performance of the distributionally robust two-stage model~\eqref{eq:two_stage_dro} when compared with the sample average approximation, as a function of sample size $N$.\label{fig:outOfSample}}
\end{figure}

\re{We now compare the out-of-sample performance of the level-1 Lov{\'a}sz-Schrijver relaxation $\tilde{\mathcal{Z}}^1$ corresponding to the best choice of $\nu$ (which is $10^{-1}$ for the 14- and 30-bus cases, and $10^{-2}$ for the 118-bus case) with \textit{(i)} the continuous relaxation $\mathcal{Z}^0$ for the same $\nu$, and \textit{(ii)} a classical two-stage robust optimization model of the following form:
\begin{equation}\label{eq:two_stage_ro}
\mathop{\text{minimize}}_{\mb{x} \in \mathcal{X}} \;
c(\mb{x})
+
\max_{\bm{\xi} \in \Xi_K}
\mathcal{Q}(\mb{x}, \bm{\xi}),
\end{equation}
where $\Xi_K = \{\bm{\xi} \in \{0, 1\}^M: \xi_1 + \ldots + \xi_M \leq K\}$ is the uncertainty set with $K$ being the \emph{budget} of uncertainty.
In particular, for each instance, we let $K \in \{0, 1, 2, 5, 10\}$ and obtain optimal first-stage solutions $\bm{x}^K$ of problem~\eqref{eq:two_stage_ro} with the Benders decomposition scheme.
As before, we estimate the out-of-sample performance of $\bm{x}^K$ as $z^K = \bm{c}(\bm{x}^K) + 1000^{-1} \sum_{i=1}^{1000} \mathcal{Q}(\bm{x}^K,\hat{\bm{\xi}}^{(i)})$, where $\hat{\bm{\xi}}^{(1)}, \ldots, \hat{\bm{\xi}}^{(1000)}$ are the same $1,000$ testing samples used to estimate $z^\nu$.
We then record the best possible $K$ yielding the lowest $z^K$, and compute the relative improvement in out-of-sample cost, which we define as $(z^K - z^\nu)/z^\nu \times 100\%$.
Figure~\ref{fig:improvement} reports the mean (solid line) and standard deviation (shaded ribbon) of the relative improvement over the 100 independent sets of training samples.
We make the following observations from Figure~\ref{fig:improvement}.
\begin{itemize}
    \item The distributionally robust model strongly outperforms its classical robust counterpart, across all instances, with relative improvements of 5\%, 1\% and 15\% for the 14-, 30- and 118-bus cases, respectively.
    In contrast to Figure~\ref{fig:outOfSample}, the relative improvements increase with increasing values of $N$; this is expected since the classical robust model~\eqref{eq:two_stage_ro} ignores all sample data and therefore, it becomes overly conservative in the presence of a moderate amount of data.
    Thus, we observe that for small to moderate values of $N$, the distributionally robust model improves upon both the sample average approximation and classical robust optimization.
    Finally, although not shown, solving the classical robust model with the Benders scheme took longer than solving the level-1 relaxation, especially for the larger 118-bus case.
    
    \item The relative improvement of the level-1 relaxation $\tilde{\mathcal{Z}}^1$ over the continuous relaxation $\mathcal{Z}^0$ is smaller, but can be as high as 4\% as seen in the 14-bus case.
    It should be noted that these quantities are necessarily upper bounded by the improvements over the sample average approximations reported in Figure~\ref{fig:outOfSample}, and can be significant for applications including optimal power flow, which are executed several times each day of the year.
    Finally, we note that the tradeoff between tighter in-sample optimality gaps (see Figure~\ref{fig:gaps}) and hence stronger finite sample reliability guarantees (see Figure~\ref{fig:reliability}) offered by the level-1 relaxation $\tilde{\mathcal{Z}}^1$, with the faster computation times for solving the continuous relaxation $\mathcal{Z}^0$ (see Figure~\ref{fig:time}), can guide the design of an algorithmic scheme.
\end{itemize}
}

\begin{figure}[!htbp]
    \centering
    \begin{subfigure}[b]{0.32\linewidth}
        \includegraphics[width=\linewidth]{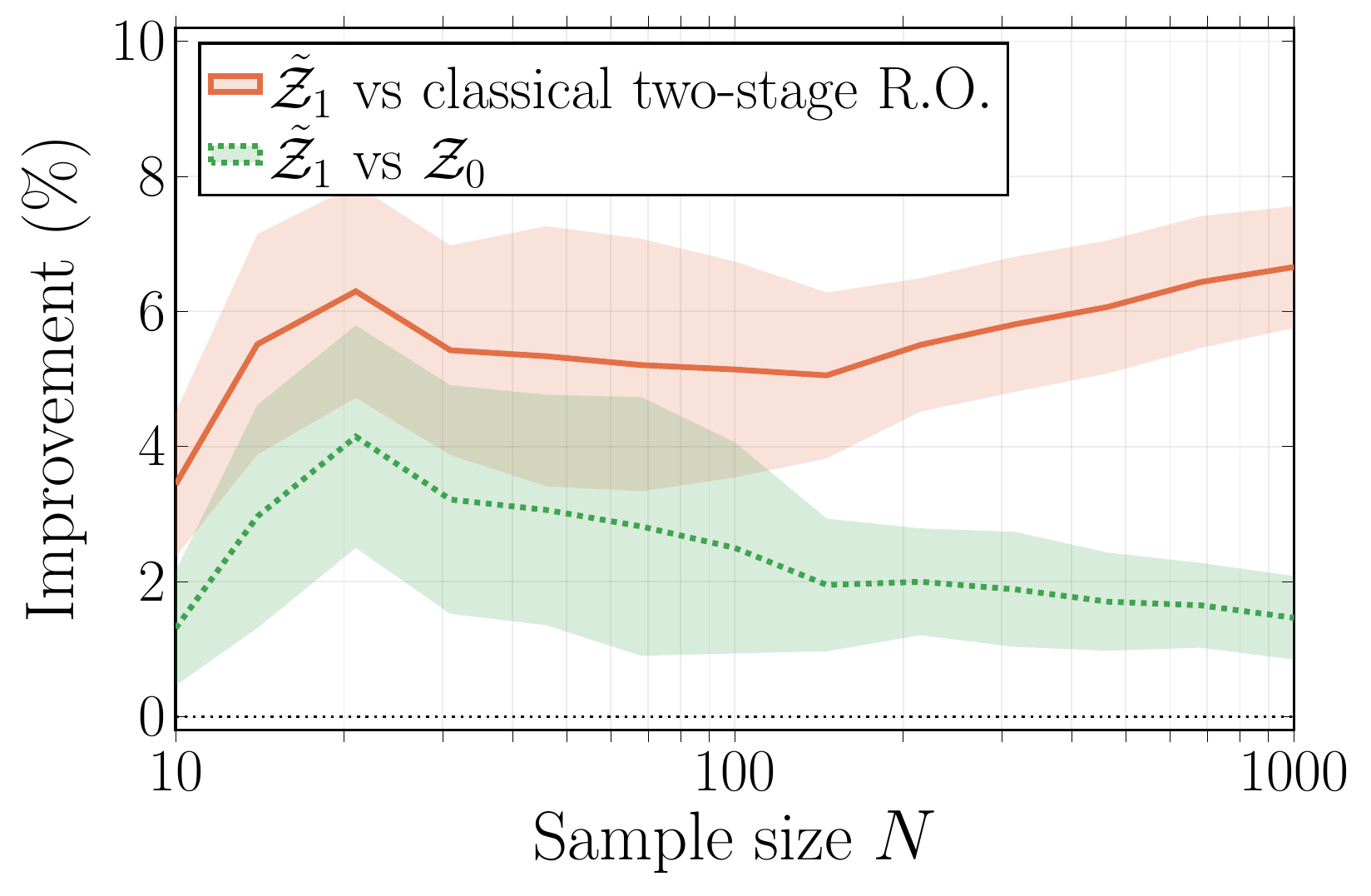}
        \caption{14-bus}\label{fig:improvement_14}
    \end{subfigure}\hfil
    \begin{subfigure}[b]{0.32\linewidth}
        \includegraphics[width=\linewidth]{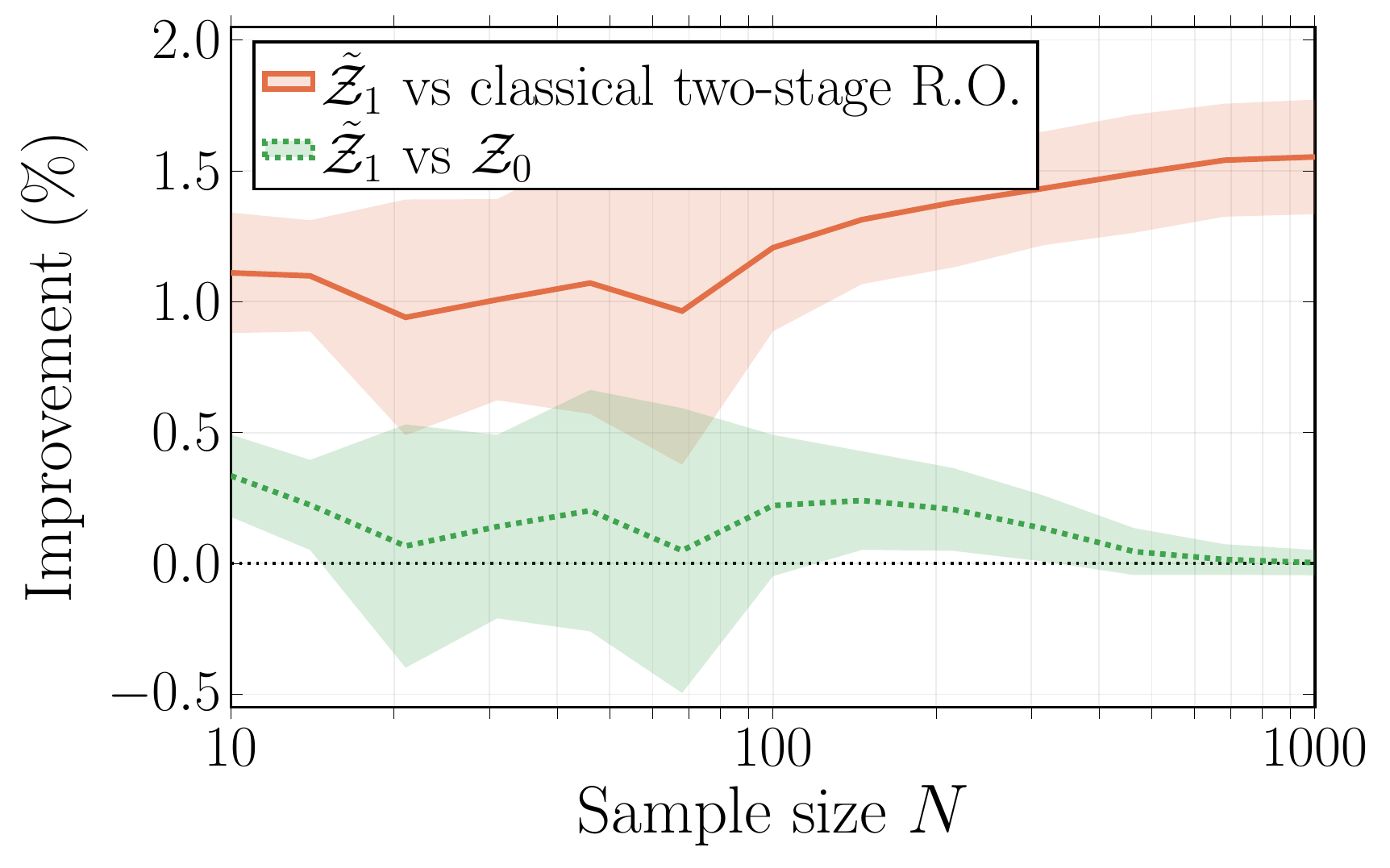}
        \caption{30-bus}\label{fig:improvement_30}
    \end{subfigure}\hfil
    \begin{subfigure}[b]{0.32\linewidth}
        \includegraphics[width=\linewidth]{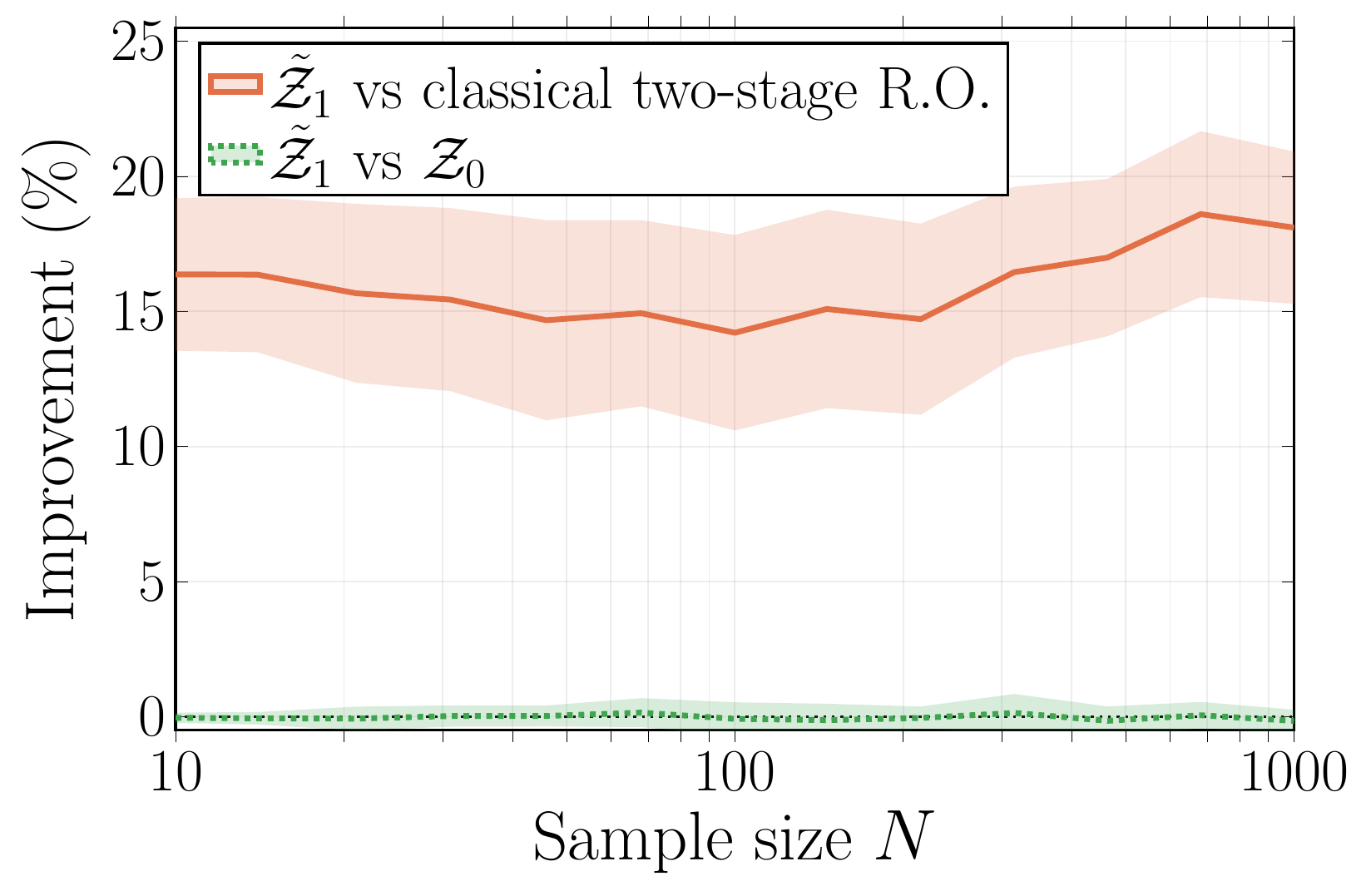}
        \caption{118-bus}\label{fig:improvement_118}
    \end{subfigure}
    \caption{\re{Relative improvement in the out-of-sample performance of the distributionally robust two-stage model~\eqref{eq:two_stage_dro} solved using the level-1 Lov{\'a}sz-Schrijver relaxation $\tilde{\mathcal{Z}}^1$, when compared with the classical two-stage robust optimization model~\eqref{eq:two_stage_ro}, and the continuous relaxation $\mathcal{Z}^0$.}\label{fig:improvement}} %
\end{figure}

\subsubsection{Sensitivity analysis}\label{sec:results_sensitivity_analysis}
The instance-dependent behavior of the out-of-sample performance from the previous subsection suggests that it might also be influenced by other parameters.
Here, we investigate the effect of the ``rareness'' $\psi$ of transmission line failures and the relative magnitude $\phi$ of ``impact'' when failures occur.
Recall from Section~\ref{sec:test_instances} that higher values of $\psi$ increase the probability of line failures, whereas higher values of $\phi$ increase the penalty cost for failing to satisfy power demand due to transmission line failures.
Figure~\ref{fig:sensitivity} shows the relative improvement of the distributionally robust two-stage model~\eqref{eq:two_stage_dro} over the sample average approximation for various choices of $\psi$ and $\phi$.
For brevity, we report results only for the 30-bus instance; the high-level insights do not change for other instances.
We make the following observations from Figure~\ref{fig:sensitivity}.
\begin{itemize}
    \item For fixed values of the line failure probability $\psi$, Figures~\ref{fig:sensitivity_phi_low}--\ref{fig:sensitivity_phi_high} show that as the impact due to failure $\phi$ increases, the relative benefits of a distributionally robust approach strongly increase. In other words, benefits increase with higher impacts of failures. \re{Interestingly, Figure~\ref{fig:sensitivity_phi_low} also shows that if failures are rare but low impact, then ignoring them (as in the sample average approximation) may not incur high out-of-sample costs, even for small values of $N$.}
    
    \item For fixed values of the magnitude of impact $\phi$, Figures~\ref{fig:sensitivity_psi_low}--\ref{fig:sensitivity_psi_high} show that as the probability of failures $\psi$ increases, the relative benefits of a distributionally robust approach increases. \re{However, observe that this does not necessarily imply that the relative benefits are small when line failure probabilities are small. Indeed, %
    we observe that the relative benefits remain as high as 10\% even when individual line failure probabilities are less than $0.05M^{-1}$.}
    
\end{itemize}

\begin{figure}[!htb]
    \centering
    \begin{subfigure}[b]{0.32\linewidth}
        \includegraphics[width=\linewidth]{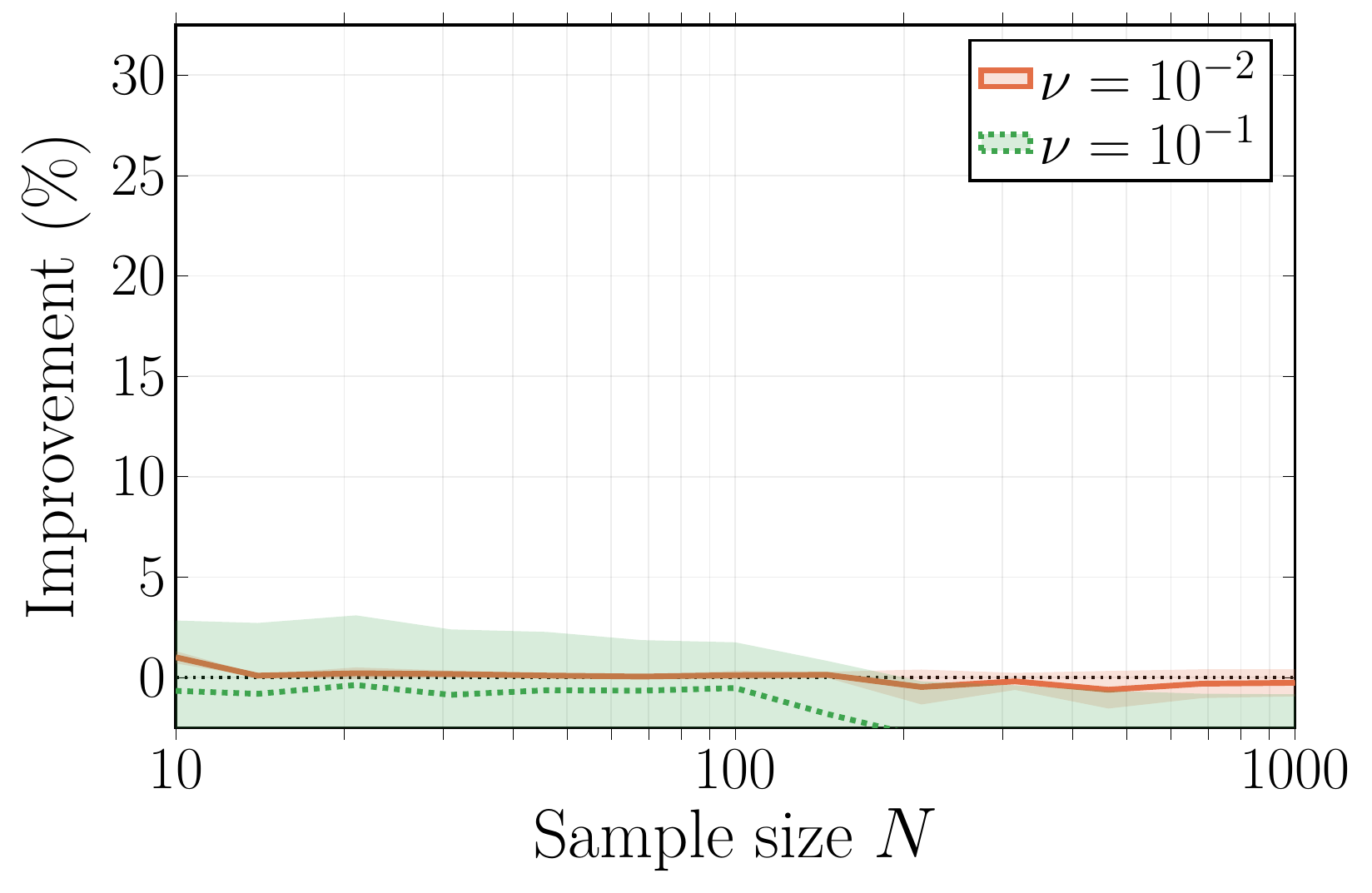} %
        \caption{$(\psi, \phi) = (0.1, 50)$}\label{fig:sensitivity_phi_low}
    \end{subfigure}\hfil
    \begin{subfigure}[b]{0.32\linewidth}
        \includegraphics[width=\linewidth]{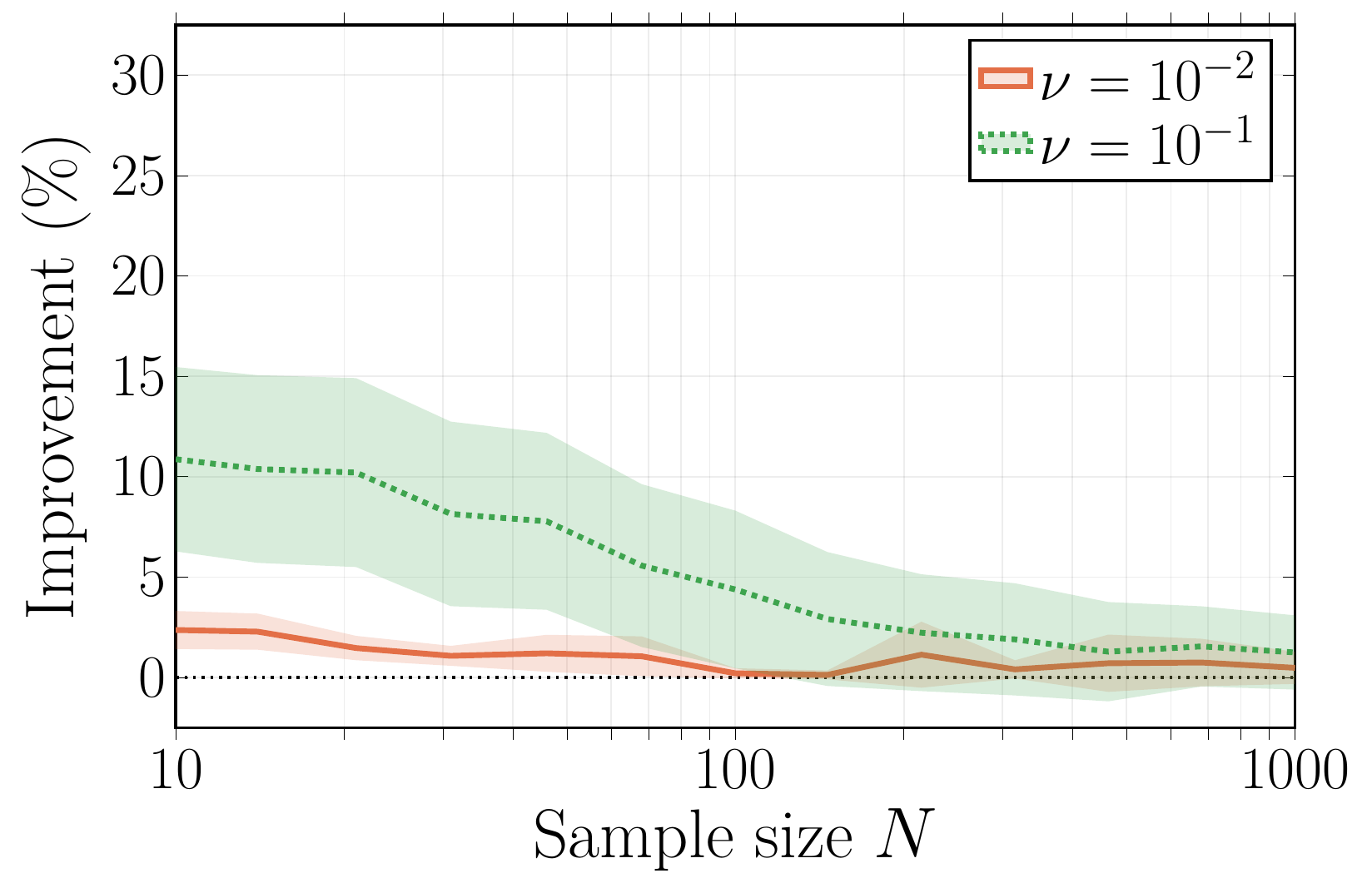} %
        \caption{$(\psi, \phi) = (0.1, 100)$}\label{fig:sensitivity_phi_med}
    \end{subfigure}\hfil
    \begin{subfigure}[b]{0.32\linewidth}
        \includegraphics[width=\linewidth]{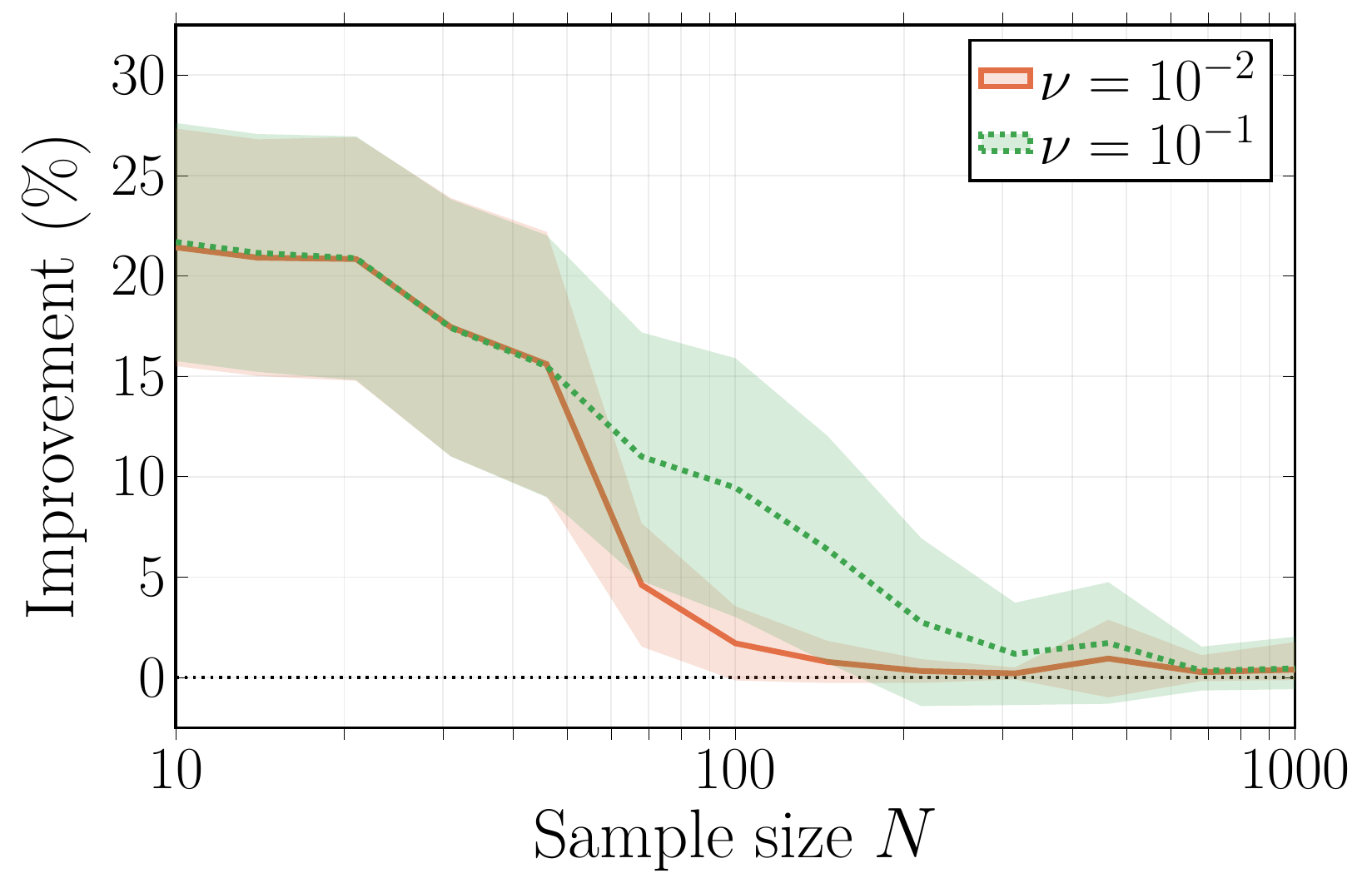} %
        \caption{$(\psi, \phi) = (0.1, 200)$}\label{fig:sensitivity_phi_high}
    \end{subfigure}
    \medskip
    \begin{subfigure}[b]{0.32\linewidth}
        \includegraphics[width=\linewidth]{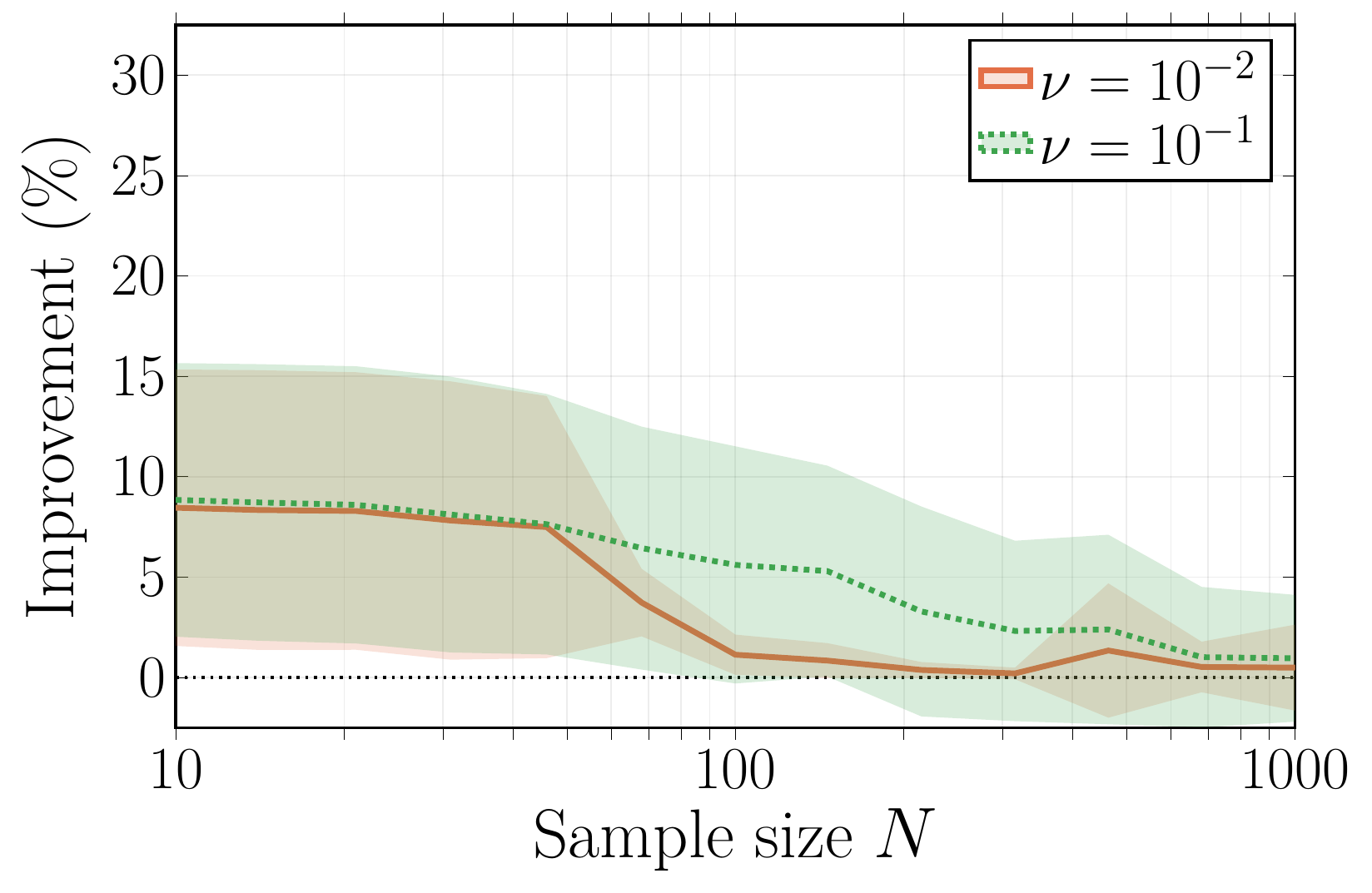} %
        \caption{\re{$(\psi, \phi) = (0.05,200)$}}\label{fig:sensitivity_psi_low}
    \end{subfigure}\hfil
    \begin{subfigure}[b]{0.32\linewidth}
        \includegraphics[width=\linewidth]{{pglib_opf_case30_ieee-phi200.0-psi0.1-outOfSample-sensitivity}} %
        \caption{\re{$(\psi, \phi) = (0.1, 200)$}}\label{fig:sensitivity_psi_med}
    \end{subfigure}\hfil
    \begin{subfigure}[b]{0.32\linewidth}
        \includegraphics[width=\linewidth]{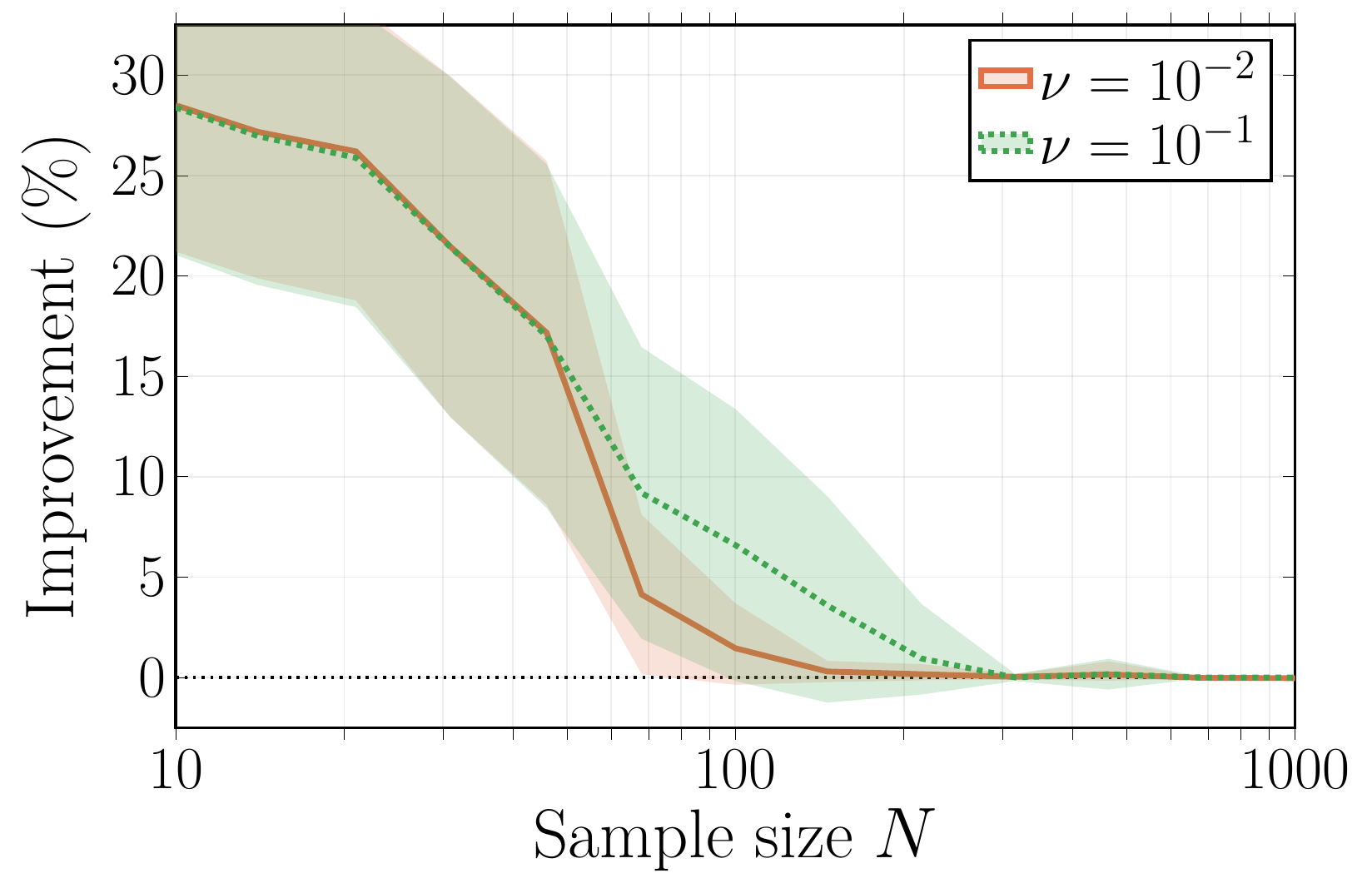} %
        \caption{\re{$(\psi, \phi) = (0.2, 200)$}}\label{fig:sensitivity_psi_high}
    \end{subfigure}
    \caption{Relative improvement in the out-of-sample performance of the distributionally robust two-stage model~\eqref{eq:two_stage_dro} when compared with the sample average approximation, for various values of $(\psi, \phi)$.\label{fig:sensitivity}}
\end{figure}

\re{
\subsection{Multi-commodity network design}\label{sec:MMCF}

We now consider multi-commodity network design problems that have applications in telecommunications, transportation, logistics and production planning, among others (e.g., see~\citet{minoux1989networks,crainic2001bundle}).
In several of these applications, it is required to send flows to satisfy known demands between multiple origin-destination pairs or commodities.
The goal is to minimize the total cost, which is the sum of fixed costs of installing arc capacities and variable costs of routing flows for each commodity.
As such, failures of network elements can lead to a reduction in its available flow capacity and a subsequent failure to meet demands.
This is particularly true in telecommunication networks where the loss of even a single (typically high-capacity) fiber-optic cable or router equipment can cause a substantial fraction of the overall flow (e.g., internet traffic) to be lost, leading to potentially huge economic impacts \citep{Markopoulou2008}.
Fortunately, these networks are typically well-engineered and therefore, such high-impact failures are rare.
At the same time, this general lack of failure data in real networks complicates the accurate estimation of their underlying distribution.

The two-stage optimization model we consider is presented in Appendix~\ref{appendix:mmcf_model}, and can be described as follows.
The first-stage problem determines the arc capacities that can be used for routing flows.
Upon failure, the second-stage model determines the routing of each commodity along the degraded network topology constrained by the first-stage arc capacities, and with the objective of minimizing the sum of variable routing costs and penalty costs for not satisfying demands.

For ease of exposition, we model only node failures, where ${\bm{\xi}}$ is supported on $\Xi = \{0, 1\}^M$ and ${\xi}_{i} =  1$ indicates that node $i$ has failed.
Since ${\bm{\xi}}$ represent on/off switches, we can employ %
Corollary~\ref{coro:two_stage_dro_obj_indicator}, and the penalty parameter $\rho = \rho^r$ can be computed using Theorem~\ref{thm:rho_computation}.
Here, the classical robust counterpart reduces to a deterministic problem (see Lemma~\ref{lem:disrupt_all_lines}), since the second-stage loss function trivially attains its worst-case value when each component of $\bm{\xi}$ is one; that is, when all nodes fail.

We conduct our experiments on the $(20,230,40,V,L)$ instance from the so-called Class~I set of instances in \citet{crainic2001bundle}.
As the name indicates, the instance has 20 nodes, 230 arcs, 40 commodities, and the letters $V$ and $L$ indicate that fixed-costs are relatively low compared to variable costs, and that the problem is loosely capacitated.
We generate empirical data by modeling each component of $\tilde{\bm{\xi}}$ as independent and identically distributed Bernoulli random variables with parameter $\psi M^{-1}$, where $\psi = 0.1$.
As before, for a fixed sample size $N$ and radius $\varepsilon$, we report average results using 100 statistically independent sets of training samples, and we estimate the variance by reporting the standard deviation over these 100 runs.
The out-of-sample performances of candidate solutions are estimated by using 1,000 statistically independent sets of testing samples.

\subsubsection{Approximation quality and computational effort}
Similar to Section~\ref{sec:results_approx_quality}, we compute the optimality gaps of the convex hull reformulation~\eqref{eq:two_stage_dro_reform} when the convex hulls $\conv{\mathcal{Z}_i}$ in~\eqref{eq:Z_function_definition}--\eqref{eq:Z_set_definition} are approximated by using the continuous relaxation $\mathcal{Z}^0$ and heuristically computed level-1 Lov{\'a}sz-Schrijver relaxation $\tilde{\mathcal{Z}}^1$, 
for sample sizes $N \in \{10, 100, 1000\}$ and radii $\varepsilon = \nu \sqrt{N^{-1}\log(N+1)}$,
$\nu \in \{0, 10^{-3}, 10^{-2}, 10^{-1}\}$.
In doing so, the true optimal value of each instance is computed by solving formulation~\eqref{eq:GK} using the column-and-constraint generation scheme.
The mixed-integer subproblems in this scheme are solved by dualizing the second-stage loss function as opposed to using its Karush-Kuhn-Tucker conditions, since it results in fewer additional variables and constraints, see~\citet{zeng2013solving}.
Any bilinear expressions involving dual variables and uncertain parameters are reformulated using indicator constraints since the lack of \textit{a priori} known upper bounds on the dual variables prohibits direct linearization using McCormick inequalities.

Figure~\ref{fig:gaps_c33} reports the average (line plot) and standard deviation (error bar) of the optimality gaps,
whereas Figure~\ref{fig:time_c33} reports the corresponding computation times,
based on 100 statistically independent sets of training samples.
Figure~\ref{fig:gaps_c33} shows that both the continuous and level-1 relaxations provide very similar and near-optimal approximations with optimality gaps never exceeding 3\% for $N=10$ and 0.5\% for $N \geq 100$.
Interestingly, the optimal first-stage decisions are different, and this can be seen from their out-of-sample performance that we present in the next subsection.

\begin{figure}[!htb]
    \centering
    \begin{subfigure}[b]{0.32\linewidth}
        \includegraphics[width=\linewidth]{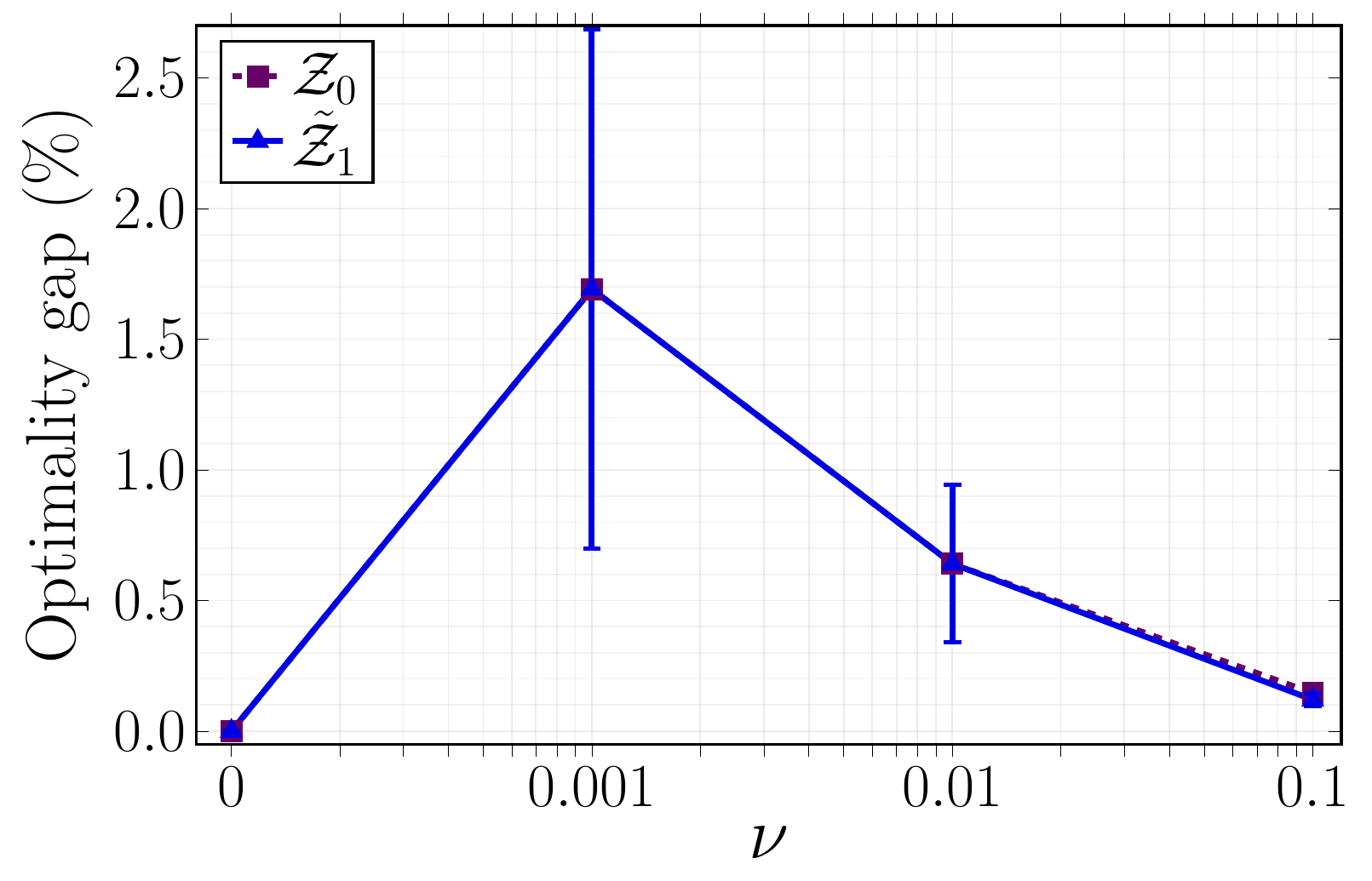}
        \caption{$N = 10$}\label{fig:gaps_c33_N10}
    \end{subfigure}\hfil
    \begin{subfigure}[b]{0.32\linewidth}
        \includegraphics[width=\linewidth]{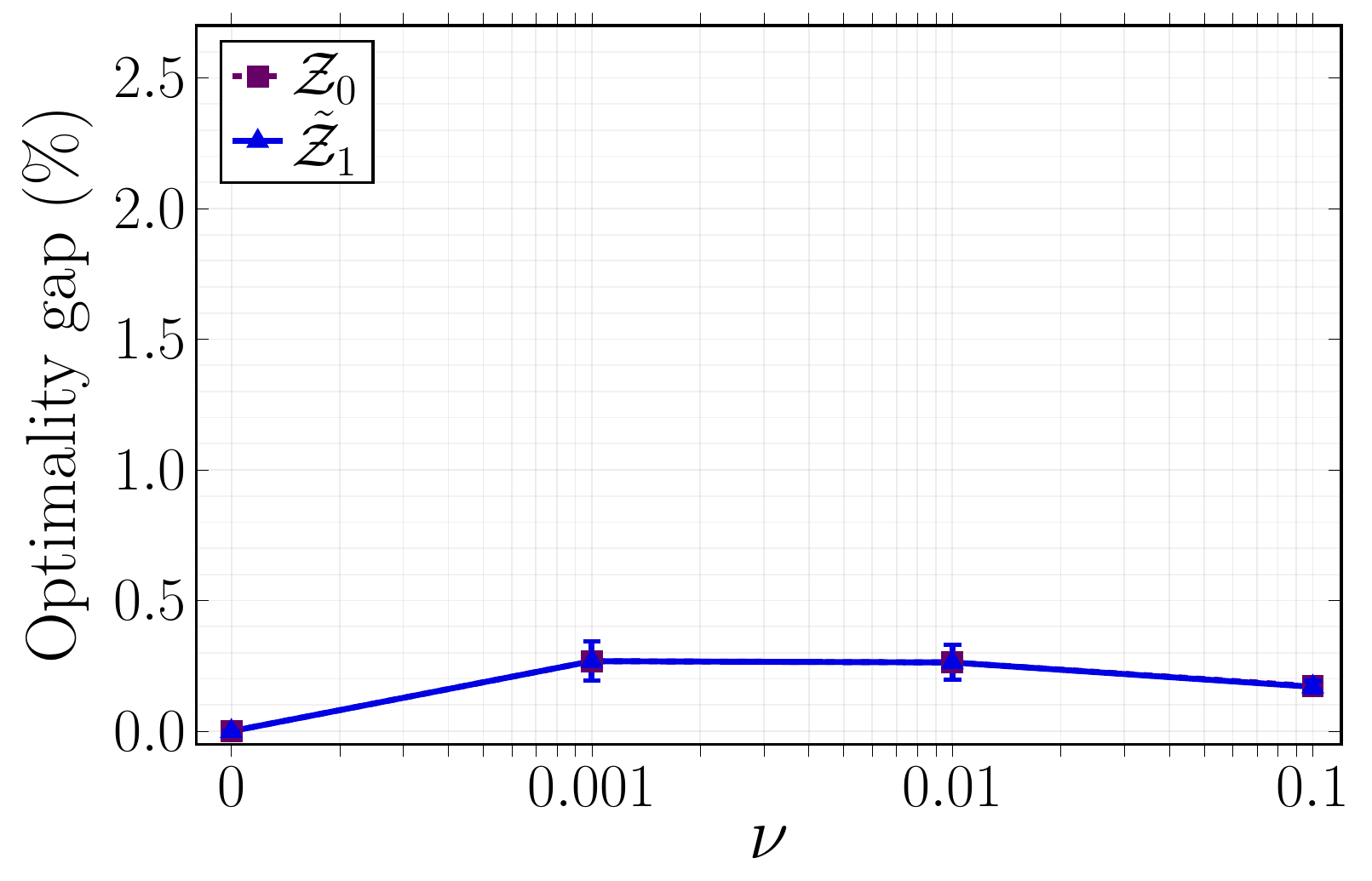}
        \caption{$N = 100$}\label{fig:gaps_c33_N100}
    \end{subfigure}\hfil
    \begin{subfigure}[b]{0.32\linewidth}
        \includegraphics[width=\linewidth]{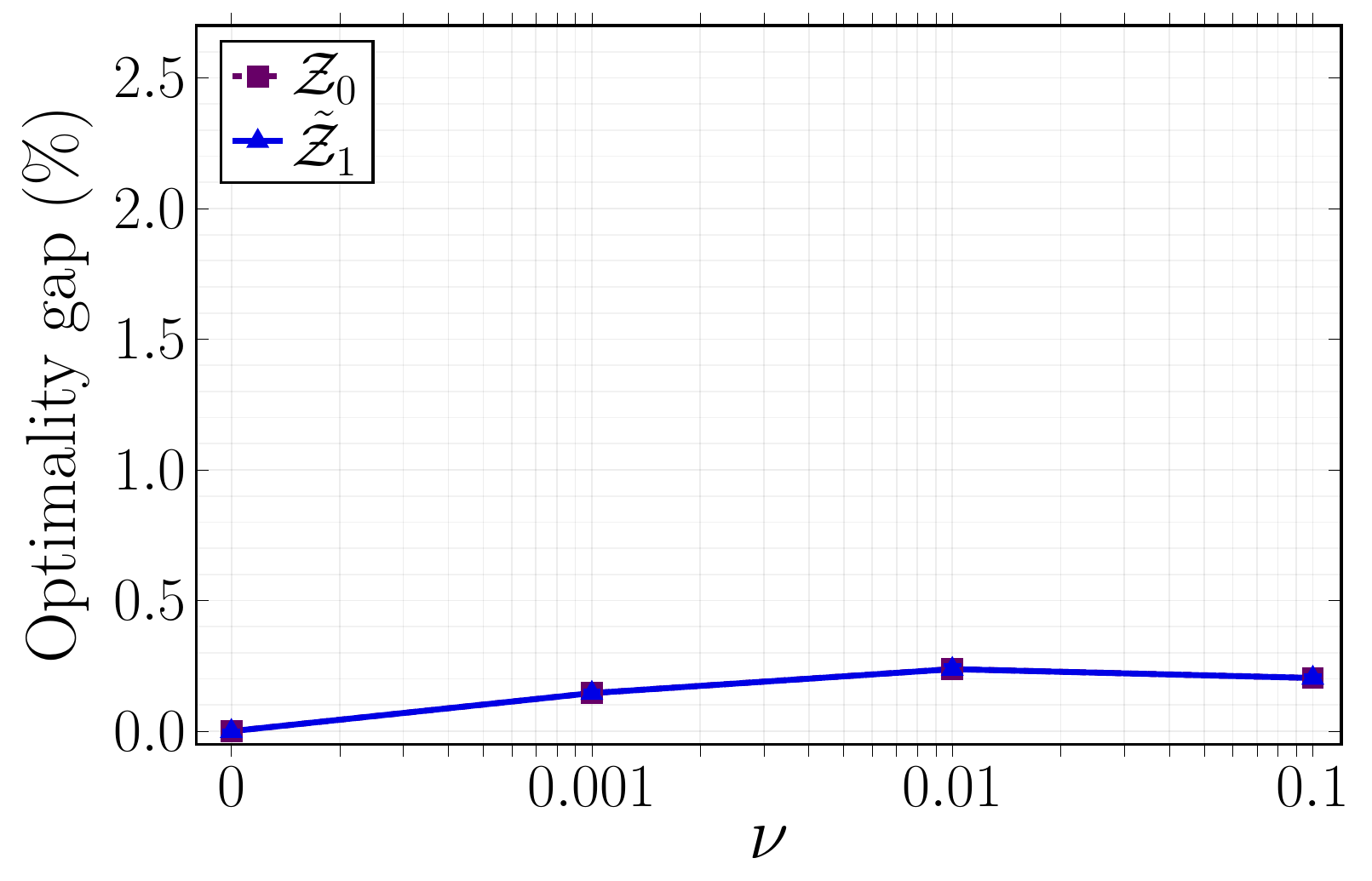}
        \caption{$N = 1000$}\label{fig:gaps_c33_N1000}
    \end{subfigure}
    \caption{\re{Optimality gaps using the continuous relaxation $\mathcal{Z}^0$ and the heuristically computed level-1 Lov{\'a}sz-Schrijver relaxation $\tilde{\mathcal{Z}}^1$, as a function of $\nu$ and $N$ where $\varepsilon = \nu \sqrt{N^{-1}\log(N+1)}$.}\label{fig:gaps_c33}}
\end{figure}

Figure~\ref{fig:time_c33} shows that both relaxations have small computation times ($<60$ seconds on average) across all $N$ and $\nu$.
When compared with the column-and-constraint generation scheme, the relative difference in their computation times is minor for small sample sizes $N$ but increases significantly for large sample sizes, where the column-and-constrain generation method can be slower by more than a factor of 10.
Similar to the Benders scheme, the mixed-integer subproblems in the column-and-constraint generation scheme can cause slow convergence, and roughly 3\% of its runs did not terminate within 10~minutes.

\begin{figure}[!htb]
    \centering
    \begin{subfigure}[b]{0.32\linewidth}
        \includegraphics[width=\linewidth]{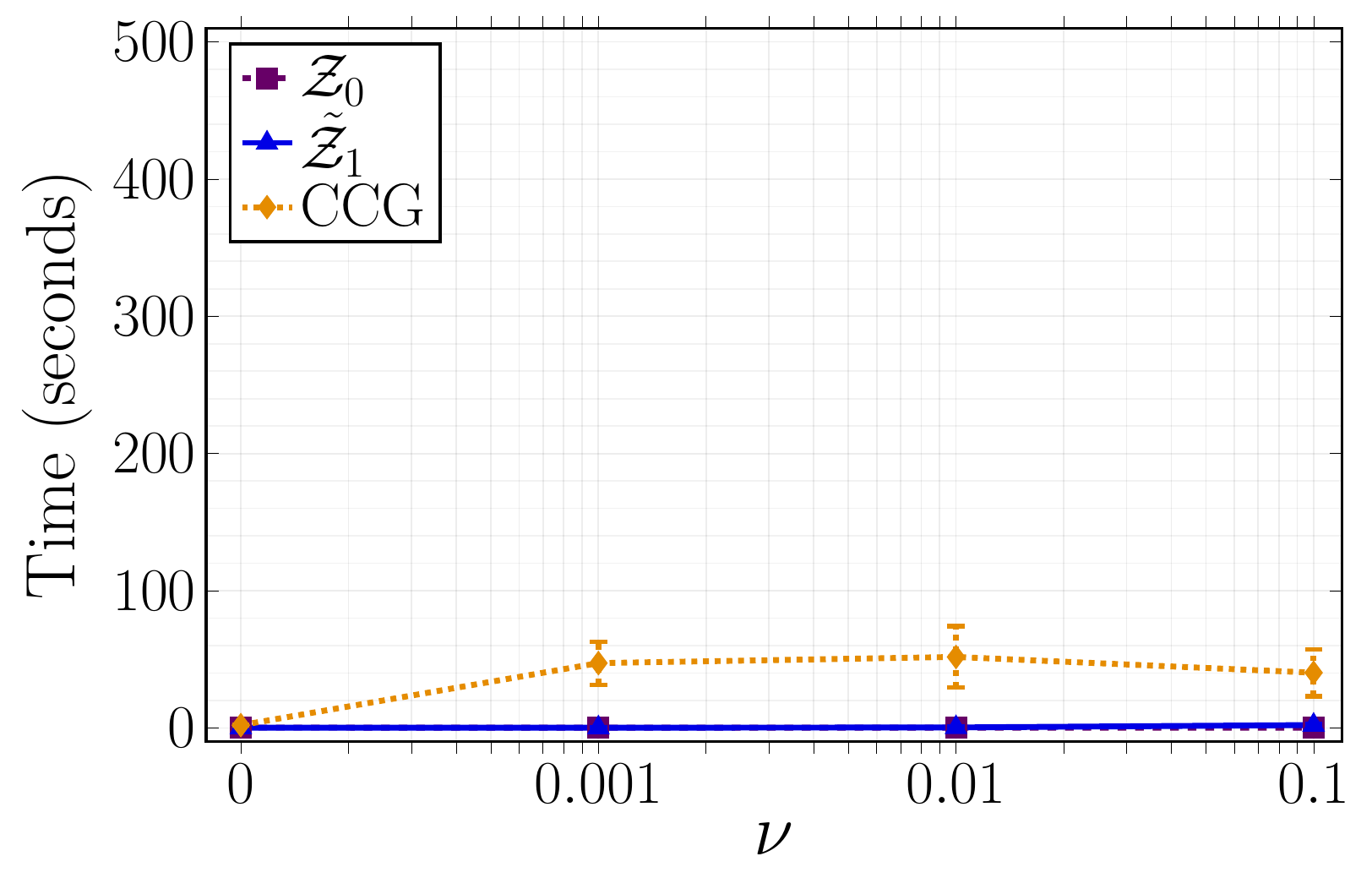}
        \caption{$N = 10$}\label{fig:time_c33_N10}
    \end{subfigure}\hfil
    \begin{subfigure}[b]{0.32\linewidth}
        \includegraphics[width=\linewidth]{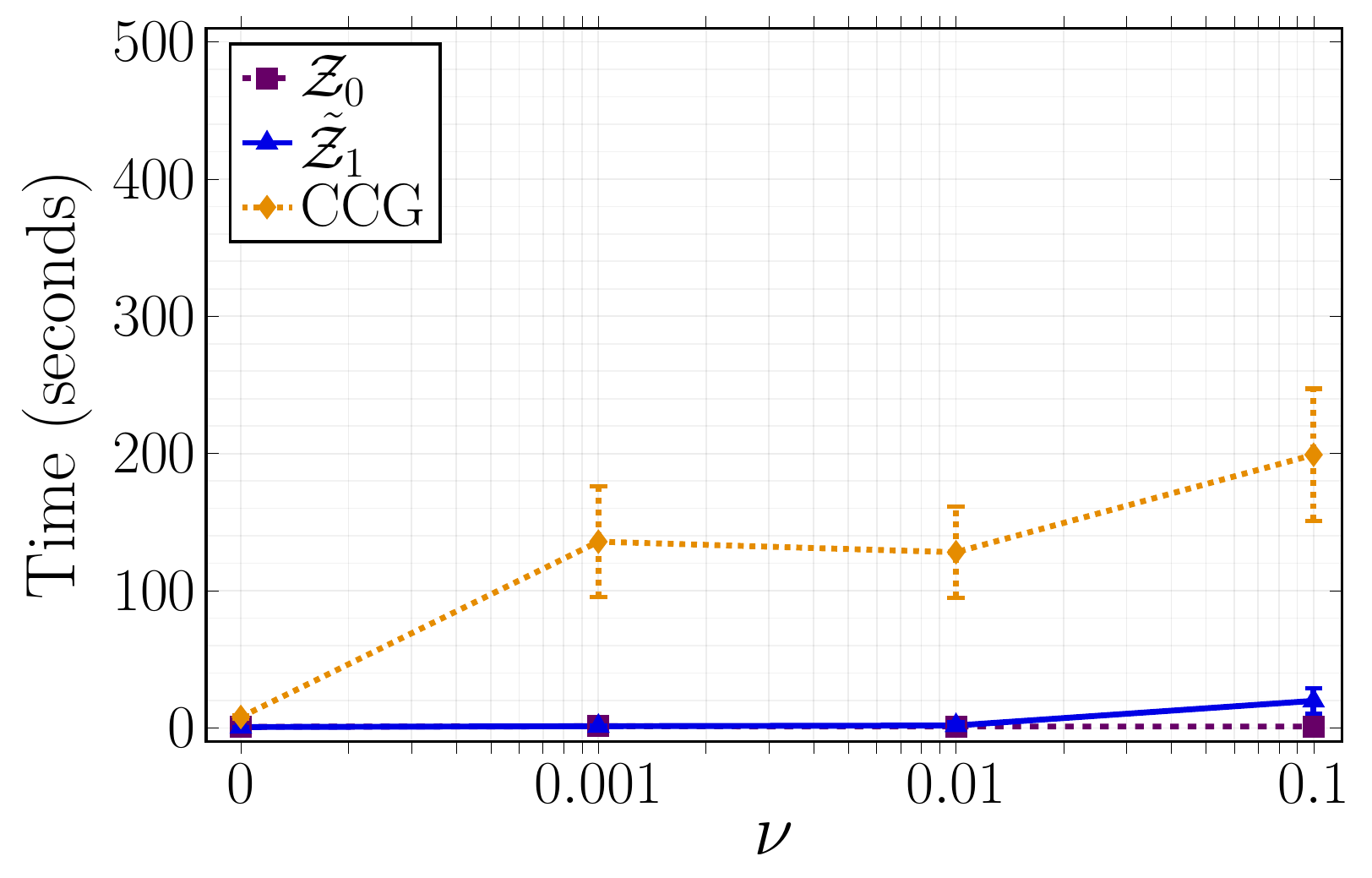}
        \caption{$N = 100$}\label{fig:time_c33_N100}
    \end{subfigure}\hfil
    \begin{subfigure}[b]{0.32\linewidth}
        \includegraphics[width=\linewidth]{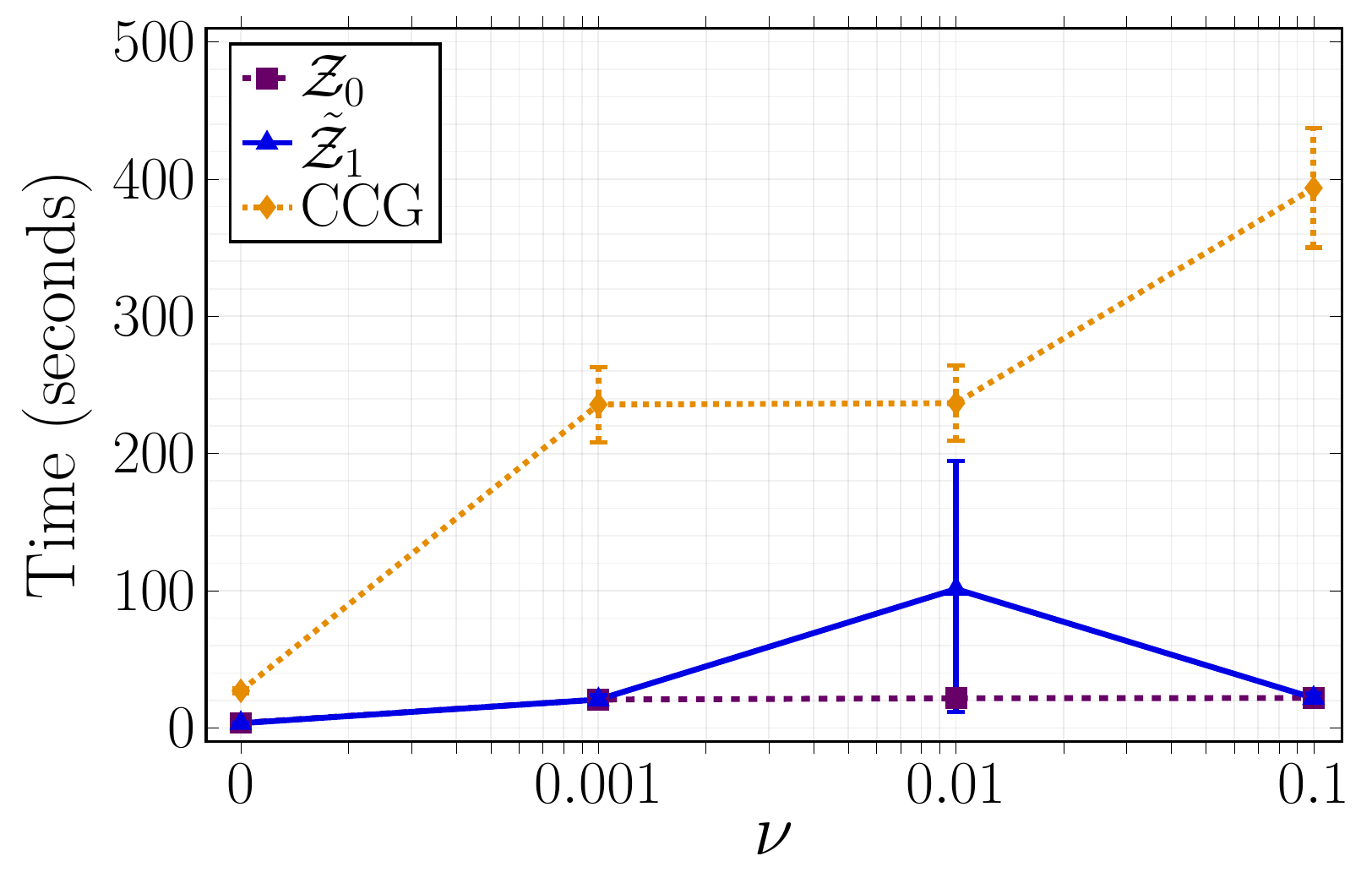}
        \caption{$N = 1000$}\label{fig:time_c33_N1000}
    \end{subfigure}
    \caption{\re{Computation times using the continuous $\mathcal{Z}^0$ and heuristically computed level-1 Lov{\'a}sz-Schrijver relaxation $\tilde{\mathcal{Z}}^1$ for solving formulation~\eqref{eq:two_stage_dro_reform}, and using the column-and-constraint generation scheme for solving formulation~\eqref{eq:GK}, as a function of $\nu$ and $N$, where $\varepsilon = \nu \sqrt{N^{-1}\log(N+1)}$}.\label{fig:time_c33}}
\end{figure}

\subsubsection{Out-of-sample performance and finite sample guarantee}

Similar to Section~\ref{sec:results_out_of_sample}, we estimate the out-of-sample performance of the first-stage solutions of the lift-and-project and sample average approximations for different sample sizes $N$ and Wasserstein radii $\varepsilon = \nu \sqrt{N^{-1}\log(N+1)}$, where $\nu \in \{0, 10^{-3}, 10^{-2}, 10^{-1}\}$.
The results are summarized in Figure~\ref{fig:c33}.

First, Figure~\ref{fig:reliability_c33} reports the {reliability} of the level-1 Lov{\'a}sz-Schrijver relaxation $\tilde{\mathcal{Z}}^1$, which is the empirical probability (over the 100 sets of training samples) that its optimal value is an upper bound on its out-of-sample cost.
We observe that the reliability increases not only with increasing values of $\nu$ (for fixed values of $N$) but also with increasing values of $N$ (for fixed values of $\nu$).
The choice $\nu = 10^{-1}$ is reliable with probability 0.8 for training datasets of small size $N$ and this increases to $>0.9$ for large values of $N$.

Figure~\ref{fig:outOfSample_c33} reports the relative improvement in out-of-sample cost of the distributionally robust model over the sample average approximation, over the 100 independent sets of training samples.
As before, we find that the distributionally robust model consistently outperforms the sample average approximation, particularly for small sample sizes $N$, where the magnitude of the relative improvement can be as high as 7\%, but decreases for large values of $N$.

Finally, Figure~\ref{fig:improvement_c33} reports the improvement in the out-of-sample performance of the level-1 Lov{\'a}sz-Schrijver relaxation $\tilde{\mathcal{Z}}^1$ for the best choice of $\nu = 10^{-1}$, relative to the continuous relaxation $\mathcal{Z}^0$ for the same $\nu$.
The relative improvement of the level-1 relaxation $\tilde{\mathcal{Z}}^1$ over the continuous relaxation $\mathcal{Z}^0$ is small yet consistently non-negative.
This is not unexpected since the latter already has tight in-sample optimality gaps, as can be seen from Figure~\ref{fig:gaps_c33}.
Nevertheless, we expect the relative improvements to be instance-dependent similar to optimal power flow, and even a few percentage points can result in long-term economic benefits.

\begin{figure}[!htbp]
    \centering
    \begin{subfigure}[b]{0.32\linewidth}
        \includegraphics[width=\linewidth]{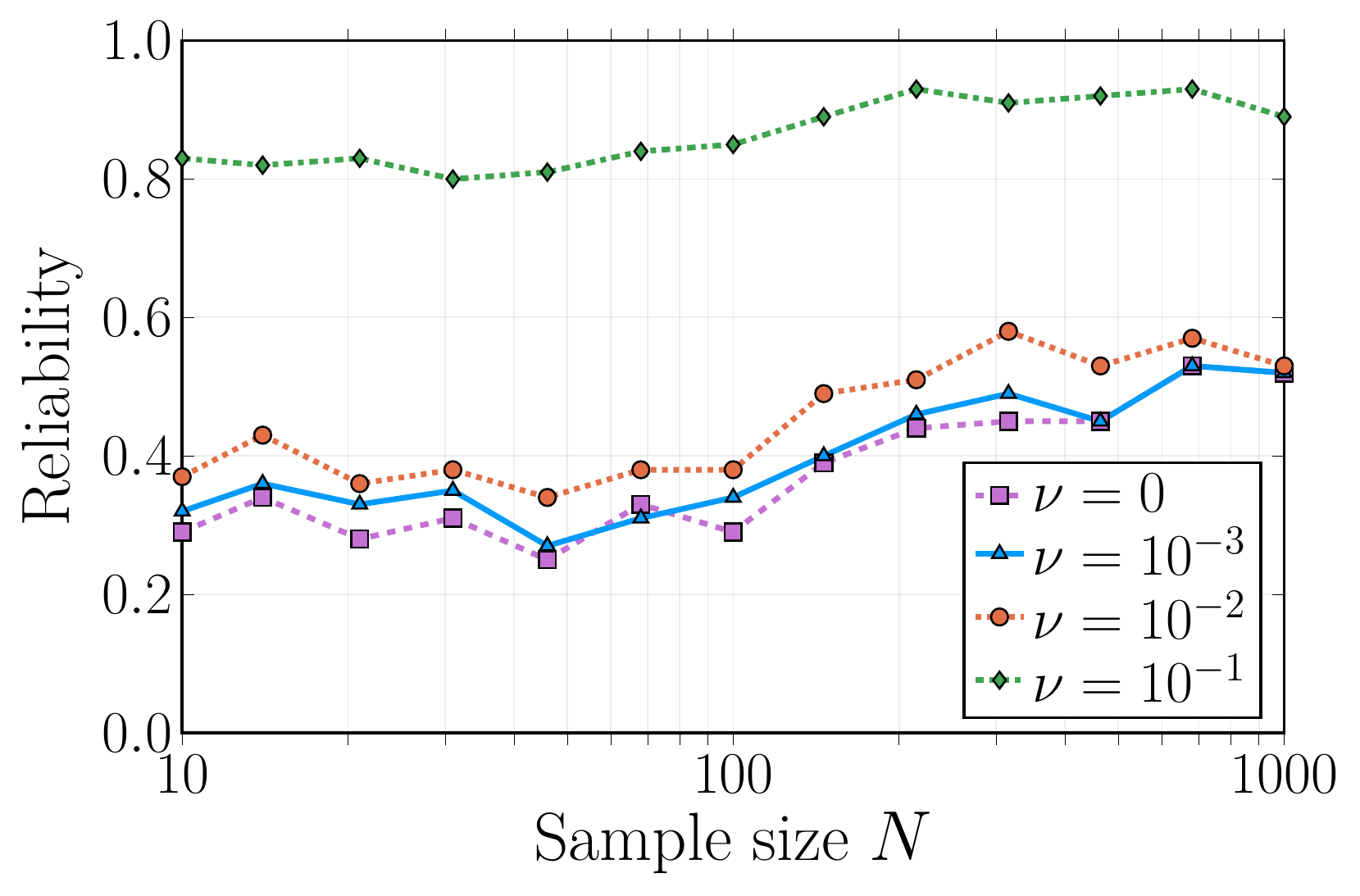}
        \caption{Reliability}\label{fig:reliability_c33}
    \end{subfigure}\hfil
    \begin{subfigure}[b]{0.32\linewidth}
        \includegraphics[width=\linewidth]{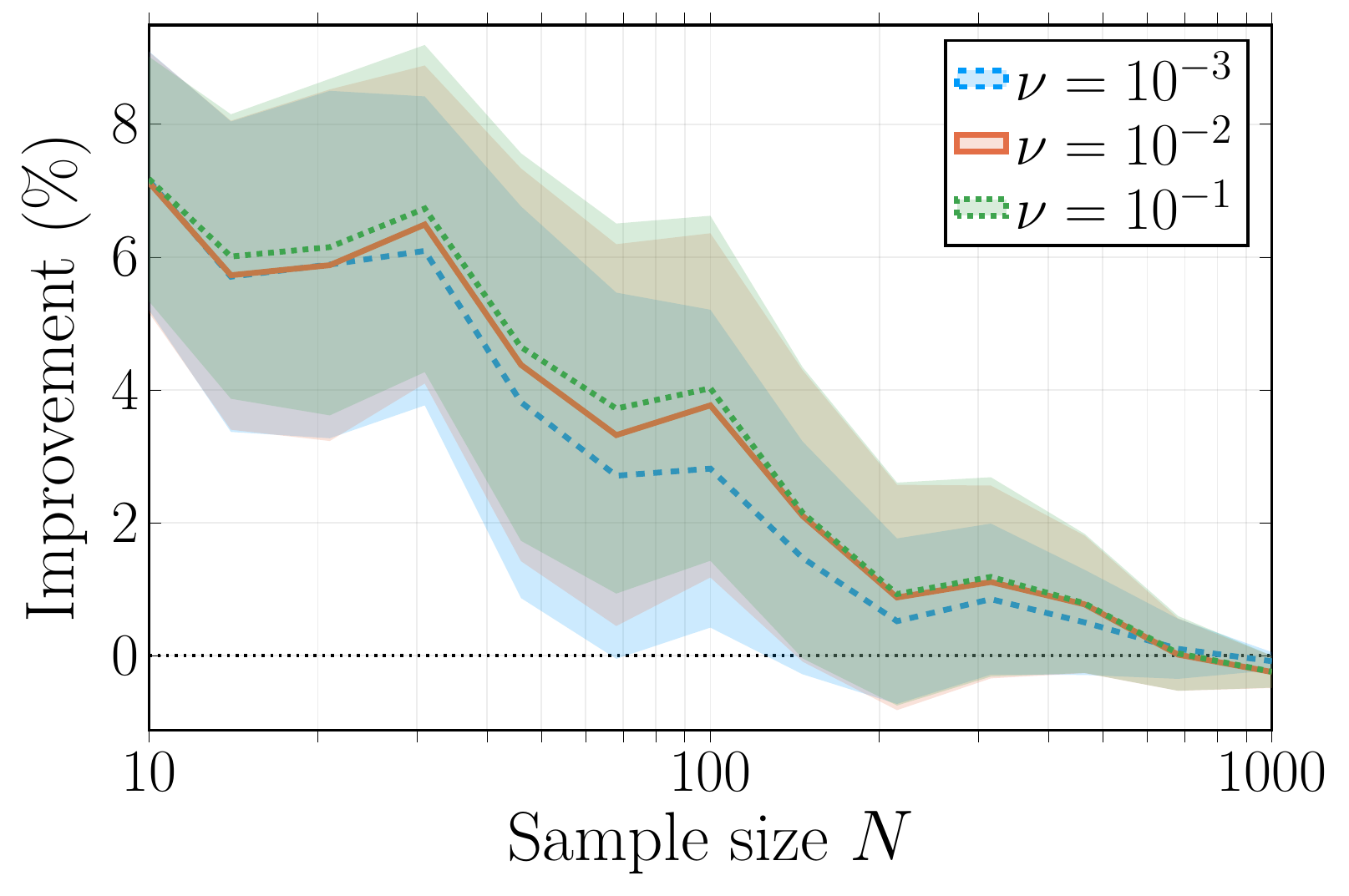}
        \caption{Out-of-sample performance}\label{fig:outOfSample_c33}
    \end{subfigure}\hfil
    \begin{subfigure}[b]{0.32\linewidth}
        \includegraphics[width=\linewidth]{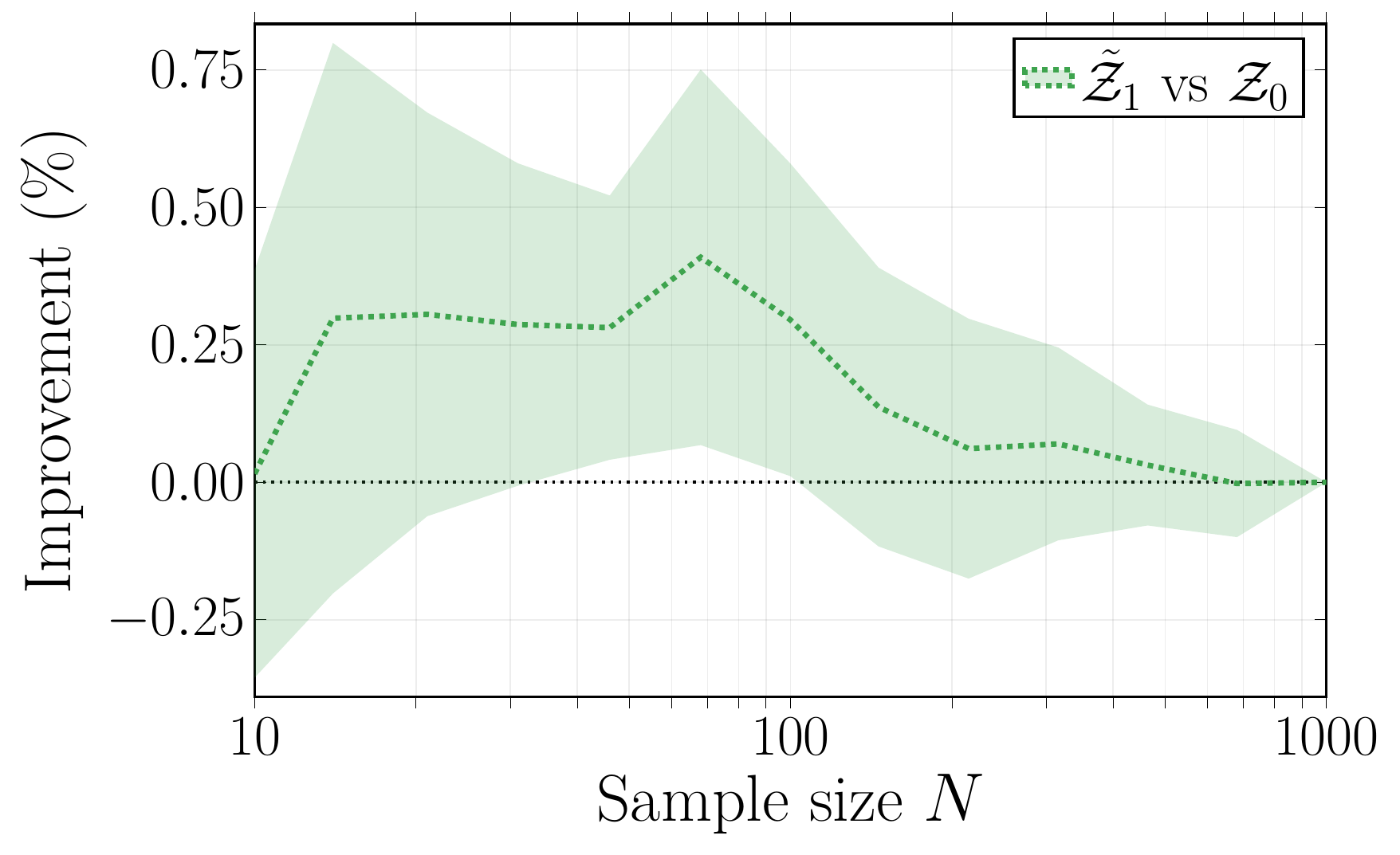}
        \caption{Relative improvement}\label{fig:improvement_c33}
    \end{subfigure}
    \caption{\re{Reliability (left plot), and relative improvements in the out-of-sample performance of the level-1 Lov{\'a}sz-Schrijver relaxation $\tilde{\mathcal{Z}}^1$ when compared with the sample average approximation (middle plot), and the continuous relaxation $\mathcal{Z}^0$ (right plot), as a function of training sample size $N$.}\label{fig:c33}}
\end{figure}

We conclude this section by noting that \citet{xie2019} address problems where $\mathcal{Q}(\bm{x}, \bm{\xi})$ is the optimal value of a linear program, and the ambiguity set $\mathcal{P}$ is the type-${\infty}$ Wasserstein ball which is closely related to yet distinct from \eqref{eq:ambiguity_set}.
They provide formulations that can serve as an alternative means to generate first-stage decisions.
Interestingly, when applied to the multi-commodity network design problem, they reduce to one of the following: \textit{(i)} sample average approximation ($\varepsilon < 1$) or deterministic problem under the worst-case realization where all nodes have failed ($\varepsilon \geq 1$), when $d$ is induced by the $\infty$-norm \citep[Theorem~4]{xie2019}, or \textit{(ii)} classical robust optimization~\eqref{eq:two_stage_ro} with $K=1$ when $d$ is induced by the $p$-norm with $p \in [1, \infty)$ and $\varepsilon < \sqrt[p]{2}$ \citep[Theorem~6]{xie2019}.
}

%% file: conclusions.tex
\section{Conclusions}\label{sec:conclusions}

Despite their ubiquity in real-world networks, optimization problems affected by rare high-impact uncertainties have not received much attention.
This is partly because of the lack of available data given their rare nature, and partly because of the incapability of classical sample average approximations to address them effectively.
This paper takes a step toward addressing these limitations by motivating a distributionally robust approach to the problem using Wasserstein ambiguity sets.
Notably, we extend the state of the art in data-driven optimization by encompassing not only two-stage conic problems but also high-dimensional discrete uncertainties.
By exploiting ideas from nonlinear penalty methods and lift-and-project techniques in global optimization, we present a simple, tractable, and tight approximation of the problem that can be efficiently computed and iteratively improved.
We use our method to tackle optimal power flow problems with random transmission line outages \re{and multi-commodity network design problems with random node failures}.
We find that the method can strongly outperform classical sample average and \re{robust optimization} approaches, especially when failures are rare but can lead to high costs associated with loss of electric power or commodity flows.

%% file: example1.tex
\section{Extended discussion of Example~1}\label{sec:examaple_1}

In Section~\ref{sec:example_1_optimal_solution_true}, we analyze the optimal solution of the network optimization problem in Example~\ref{ex:network} under the true failure distribution.
In Section~\ref{sec:example_1_optimal_solution_saa_dro}, we analyze the solutions of the sample average and distributionally robust formulations when the empirical distribution puts all its weight on the realization $\bm{\xi} = \bm{0}$, where none of the nodes fail.

\subsection{Optimal solution under the true distribution}\label{sec:example_1_optimal_solution_true}
Problem symmetry implies that the optimal arc capacities have the structure shown in Figure~\ref{figure:network_flow_optimal_structure}.
Table~\ref{table:example_1_detailed} presents a calculation of the second-stage loss function $\mathcal{Q}(\bm{x}, \bm{\xi})$ for each $\bm{\xi} \in \{0, 1\}^3$.
The optimal arc capacities can then be determined by solving the two-dimensional piecewise linear convex optimization problem (note that $0 < \epsilon \ll 1$ can be any small positive constant):
\[
\mathop{\text{minimize}}_{\bm{x} \in \mathbb{R}^2_{+}} \; 4x_1 + (4 - \epsilon)x_2 + \sum_{\bm{\xi} \in \{0, 1\}^3} \mathbb{P}[\tilde{\bm{\xi}} = \bm{\xi}] \mathcal{Q}(\bm{x}, \bm{\xi})
\]
A straightforward calculation shows that the objective function minimized at $(x_1, x_2) = (100, 100)$.

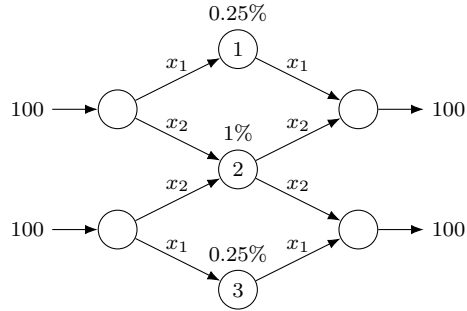
\begin{figure}[!htb]
    \centering
        \begin{tikzpicture}[scale=0.8,%
            every label/.append style={rectangle, font=\scriptsize},
            ed/.style = {-Latex},%
            cr/.style = {circle, draw, minimum size = 0.5},%
            crX/.style = {cr,pattern={north east lines},pattern color=gray!80}]
            \scriptsize
            \node[cr]                   (1) at (0,  1) {\phantom{1}};
            \node[cr]                   (2) at (0, -1) {\phantom{2}};
            \node[cr,label={$0.25\%$}] (3) at (2,  2) {1};
            \node[cr,label={$1\%$}]    (4) at (2,  0) {2};
            \node[cr,label={$0.25\%$}] (5) at (2, -2) {3};
            \node[cr]                   (6) at (4,  1) {\phantom{6}};
            \node[cr]                   (7) at (4, -1) {\phantom{7}};
            \node (o1) at (-1.5,  1) {100};
            \node (o2) at (-1.5, -1) {100};
            \node (d6) at ( 5.5,  1) {100};
            \node (d7) at ( 5.5, -1) {100};
            \draw[ed] (o1) -- (1);
            \draw[ed] (o2) -- (2);
            \draw[ed] (1) --node[above] {$x_1$} (3);
            \draw[ed] (1) --node[above] {$x_2$} (4);
            \draw[ed] (2) --node[above] {$x_2$} (4);
            \draw[ed] (2) --node[above] {$x_1$} (5);
            \draw[ed] (3) --node[above] {$x_1$} (6);
            \draw[ed] (4) --node[above] {$x_2$} (6);
            \draw[ed] (4) --node[above] {$x_2$} (7);
            \draw[ed] (5) --node[above] {$x_1$} (7);
            \draw[ed] (6) -- (d6);
            \draw[ed] (7) -- (d7);
        \end{tikzpicture}
    \caption{The structure of the optimal solution, with arc capacities $x_1$ and $x_2$ indicated above each arc.\label{figure:network_flow_optimal_structure}}
\end{figure}

\begin{table}[!htbp]
    \caption{The second-stage loss function $\mathcal{Q}(\bm{x}, \bm{\xi})$ for each $\bm{\xi} \in \{0, 1\}^3$. Here, we use $q = 1000$ to denote the penalty cost per unit of supply shortfall, and we use $[\cdot]_{+} \coloneqq \max\{\cdot, 0\}$.\label{table:example_1_detailed}}
    \begin{tabularx}{0.9\textwidth}{cccC}
            \toprule
            $\bm{\xi} = (\xi_1, \xi_2, \xi_3)$  & $\norm{\bm{\xi}}_1$ & $\mathbb{P}[\tilde{\bm{\xi}} = \bm{\xi}]$ & $\mathcal{Q}(\bm{x}, \bm{\xi})$  \\ \midrule
            $(1,1,1)$ & $3$ & $6.3 \times 10^{-8}$ & $200q$ \\
            $(1,0,1)$ & $2$ & $6.2 \times 10^{-6}$ & $2[100 - x_2]_{+}q$ \\
            $(1,1,0)$ & $2$ & $2.5 \times 10^{-5}$ & $\left(100 + [100 - x_1]_{+}\right)q$ \\
            $(0,1,1)$ & $2$ & $2.5 \times 10^{-5}$ & $\left(100 + [100 - x_1]_{+}\right)q$ \\
            $(1,0,0)$ & $1$ & $2.5 \times 10^{-3}$ & $\left(100 - x_2 + [100 - x_1 - x_2]_{+}\right)q$ \\
            $(0,0,1)$ & $1$ & $2.5 \times 10^{-3}$ & $\left(100 - x_2 + [100 - x_1 - x_2]_{+}\right)q$ \\
            $(0,1,0)$ & $1$ & $1.0 \times 10^{-2}$ & $2[100 - x_1]_{+}q$ \\
            $(0,0,0)$ & $0$ & $9.6 \times 10^{-1}$ & $2\left[100 - x_1 - x_2\right]_{+}q$ \\
            \bottomrule
    \end{tabularx}
\end{table}

\subsection{Optimal solutions using the sample average and distributionally robust formulations}\label{sec:example_1_optimal_solution_saa_dro}
Suppose now that the empirical distribution $\hat{\mathbb{P}}_N = \delta_{\bm{0}}$; that is, it puts all its weight on the realization $\bm{\xi}=(0, 0, 0)$, where none of the nodes fail.
The problem symmetry carries forth to the sample average approximation, and its optimal solution has the same symmetry as Figure~\ref{figure:network_flow_optimal_structure}.
The solution can be determined by solving the following problem:
\[
\mathop{\text{minimize}}_{\bm{x} \in \mathbb{R}^2_{+}} \; 4x_1 + (4 - \epsilon)x_2 + \mathcal{Q}(\bm{x}, \bm{0})
\;=\;
\mathop{\text{minimize}}_{\bm{x} \in \mathbb{R}^2_{+}} \; 4x_1 + (4 - \epsilon)x_2 + 2\left[100 - x_1 - x_2\right]_{+}q
\]
Since $\epsilon > 0$, the objective function is minimized at $(x_1, x_2) = (0, 100)$; this is plotted in Figure~\ref{figure:network_flow_saa:a}.

Consider now the distributionally robust formulation~\eqref{eq:two_stage_dro} using a Wasserstein ambiguity set~\eqref{eq:ambiguity_set} of radius $\varepsilon > 0$, and defined with respect to the metric $d(\bm{\xi}, \bm{\xi'}) = \norm{\bm{\xi} - \bm{\xi'}}_1$.
Again, problem symmetry leads to the same optimal solution structure as Figure~\ref{figure:network_flow_optimal_structure}.
The optimal solution can be determined by solving formulation~\eqref{eq:GK}:
\[
\mathop{\text{minimize}}_{(x_1, x_2, \alpha) \in \mathbb{R}^3_{+}} \; 4x_1 + (4 - \epsilon)x_2 + \alpha \varepsilon + \max_{\bm{\xi} \in \{0, 1\}^3} \left\{\mathcal{Q}(\bm{x}, \bm{\xi}) - \alpha \norm{\bm{\xi}}_1 \right\}
\]
For any $\varepsilon \in (0, 3)$, the minimizer is $(x_1, x_2, \alpha) = (\frac23 100, \frac13 100, \frac23 100q)$; this is plotted in Figure~\ref{figure:network_flow_dro:a}.

If we change the definition of the metric to be $d(\bm{\xi}', \bm{\xi}'') = 1$ whenever $\bm{\xi}' \neq \bm{\xi}''$ and $0$ otherwise, then the ambiguity set is equivalent to the $\phi$-divergence ambiguity set based on the {total variation distance}.
The corresponding optimal solution can be determined by solving formulation~\eqref{eq:GK}:
\[
\mathop{\text{minimize}}_{(x_1, x_2, \alpha) \in \mathbb{R}^3_{+}} \; 4x_1 + (4 - \epsilon)x_2 + \alpha \varepsilon + \max\left\{ \mathcal{Q}(\bm{x}, \bm{0}), \max_{\bm{\xi} \in \{0, 1\}^3\setminus \{\bm{0}\}} \mathcal{Q}(\bm{x}, \bm{\xi}) - \alpha \right\}
\]
The innermost maximum is attained at $\bm{\xi} = \one = (1, 1, 1)$ irrespective of $\varepsilon, \bm{x}, \alpha$.
Therefore, the outer maximum becomes 
$
\max\left\{ \mathcal{Q}(\bm{x}, \bm{0}), \mathcal{Q}(\bm{x}, \one) - \alpha \right\}
$
and it is minimized at its break-point where $\mathcal{Q}(\bm{x}, \bm{0}) = \mathcal{Q}(\bm{x}, \one) - \alpha$.
The problem then becomes equivalent to the sample average approximation
\[
\mathop{\text{minimize}}_{(x_1, x_2, \alpha) \in \mathbb{R}^3_{+}} \; 4x_1 + (4 - \epsilon)x_2 + \alpha \varepsilon + \mathcal{Q}(\bm{x}, \bm{0})
\]
and its minimum is attained at $(x_1, x_2, \alpha) = (0, 100, 0)$ for any $\varepsilon \geq 0$.

%% file: complexity.tex
\section{Complexity analysis}\label{sec:complexity}

The feasible region of the inner optimization problem in Theorem~\ref{thm:convex_reformulation} is the convex hull of an MICP-representable set.
This allows us to exploit existing tools on characterizing the convex hulls of MICP sets.
Section~\ref{sec:tractable_cases} highlights several problem instances (i.e., sufficient conditions) where this convex hull is efficiently computable and hence the distributionally robust two-stage problem~\eqref{eq:two_stage_dro} is tractable.
In general, however, the presence of discrete random parameters makes the problem intractable even when the second-stage problem $\mathcal{Q}(\bm{x}, \bm{\xi})$ is benign.
This is proved in Section~\ref{sec:complexity_analysis}.

\subsection{Tractable cases}\label{sec:tractable_cases}
We preface this subsection by noting that any notion of tractability of the distributionally robust two-stage problem~\eqref{eq:two_stage_dro} requires that optimizing $c(\bm{x})$ over $\bm{x} \in \mathcal{X}$ can be done in a tractable manner. We therefore assume this to be the case throughout this subsection.

Suppose that $\Xi = \{\bm{v}^{(1)}, \ldots, \bm{v}^{(K)}\}$ has a known inner description, where the number of vertices $K$ of $\Xi$ grows polynomially with the dimension $M$.
For example, this is the case of budget supports of the form $\Xi = \{\bm{\xi} \in \{0, 1\}^M: \bm{e}^\top \bm{\xi} \leq k\}$, where $k$ is a small, fixed input (say, $\leq 3$) and $K = O(M^k)$.
Now, if $\bm{\xi}$ is fixed to one of $\bm{v}^{(k)}$, $k \in [K]$, then the set $\mathcal{Z}_i$ in~\eqref{eq:Z_set_definition} becomes a closed convex set.
In other words, $\mathcal{Z}_i$ is the union of $K$ closed convex sets, and its convex hull can be described compactly.
\begin{proposition}
    Assume $\Xi = \{\bm{v}^{(1)}, \ldots, \bm{v}^{(K)}\}$, where $K$ is a polynomial function of $M$.
    Then, the distributionally robust two-stage problem~\eqref{eq:two_stage_dro} is equivalent to the following tractable problem:
    \begin{equation*}
    \begin{array}{l@{\;\;}l}
    \displaystyle\mathop{\text{minimize}} & \displaystyle c(\mb{x}) + \alpha \varepsilon + \frac{1}{N}\sum_{i = 1}^N \sigma_i \\
    \text{subject to} & \mb{x} \in \mathcal{X}, \; \alpha \geq 0 \\
    & \begin{rcases}
    \sigma_i \in \mathbb{R}, \; \bm{y}^{(k)} \in \mathcal{Y} \\
    \sigma_i \geq \bm{q}(\bm{v}^{(k)})^\top \bm{y}^{(k)} - \alpha d(\bm{v}^{(k)}, \hat{\mb{\xi}}^{(i)}) \\
    \bm{W}(\bm{v}^{(k)}) \bm{y}^{(k)} \geq \bm{T}(\bm{x}) \bm{v}^{(k)} + \bm{h}(\bm{x})
    \end{rcases} \; \forall k \in [K], i \in [N].
    \end{array}
    \end{equation*}
\end{proposition}
\begin{proof}
    From Lemma~\ref{prop:GK} and the stated hypothesis, problem~\eqref{eq:two_stage_dro} is equivalent to 
    \begin{equation*}
    \begin{array}{l@{\;\;\;}l}
    \phantom{=}\displaystyle\mathop{\text{minimize}}_{\mb{x} \in \mathcal{X}, \alpha \geq 0} & \displaystyle c(\mb{x}) + \alpha \varepsilon + \frac{1}{N}\sum_{i = 1}^N \max_{k \in [K]} \left\{ \mathcal{Q}(\mb{x}, \bm{v}^{(k)}) - \alpha d(\bm{v}^{(k)}, \hat{\mb{\xi}}^{(i)}) \right\}.
    \end{array}
    \end{equation*}
    We can then introduce the epigraphical variables $\sigma_i$, $i \in [N]$ to model the $i$th inner maximization in the above sum.
    The stated reformulation then follows after introducing second-stage variables $\bm{y}^{(k)}$ to represent a minimizer of $\mathcal{Q}(\bm{x}, \bm{v}^{(k)})$, $k \in [K]$.
\end{proof}

Suppose now that we have an outer description of $\Xi = \{\bm{\xi} \in \mathbb{Z}^M: \bm{E} \bm{\xi} \leq \bm{f}\}$
and the second-stage problem $\mathcal{Q}(\bm{x}, \bm{\xi})$ is a linear program $\mathcal{Y} = \mathbb{R}^{N_2}_{+}$ with objective uncertainty so that Corollary~\ref{coro:two_stage_dro_reform_obj_uncertainty} is applicable.
The next observation relies on the concept of an \textit{ideal formulation}~\citep{Padberg2001,Vielma2015:SIREV}.
A mixed-integer linear set $\{\bm{x} \in \mathbb{Z}^{p} \times \mathbb{R}^{n-p}: \bm{A}\bm{x} \leq \bm{b}\}$ is said to be \textit{ideal} if its convex hull has at least one vertex and it is equal to its linear relaxation.
We show that the convex hull reformulation~\eqref{eq:two_stage_dro_reform} becomes tractable if the set $\mathcal{Z}_i$ defined in~\eqref{eq:Z_set_definition_obj_unc} is mixed-integer linear and ideal.
\begin{proposition}\label{obs:ideal}
    Suppose that $\Xi = \{\xi \in \mathbb{Z}^M : \bm{E} \bm{\xi} \leq \bm{f}\}$ and $\mathcal{Q}(\bm{x}, \bm{\xi})$ is the optimal value of a linear program %
    with uncertain objective coefficients: $\bm{W}(\bm{\xi}) = \bm{W}_0$ and $\bm{T}(\bm{x}) = \bm{0}$.
    Suppose also that the mixed-integer linear formulation $\{(\bm{\xi}, \bm{\lambda}) \in \mathbb{Z}^M \times \mathbb{R}^L_{+}: \bm{E} \bm{\xi} \leq \bm{f}, \; \bm{W}_0^\top\bm{\lambda} -\bm{Q}\bm{\xi} \leq \bm{q}_0 \}$ is ideal.
    \begin{enumerate}
        \item[(i)] If the reference metric $d(\bm{\xi}, \bm{\xi'}) = \norm{\bm{\xi} - \bm{\xi'}}_1$ is the 1-norm, then the two-stage problem~\eqref{eq:two_stage_dro} is equivalent to the following tractable optimization problem:
        \begin{equation*}
        \begin{array}{l@{\;\;}l}
        \displaystyle\mathop{\text{minimize}} & \displaystyle c(\mb{x}) + \alpha \varepsilon + \frac{1}{N}\sum_{i = 1}^N \left[\bm{q}_0^\top \bm{y}^{(i)} + \bm{f}^\top \bm{\eta}^{(i)} - \alpha \one^\top \hat{\bm{\xi}}^{(i)}\right] \\
        \text{subject to} & \mb{x} \in \mathcal{X}, \; \alpha \geq 0 \\
        & \begin{rcases}
        \bm{y}^{(i)} \in \mathcal{Y}, \; \bm{\eta}^{(i)} \in \mathbb{R}^F_{+} \\
        -\bm{Q}^\top \bm{y}^{(i)} + \bm{E}^\top \bm{\eta}^{(i)} \geq -\alpha (\one - 2\hat{\bm{\xi}}^{(i)}) \\
        \bm{W}_0 \bm{y}^{(i)} \geq \bm{h}(\bm{x})
        \end{rcases} \; \forall i \in [N].
        \end{array}
        \end{equation*}
        
        \item[(ii)] If $\varepsilon \geq \max_{\bm{\xi}, \bm{\xi}' \in \Xi} d(\bm{\xi}, \bm{\xi}')$, then the two-stage (classical robust optimization) problem~\eqref{eq:two_stage_dro} is equivalent to the following tractable optimization problem:
        \begin{equation*}
        \begin{array}{l@{\;\;}l}
        \displaystyle\mathop{\text{minimize}} & \displaystyle c(\mb{x}) + \bm{q}_0^\top \bm{y} + \bm{f}^\top \bm{\eta} \\
        \text{subject to} & \displaystyle \mb{x} \in \mathcal{X}, \; \bm{y} \in \mathcal{Y}, \; \bm{\eta} \in \mathbb{R}^F_{+} \\
        & \displaystyle
        -\bm{Q}^\top \bm{y} + \bm{E}^\top \bm{\eta} \geq \bm{0} \\
        & \displaystyle
        \bm{W}_0 \bm{y} \geq \bm{h}(\bm{x})
        \end{array}
        \end{equation*}
    \end{enumerate}
\end{proposition}
\begin{proof}
    The conditions of this proposition allow us to use the result of Corollary~\ref{coro:two_stage_dro_reform_obj_uncertainty}.
    
    \textit{(i)} When the reference metric $d(\bm{\xi}, \bm{\xi'}) = \norm{\bm{\xi} - \bm{\xi'}}_1$ is the 1-norm, the condition $(\bm{\xi} - \hat{\bm{\xi}}^{(i)}, \tau) \in \mathcal{C}^{M+1}$ appearing in~\eqref{eq:Z_set_definition_obj_unc} is equivalent to $\tau \geq \norm{\bm{\xi} - \hat{\bm{\xi}}^{(i)}}_1$. If we exploit the fact that $\bm{\xi}$ and $\hat{\bm{\xi}}^{(i)}$ are both binary-valued, then the right-hand side of this inequality is equal to $(\one - 2\hat{\bm{\xi}}^{(i)})^\top \bm{\xi} + \one^\top \hat{\bm{\xi}}^{(i)}$ and hence, linear in $\bm{\xi}$.
    Since the objective function of~\eqref{eq:Z_function_definition_obj_unc} is linear in $\tau$ and $\tau$ only appears in a single constraint in $\mathcal{Z}_i$ that is linear in $\bm{\xi}$, we can reformulate~\eqref{eq:Z_function_definition_obj_unc}--\eqref{eq:Z_set_definition_obj_unc} as follows:
    \begin{align*}
    &\displaystyle Z_i(\mb{x}, \alpha) = 
    \mathop{\text{maximize}}_{(\bm{\xi}, \bm{\lambda}) \in \mathop{\mathrm{cl\,conv}}(\hat{\mathcal{Z}})}  \displaystyle \left\{\bm{h}(\bm{x})^\top \bm{\lambda}
    - \alpha \left((\one - 2\hat{\bm{\xi}}^{(i)})^\top \bm{\xi} + \one^\top \hat{\bm{\xi}}^{(i)}\right) \right\}, \\
    &\displaystyle \hat{\mathcal{Z}} = \left\{
    (\bm{\xi}, \bm{\lambda}) \in \mathbb{Z}^M \times \mathbb{R}^L_{+}: 
    \displaystyle \bm{E} \bm{\xi} \leq \bm{d}, \; \bm{q}_0 + \bm{Q}\mb{\xi} - \bm{W}_0^\top \bm{\lambda} \geq \bm{0}
    \right\},
    \end{align*}
    where $\Xi = \{\xi \in \mathbb{Z}^M : \bm{E} \bm{\xi} \leq \bm{d}\}$ and $\mathcal{Y} = \mathbb{R}^{N_2}_{+}$.
    By hypothesis, this  formulation of $\hat{\mathcal{Z}}$ is ideal, and hence $\conv{\hat{\mathcal{Z}}} = \{(\bm{\xi}, \bm{\lambda}) \in \mathbb{R}^M \times \mathbb{R}^L_{+}: \bm{E} \bm{\xi} \leq \bm{f}, \; \bm{W}_0^\top\bm{\lambda} -\bm{Q}\bm{\xi} \leq \bm{q}_0 \}$. We then obtain the stated result by exploiting strong linear programming duality.
    
    \textit{(ii)} When $\varepsilon \geq \max_{\bm{\xi}, \bm{\xi}' \in \Xi} d(\bm{\xi}, \bm{\xi}')$, problem~\eqref{eq:two_stage_dro} reduces to a classical two-stage robust optimization as per Remark~\ref{rem:reduction_to_ro_and_sp}. In this case, the optimal value of $\alpha$ in the convex hull reformulation~\eqref{eq:two_stage_dro_reform} is $0$ and the two-stage problem~\eqref{eq:two_stage_dro} is equivalent to
    \begin{equation*}
        \displaystyle\mathop{\text{minimize}}_{\bm{x} \in \mathcal{X}} \;\; \displaystyle c(\mb{x}) + \hat{Z}(\mb{x}),
    \end{equation*}
    where the set $\hat{\mathcal{Z}}$ is defined as above and we define the function $\hat{Z} :\mathcal{X} \mapsto \mathbb{R}$ as follows:
    \begin{align*}
     &\displaystyle \hat{Z}(\mb{x}) = 
    \mathop{\text{maximize}}_{(\bm{\xi}, \bm{\lambda}) \in \mathop{\mathrm{cl\,conv}}(\hat{\mathcal{Z}})}  \displaystyle \bm{h}_0(\bm{x})^\top \bm{\lambda} .
    \end{align*}
    By the same argument as in part \textit{(i)}, the above formulation of $\hat{\mathcal{Z}}$ is ideal, and we obtain the stated result using strong linear programming duality.
\end{proof}
A well-known sufficient condition that guarantees idealness of a mixed-integer linear formulation is \textit{total unimodularity} of the constraint matrices.
In particular, if $\bm{W}_0$, $\bm{Q}$ and $\bm{E}$ are totally unimodular (e.g., they are network matrices), then the conditions of Proposition~\ref{obs:ideal} can be satisfied.
We refer to~\citet{Vielma2015:SIREV} for a more general overview of ideal formulations.

\subsection{Computational complexity}\label{sec:complexity_analysis}
Even though the preceding subsection presented some sufficient conditions for tractability, the distributionally robust two-stage problem~\eqref{eq:two_stage_dro} is intractable even in benign settings. %
\begin{theorem}[NP-hardness]
    The distributionally robust two-stage problem~\eqref{eq:two_stage_dro} is strongly NP-hard even if there are no first-stage decisions ($N_1 = 0$), the support $\Xi = \{0, 1\}^M$ is the zero-one hypercube, $\mathcal{Q}(\bm{x}, \bm{\xi})$ is the optimal value of a linear program ($\mathcal{Y} = \mathbb{R}^{N_2}_{+}$), and either
    \begin{enumerate}
        \item[(i)] uncertainty affects only the objective function, $\bm{W}(\bm{\xi}) = \bm{W}_0$ and $\bm{T}(\bm{x}) = \bm{0}$, or
        \item[(ii)] uncertainty affects only the constraint right-hand sides: $\bm{W}(\bm{\xi}) = \bm{W}_0$, $\bm{Q} = \bm{0}$ and $\bm{T}(\bm{x}) = \bm{T}_0$.
    \end{enumerate}
    It remains NP-hard even if $N_2 = 2$ in case {(i)}.
\end{theorem}
\begin{proof}
    We prove part \textit{(i)} by describing a polynomial reduction of the strongly NP-hard integer programming feasibility problem~\citep{Garey:1979}:
    \begin{quote}
        Given $\bm{Q} \in \mathbb{Z}^{N_2 \times M}$ and $\bm{q}_0 \in \mathbb{Z}^{N_2}$, is there a vector $\bm{\xi} \in \{0, 1\}^{M}$ such that $\bm{Q}\bm{\xi} \geq -\bm{q}_0$?
    \end{quote}
    We show that this problem has an affirmative answer if and only if the problem
    \begin{equation*}
    \max_{\bm{\xi} \in \{0, 1\}^M} \min_{\bm{y} \in \mathbb{R}^{N_2}_{+}} \left\{
    (\bm{Q}\bm{\xi} + \bm{q}_0)^\top \bm{y} : \one^\top \bm{y} = 1
    \right\}
    \end{equation*}
    has an optimal value greater than or equal to $0$. This can be viewed as an instance of the two-stage problem~\eqref{eq:two_stage_dro} without first-stage decisions $(N_1 = 0)$, %
    see Remark~\ref{rem:reduction_to_ro_and_sp}.
    For fixed $\bm{\xi}$, the inner minimization evaluates to $\min \left\{\one_1^\top(\bm{Q}\bm{\xi} + \bm{q}_0), \ldots, \one_{N_2}^\top(\bm{Q}\bm{\xi} + \bm{q}_0)\right\}$, where $\one_j$ denotes the $j$th unit canonical vector in $\mathbb{R}^{N_2}$. Therefore, the optimal value of the above problem is greater than or equal to $0$ if and only if there exists $\bm{\xi} \in \{0, 1\}^M$ such that $\min \left\{\one_1^\top(\bm{Q}\bm{\xi} + \bm{q}_0), \ldots, \one_{N_2}^\top(\bm{Q}\bm{\xi} + \bm{q}_0)\right\} \geq 0$, that is, if and only if $\one_j^\top(\bm{Q}\bm{\xi} + \bm{q}_0) \geq 0$ for all $j \in [N_2]$, that is, if and only if $\bm{Q}\bm{\xi} + \bm{q}_0 \geq \bm{0}$.
    
    To prove part \textit{(ii)}, we recall the strongly NP-hard problem of convex quadratic maximization over the unit hypercube~\citep{deKlerk2008complexity}:
    \begin{quote}
        Given a symmetric positive semidefinite matrix $\bm{T}_0 \in \mathbb{R}^{M \times M}$ and a scalar $t_0 \geq 0$, is there a vector $\bm{\xi} \in [0, 1]^{M}$ such that $\bm{\xi}^\top \bm{T}_0 \bm{\xi} \geq t_0$?
    \end{quote}
    We show that this problem has an affirmative answer if and only if the problem
    \begin{equation*}
    \max_{\bm{\xi} \in \{0, 1\}^M} \min_{\bm{y} \in \mathbb{R}^{M}_{+}} \left\{
    \one^\top \bm{y} : \bm{y} \geq \bm{T}_0\bm{\xi}
    \right\}
    \end{equation*}
    has an optimal value greater than or equal to $h_0$. As before, this problem is an instance of the two-stage problem~\eqref{eq:two_stage_dro} with $\bm{W}_0 = \eye$ and where the uncertainty affects only the right-hand sides of the second-stage problem. By strong linear programming duality, it is equivalent to
    \begin{equation*}
    \max_{\substack{\bm{\xi} \in \{0, 1\}^M\\ \bm{\lambda} \in [0, 1]^M}} \bm{\xi}^\top \bm{T}_0 \bm{\lambda} = \max_{\substack{\bm{\xi} \in [0, 1]^M\\ \bm{\lambda} \in [0, 1]^M}} \bm{\xi}^\top \bm{T}_0 \bm{\lambda},
    \end{equation*}
    where the equality follows from the fact that there always exists an optimal vertex solution of the bilinear program on the right-hand side of the equality~\citep[Theorem~2.1]{konno1976maximization}.
    The claim now follows from the fact that the optimal value of the right-hand side bilinear program is greater than or equal to $t_0$ if and only if there exists a vector $\bm{\xi} \in [0, 1]^{M}$ such that $\bm{\xi}^\top \bm{T}_0 \bm{\xi} \geq t_0$ \citep[Theorem~2.2]{konno1976maximization}.
    
    To show NP-hardness in part \textit{(i)} even when $N_2 = 2$, we describe a polynomial reduction from the weakly NP-hard subset sum problem~\citep{Garey:1979}:
    \begin{quote}
        Given $\bm{q} \in \mathbb{Z}^M$ and $q_0 \in \mathbb{Z}$, is there a subset $J \subseteq [M]$ such that $\sum_{j \in J} q_j = q_0$?
    \end{quote}
    We show that this problem has an affirmative answer if and only if the problem
    \begin{equation*}
    \max_{\bm{\xi} \in \{0, 1\}^M} \min_{\bm{y} \in \mathbb{R}^{2}_{+}} \left\{
    (\bm{q}^\top \bm{\xi} - q_0) (y_1 - y_2) : y_1 + y_2 = 1
    \right\}
    \end{equation*}
    has an optimal value greater than or equal to $0$. As before, this problem can be viewed as a special case of the two-stage problem~\eqref{eq:two_stage_dro} with objective uncertainty. For fixed $\bm{\xi}$, the inner minimization evaluates to $-\abs{\bm{q}^\top \bm{\xi} - q_0}$ which is always non-positive. Therefore, the optimal value is greater than or equal to $0$ if and only if there exists a vector $\bm{\xi} \in \{0, 1\}^M$ such that $\abs{\bm{q}^\top \bm{\xi} - q_0} = 0$, that is, if and only if there exists a subset $J = \{j \in [M] : \xi_j = 1\}$ such that $\sum_{j \in J} q_j = q_0$.
\end{proof}

%% file: benders.tex
\section{Benders decomposition}\label{sec:benders}

The central idea in Benders decomposition is to solve the convex hull reformulation~\eqref{eq:two_stage_dro_reform} by iteratively refining an inner approximation of the value function $Z_i(\bm{x}, \alpha)$, for each $i \in [N]$.
We generically write the latter as $Z_i(\bm{x}, \alpha) = \max_{\bm{z} \in \mathcal{Z}_i} \bm{\gamma}(\bm{x}, \alpha)^\top \bm{z}$.
Recall that the latter optimization problem can be formulated as an MICP, in one of two equivalent forms, by using the linearized reformulation (see Section~\ref{sec:linearization}) or the penalty reformulation (see Section~\ref{sec:penalty_based_formulation}).
To present the Benders decomposition algorithm, we re-write~\eqref{eq:two_stage_dro_reform} as a semi-infinite program:
\begin{equation*}
\begin{array}{l@{\;\;}l}
\displaystyle\mathop{\text{minimize}}_{\mb{x} \in \mathcal{X}, \; \alpha \geq 0} & \displaystyle c(\mb{x}) + \alpha \varepsilon + \frac{1}{N}\sum_{i = 1}^N \sigma_i \\
\text{subject to} & \displaystyle \sigma_i \geq \bm{\gamma}(\bm{x}, \alpha)^\top \bm{z}, \quad \forall \bm{z} \in \mathcal{Z}_i, i \in [N].
\end{array}
\end{equation*}
Observe that, in both cases where an MICP representation is possible, $\bm{\gamma}(\bm{x}, \alpha)$ is componentwise convex and $\mathcal{Z}_i \subseteq \mathbb{R}^n_{+}$; therefore, each of the semi-infinite constraints in the above problem defines a convex feasible region (in $\bm{x}$, $\alpha$ and $\bm{\sigma}$).
We present the algorithm next.
\begin{enumerate}
    \item Initialize $\hat{\mathcal{Z}}_i = \emptyset$, for each $i \in [N]$.
    \item Solve the following \textit{master problem}.
    \begin{equation}\label{eq:benders:master}
    \begin{array}{l@{\;\;}l}
    \displaystyle\mathop{\text{minimize}}_{\mb{x} \in \mathcal{X}, \; \alpha \geq 0, \; \bm{\sigma} \in \mathbb{R}^N} & \displaystyle c(\mb{x}) + \alpha \varepsilon + \frac{1}{N}\sum_{i = 1}^N \sigma_i \\
    \text{subject to} & \displaystyle \sigma_i \geq \max\left\{\mathcal{Q}(\bm{x}, \hat{\bm{\xi}}^{(i)}), \bm{\gamma}(\bm{x}, \alpha)^\top \bm{z} \right\}, \quad \forall \bm{z} \in \hat{\mathcal{Z}}_i, i \in [N].
    \end{array}
    \end{equation}
    Let $(\bm{x}^\star, \alpha^\star, \bm{\sigma}^\star)$ denote an optimal solution.
    \item Solve the following \textit{subproblem}, for each $i \in [N]$:
    \begin{equation}\label{eq:benders:slave}
    \mathop{\text{maximize}}_{\bm{z} \in \mathcal{Z}_i}\, \bm{\gamma}(\bm{x}^\star, \alpha^\star)^\top \bm{z}
    \end{equation}
    Let $\bm{z}^{\star,i}$ denote an optimal solution.
    \item For each $i \in [N]$, if $\bm{\gamma}(\bm{x}^\star, \alpha^\star)^\top \bm{z}^{\star,i} > \sigma_i^\star$, add $\bm{z}^{\star,i}$ to $\hat{\mathcal{Z}}_i$.
    
    If $\hat{\mathcal{Z}}_i$ was not updated for any $i \in [N]$, stop. Otherwise, go to Step~2.    
\end{enumerate}

We make some remarks about the algorithm next.
\begin{itemize}
    \item The master problem~\eqref{eq:benders:master} is always feasible. Indeed, its optimal value always constitutes a lower bound to the optimal value of the distributionally robust two-stage problem, which always exists and is finite (see Section~\ref{sec:reformulation}).
    Similarly, the term $\mathcal{Q}(\bm{x}, \hat{\bm{\xi}}^{(i)})$ in the constraint ensures that it is also always bounded.
    
    \item The optimal value of the subproblem~\eqref{eq:benders:slave} also always exists and is finite, because of the assumption of complete recourse.
    
    \item The subproblem~\eqref{eq:benders:slave} can be solved as an MICP by using either the McCormick linearization or the penalty-based formulation from Section~\ref{sec:reformulation}.
    
    \item The computational efficiency of the algorithm can be improved in several ways. First, the solution of the subproblems in Step~3 can be carried out in parallel, if desired. Second, we don't need to solve the subproblems to global optimality; we can stop as soon as we find a solution $\bm{z}^{\dagger,i}$ that satisfies $\bm{\gamma}(\bm{x}^\star, \alpha^\star)^\top \bm{z}^{\dagger,i} > \sigma_i^\star$. Third, for the same reason, we can employ $\sigma^\star$ as a lower bound when solving the subproblem for $i \in [N]$.
\end{itemize}

We note that convergence of the algorithm is guaranteed only asymptotically in general. Finite convergence is guaranteed only if the feasible region of the second-stage problem is a linear program, that is, $\mathcal{Y}= \mathbb{R}^{N_2}_{+}$; and Step~3 solves the resulting mixed-integer linear programs to global optimality. The reason is the finite number of extreme points of the sets $\mathcal{Z}_i$, $i \in [N]$, in this case.

%% file: risk_averse.tex
\section{Extension to risk-averse objective functions}\label{sec:cvar}
The objective function of the distributionally robust two-stage problem~\eqref{eq:two_stage_dro} minimizes the worst-case expectation of the loss and reflects a risk-neutral approach.
In the context of rare high-impact events, where the loss function increases sharply with extreme realizations of the uncertainty, it might be preferable to adopt a risk-averse approach and minimize the tail of the distribution of the random loss.
A natural risk measure in this case is the conditional value-at-risk, which is the conditional expectation above the $(1 - p)$-quantile of the random loss function $\mathcal{Q}(\bm{x}, \tilde{\bm{\xi}})$:
\begin{equation}\label{eq:cvar_definition}
\cvar_p \left[
\mathcal{Q}(\bm{x}, \tilde{\bm{\xi}})
\right] =
\inf_{w \in \mathbb{R}} w + \frac{1}{p}\mathbb{E}_\mathbb{P} \left[
\max\left\{\mathcal{Q}(\bm{x}, \tilde{\bm{\xi}}) - w, 0\right\}
\right].
\end{equation}
In this subsection, we show that convex hull reformulation of Theorem~\ref{thm:convex_reformulation} also extends to this setting.
Specifically, we study the following distributionally robust two-stage risk-averse stochastic program:
\begin{equation}\label{eq:two_stage_dro_cvar}
\mathop{\text{minimize}}_{\bm{x} \in \mathcal{X}} \;
c(\bm{x})
+
\sup_{\mathbb{P} \in \mathcal{P}}
\cvar_p \left[
\mathcal{Q}(\bm{x}, \tilde{\bm{\xi}})
\right]
\end{equation}
\begin{theorem}[Convex hull reformulation for conditional-value-at-risk]\label{thm:convex_reformulation_cvar}
    For any $p \in (0, 1]$, the distributionally robust two-stage risk-averse stochastic program~\eqref{eq:two_stage_dro_cvar} admits the following reformulation, 
    \begin{equation}\label{eq:two_stage_dro_cvar_reform}
    \displaystyle\mathop{\text{minimize}}_{\bm{x} \in \mathcal{X}, \alpha \geq 0, w \in \mathbb{R}} \;\; \displaystyle c(\bm{x}) + \alpha \varepsilon + w + \frac{1}{N} \sum_{i = 1}^N Z_i(\bm{x}, \alpha, w),
    \end{equation}
    where, for each $i \in [N]$, we define the function $Z_i :\mathcal{X} \times \mathbb{R}_{+} \times \mathbb{R} \mapsto \mathbb{R}$ and the set $\mathcal{Z}_i$ as follows:
    \begin{subequations}
        \begin{align}
        &\displaystyle Z_i(\bm{x}, \alpha, w) = 
        \mathop{\text{maximize}}_{(\bm{\xi}, \bm{\lambda}, \bm{\Lambda}, \tau, \theta) \in \mathop{\mathrm{cl\,conv}}(\mathcal{Z}_i)}  \displaystyle \left\{\inner{\bm{T}(\bm{x})}{\bm{\Lambda}} + \bm{h}(\bm{x})^\top \bm{\lambda} - w \theta %
        - \alpha \tau \right\} \label{eq:Z_function_definition_cvar} \\
        &\displaystyle \mathcal{Z}_i = \left\{
        (\bm{\xi}, \bm{\lambda}, \bm{\Lambda}, \tau, \theta) \in \Xi \times \mathbb{R}^L_{+} \times \mathbb{R}^{L \times M} \times \mathbb{R}_{+} \times [0, p^{-1}]: 
        \begin{array}{l}
        \displaystyle \bm{\Lambda} = \bm{\lambda} \bm{\xi}^\top, \; (\bm{\xi} - \hat{\bm{\xi}}^{(i)}, \tau) \in \mathcal{C}^{M+1} \\
        \displaystyle (\bm{q}_0 + \bm{Q}\bm{\xi})\theta - \bm{W}_0^\top \bm{\lambda} - \sum_{j \in [M]} \bm{W}_j^\top \bm{\Lambda}\one_{j}  \in \mathcal{Y}^*
        \end{array}
        \right\}. \label{eq:Z_set_definition_cvar}
        \end{align}
    \end{subequations}
\end{theorem}
\begin{proof}
    Substituting~\eqref{eq:cvar_definition} into~\eqref{eq:two_stage_dro_cvar}, we obtain
    \begin{equation*}
    \mathop{\text{minimize}}_{\bm{x} \in \mathcal{X}, w \in \mathbb{R}} \;
    c(\bm{x})
    +
    w
    +
    \sup_{\mathbb{P} \in \mathcal{P}}
    \mathbb{E}_\mathbb{P} \left[
    \frac{1}{p}\max\left\{\mathcal{Q}(\bm{x}, \tilde{\bm{\xi}}) - w, 0\right\}
    \right].
    \end{equation*}
    We can thus interpret the above problem as an instance of the two-stage problem~\eqref{eq:two_stage_dro} where the second-stage loss function is given by $\frac{1}{p}\max\left\{\mathcal{Q}(\bm{x}, \bm{\xi}) - w, 0\right\}$.
    This, in turn, is equivalent to $\max_{\theta \in [0, p^{-1}]} \theta \cdot \left\{\mathcal{Q}(\bm{x}, \bm{\xi}) - w\right\} = \max_{\theta \in [0, p^{-1}]} \left\{\theta \mathcal{Q}(\bm{x}, \bm{\xi}) - \theta w\right\}$.
    By definition of the loss function, the first term is equal to $\inf_{\bm{y} \in \mathcal{Y}}
    \left\{
    \theta \bm{q}(\bm{\xi})^\top \bm{y} : \bm{W}(\bm{\xi}) \bm{y} \geq %
    \bm{T}(\bm{x}) \bm{\xi} + \bm{h}(\bm{x})
    \right\}$.
    An application of Theorem~\ref{thm:convex_reformulation} to this instance of problem~\eqref{eq:two_stage_dro} leads to the stated result. Details are omitted for brevity.
\end{proof}

We note that the sets $\mathcal{Z}_i$, $i \in [N]$ appearing in the statement of Theorem~\ref{thm:convex_reformulation_cvar} are similar to those appearing in Theorem~\ref{thm:convex_reformulation} except for the additional variable $\theta$ that multiplies the objective coefficients $(\bm{q}_0 + \bm{Q} \bm{\xi})$.
This can be exactly linearized by using McCormick inequalities since $\bm{\xi}$ is binary-valued and $\theta$ is bounded.
The other considerations remain exactly the same as in Sections~\ref{sec:linearization} and~\ref{sec:penalty_based_formulation}.
Therefore, we obtain MICP-representable sets $\mathcal{Z}_i$ even in the risk-averse setting.

%% file: model.tex
\section{Two-stage optimal power flow model}\label{appendix:model}
Our presentation of the first-stage model closely follows~\citet{kocuk2016strong}, whereas the second-stage model is inspired by~\citet{gocompetition}.
Conceptually, the first-stage problem determines power generation levels in the so-called \emph{base case}, where there are no line outages.
After transmission lines fail, the second-stage model may adjust the first-stage power generation levels subject to physical and engineering constraints where failed lines cannot be used.

Let $\mathcal{G}, \mathcal{B}$, and $\mathcal{M}$ be the set of generators, buses, and transmission lines, respectively, and let $\mathcal{G}_i$ be the set of generators associated with bus $i \in \mathcal{B}$. We define $\delta(i) \coloneqq \{j \in \mathcal{B}: (i,j) \in \mathcal{M} \textrm{ or } (j,i) \in \mathcal{M}\}$ to be the set of neighbors of bus $i \in \mathcal{B}$. Let $p_k^g$ and $q_k^g$ be the real and reactive power output of generator $k \in \mathcal{G}$, respectively, with lower and upper bounds denoted by $p^{\min}_k,p^{\max}_k$ and $q^{\min}_k,q^{\ max}_k$. We assume a linear cost $c_k$ of power generation for generator $k \in \mathcal{G}$. Real load and reactive load at bus $i \in \mathcal{B}$ are denoted by $p^d_i$ and $q^d_i$, respectively, and are known data. Let $p_{ij}^F$ and $q_{ij}^F$ be the real and reactive power flow on line $(i,j)$, respectively, defined for $(i,j) \in \mathcal{M}$ and $(j,i) \in \mathcal{M}$, with line rating limit $f_{ij}^{\max}$ (note that $f_{ij}^{\max} = f_{ji}^{\max}$). Let $\bm{Y}$ be the $\abs{B} \times \abs{B}$ complex-valued nodal admittance matrix, whose components are $Y_{ij} = G_{ij} + \i B_{ij}$, $\i = \sqrt{-1}$, where $G_{ij}$ and $B_{ij}$ are the conductance and susceptance of line $(i,j) \in \mathcal{M}$, respectively (see~\citet{zimmerman2016matpower} for details on computing $\bm{Y}$). We denote the real and imaginary parts of the complex voltage by $e_i$ and $f_i$, respectively. As in~\citet{kocuk2016strong}, we define new variables such that $c_{ii} = e_i^2 + f_i^2$, $c_{ij} = e_ie_j + f_if_j$ and $s_{ij} = e_if_j - e_jf_i$.
We define $\tilde{\bm{\xi}}$ to be a random binary vector with support $\Xi = \{0,1\}^{|\mathcal{M}|}$, where $\tilde{\xi}_{ij} =  1$ if line $(i,j) \in \mathcal{M}$ fails and 0 otherwise. \re{We have $\bm{x} = \left(\bm{p^g}, \ \bm{q^g}, \ \bm{p^F}, \ \bm{q^F}, \ \bm{c}, \ \bm{s}, \ \bm{\sigma^{p+}}, \ \bm{\sigma^{p-}}, \ \bm{\sigma^{q+}}, \ \bm{\sigma^{q-}}\right)$ as first-stage variables and $\bm{y} = \left(\bm{{\tilde{\delta}}}, \ \bm{{\tilde{p}^g}}, \ \bm{{\tilde{q}^g}}, \ \bm{{\tilde{p}^F}}, \ \bm{{\tilde{q}^F}}, \ \bm{{\tilde{c}}}, \ \bm{{\tilde{s}}}, \ \bm{{\tilde{\sigma}^{p+}}}, \ \bm{{\tilde{\sigma}^{p-}}}, \ \bm{{\tilde{\sigma}^{q+}}}, \ \bm{{\tilde{\sigma}^{q-}}}, \ \bm{{\tilde{\sigma}^{pF}}}, \ \bm{{\tilde{\sigma}^{qF}}}\right)$ as second-stage variables.} The two-stage model can be written as follows:
{\small%
\begin{subequations}
\label{eq:stage1}
\begin{align}
\re{\mathop{\text{minimize}}_{\substack{\bm{p^g}, \bm{q^g}, \bm{p^F}, \bm{q^F}, \bm{c}, \bm{s}\\ \bm{\sigma^{p+}}, \bm{\sigma^{p-}}, \bm{\sigma^{q+}}, \bm{\sigma^{q-}}}}} & \sum_{k \in \mathcal{G}} c_k p_k^g + \sum_{i \in \mathcal{B}} g_i \left( \sigma_i^{p+} + \sigma_i^{p-} + \sigma_i^{q+} + \sigma_i^{q-} \right) + \E_\P\left[ \mathcal{Q}(\bm{p^g},\tilde{\bm{\xi}})\right] \nonumber\\
\text{subject to} \quad & \sum_{k \in \mathcal{G}_i}p_k^g - p_i^d + \sigma_i^{p+} - \sigma_i^{p-} =  g_{ii}c_{ii} + \sum_{j \in \delta(i)} p_{ij}^F, && \forall i \in \mathcal{B}, \label{subeq:stage1.balance_p}\\
&\sum_{k \in \mathcal{G}_i}q_k^g - q_i^d + \sigma_i^{q+} - \sigma_i^{q-} = -b_{ii}c_{ii} + \sum_{j \in \delta(i)} q_{ij}^F, && \forall i \in \mathcal{B}, \label{subeq:stage1.balance_q}\\
&p^F_{ij} = -G_{ij}c_{ii} + G_{ij}c_{ij}+B_{ij}s_{ij}, && \forall (i,j), (j, i) \in \mathcal{M}, \label{subeq:stage1.pF}\\
&q^F_{ij} = B_{ij}c_{ii}-B_{ij}c_{ij} + G_{ij}s_{ij}, && \forall (i,j), (j, i) \in \mathcal{M},  \label{subeq:stage1.qF}\\
&c_{ij} = c_{ji}, \ s_{ij} = -s_{ji}, && \forall (i,j) \in \mathcal{M}, \label{subeq:stage1.var_change}\\
&c_{ij}^2 + s_{ij}^2 + \left( \frac{c_{ii} - c_{jj}}{2} \right)^2 \leq \left( \frac{c_{ii} + c_{jj}}{2} \right)^2, && \forall (i,j) \in \mathcal{M}, \label{subeq:stage1.socc}\\
&\underline{V}_i^2 \leq c_{ii} \leq \bar{V}_i^2, && \forall i \in \mathcal{B}, \label{subeq:stage1.bounds_c}\\
&p_k^{\min} \leq p_k^g \leq p_k^{\max}, && \forall k \in \mathcal{G}, \label{subeq:stage1.bounds_p}\\
&q_k^{\min} \leq q_k^g \leq q_k^{\max}, && \forall k \in \mathcal{G}, \label{subeq:stage1.bounds_q}\\
&(p^F_{ij})^2 + (q^F_{ij})^2 \leq (f^{\max}_{ij})^2, && \forall (i,j), (j,i) \in \mathcal{M}, \label{subeq:stage1.bounds_flow} \\
&\sigma_i^{p+}, \sigma_i^{p-}, \sigma_i^{q+}, \sigma_i^{q-} \geq 0, && \forall i \in \mathcal{B},
\end{align}
\end{subequations}}%
where $\mathcal{Q}(\bm{p^g},\tilde{\bm{\xi}})$ is the optimal value of
{\small
\begin{subequations}
\label{eq:stage2}
\begin{align}
\re{\mathop{\text{minimize}}_{\substack{%
\bm{{\tilde{\delta}}}, \bm{{\tilde{p}^g}}, \bm{{\tilde{q}^g}}, \bm{{\tilde{p}^F}}, \bm{{\tilde{q}^F}}, \bm{{\tilde{c}}}, \bm{{\tilde{s}}}\\ \bm{{\tilde{\sigma}^{p+}}}, \bm{{\tilde{\sigma}^{p-}}}, \bm{{\tilde{\sigma}^{q+}}}, \bm{{\tilde{\sigma}^{q-}}}\\ \bm{{\tilde{\sigma}^{pF}}}, \bm{{\tilde{\sigma}^{qF}}}}%
}} & \sum_{i \in \mathcal{B}} g_i \left( \tilde{\sigma}_i^{p+} + \tilde{\sigma}_i^{p-} + \tilde{\sigma}_i^{q+} + \tilde{\sigma}_i^{q-} \right) \nonumber\\
\text{subject to} \quad &\tilde{p}_k^g = p_k^g + \Delta_k\tilde{\delta} && \forall k \in \mathcal{G}, \label{subeq:stage2.correction}\\ 
& \sum_{k \in \mathcal{G}_i}\tilde{p}_k^g - p_i^d + \tilde{\sigma}_i^{p+} - \tilde{\sigma}_i^{p-} =  g_{ii} \tilde{c}_{ii} + \sum_{j \in \delta(i)} \tilde{p}^F_{ij}, && \forall i \in \mathcal{B}, \label{subeq:stage2.balance_p} \\
&\sum_{k \in \mathcal{G}_i}\tilde{q}_k^g - q_i^d + \tilde{\sigma}_i^{q+} - \tilde{\sigma}_i^{q-} =  -b_{ii}\tilde{c}_{ii} + \sum_{j \in \delta(i)} \tilde{q}^F_{ij}, && \forall i \in \mathcal{B}, \label{subeq:stage2.balance_q} \\
&\tilde{p}^F_{ij} = -G_{ij}\tilde{c}_{ii} + G_{ij}\tilde{c}_{ij}+B_{ij}\tilde{s}_{ij}   + \tilde{\sigma}_{ij}^{pF}, && \forall (i,j),  (j, i) \in \mathcal{M},  \label{subeq:stage2.pF} \\
&\tilde{q}^F_{ij} = B_{ij} \tilde{c}_{ii}-B_{ij}\tilde{c}_{ij} + G_{ij}\tilde{s}_{ij} + \tilde{\sigma}_{ij}^{qF}, && \forall (i,j), (j, i) \in \mathcal{M},  \label{subeq:stage2.qF} \\
&\tilde{c}_{ij} = \tilde{c}_{ji}, \ \tilde{s}_{ij} = -\tilde{s}_{ji}, && \forall (i,j) \in \mathcal{M}, \nonumber \\
&\tilde{c}_{ij}^2 + \tilde{s}_{ij}^2 + \left( \frac{\tilde{c}_{ii} - \tilde{c}_{jj}}{2} \right)^2 \leq \left( \frac{\tilde{c}_{ii} + \tilde{c}_{jj}}{2} \right)^2, && \forall (i,j) \in \mathcal{M}, \nonumber\\
&\underline{V}_i^2 \leq \tilde{c}_{ii} \leq \bar{V}_i^2, && \forall i \in \mathcal{B},  \nonumber \\
&p_k^{\min} \leq \tilde{p}_k^g \leq p_k^{\max},  && \forall k \in \mathcal{G}, \nonumber \\
&q_k^{\min} \leq \tilde{q}_k^g \leq q_k^{\max},  && \forall k \in \mathcal{G}, \nonumber \\
&(\tilde{p}_{ij}^{F})^2+(\tilde{q}_{ij}^{F})^2 \leq (f^{\max}_{ij})^2, && \forall (i,j),  (j, i) \in \mathcal{M}, \nonumber \\
& \re{\tilde{\xi}_{ij} = 1  \implies \left[\tilde{p}_{ij}^{F} = 0, \; \tilde{q}_{ij}^{F} = 0 \right],} && \re{\forall (i,j),  (j, i) \in \mathcal{M},} \label{subeq:stage2.obj_unc1}\\
& \re{\tilde{\xi}_{ij} = 0  \implies \left[\tilde{\sigma}_{ij}^{pF} = 0, \; \tilde{\sigma}_{ij}^{qF} = 0 \right],} && \re{\forall (i,j),  (j, i) \in \mathcal{M},} \label{subeq:stage2.obj_unc3}\\
&\tilde{\sigma}_i^{p+}, \tilde{\sigma}_i^{p-}, \tilde{\sigma}_i^{q+}, \tilde{\sigma}_i^{q-} \geq 0, && \forall i \in \mathcal{B} \nonumber .
\end{align}
\end{subequations}}%
Constraints \eqref{subeq:stage1.balance_p} and \eqref{subeq:stage1.balance_q} are the real and reactive power balance equations, respectively. %
Constraints \eqref{subeq:stage1.pF} and \eqref{subeq:stage1.qF} define the real and reactive power flow in both directions of all lines, respectively. Constraints \eqref{subeq:stage1.var_change} and \eqref{subeq:stage1.socc} model the change of variables (see~\citet{kocuk2016strong} for details), where the latter is the result of convexifying the original constraint $c_{ij}^2 + s_{ij}^2 = c_{ii}c_{jj}$. Constraints \eqref{subeq:stage1.bounds_c}, \eqref{subeq:stage1.bounds_p} and \eqref{subeq:stage1.bounds_q} enforce bounds on the voltage magnitude, and the real and reactive power generation, respectively. In the first stage, each generator $k \in \mathcal{G}$ has an associated generation cost $c_k$, and we penalize violating constraints \eqref{subeq:stage1.balance_p} and \eqref{subeq:stage1.balance_q} by $g_i$.

The second stage involves the same constraints with a few modifications. Namely, constraint \eqref{subeq:stage2.correction} adjusts the first-stage real power generation, where all generators are adjusted by a constant $\tilde{\delta}$, scaled by their predefined automatic generation control participation factor $\Delta_k$, also known as the droop control policy. We set participation factors $\Delta_k$ to be proportional to the generation capacity for each generator $k \in \mathcal{G}$. Constraints \eqref{subeq:stage2.obj_unc1} %
ensure that no power can flow through lines under contingency ($\xi_{ij} = 1$). %
Note that if line $(i,j)$ fails, then we must have $\tilde{p}_{ij}^F = \tilde{p}_{ji}^F = 0$ and $\tilde{q}_{ij}^F = \tilde{q}_{ji}^F = 0$, but variables $\tilde{c}_{ii}, \tilde{c}_{ij}$ and $\tilde{s}_{ij}$ should not be affected. Thus, unlike in the first stage, slacks $\tilde{\sigma}_{ij}^{pF}$ and  $\tilde{\sigma}_{ij}^{qF}$ are added in constraints \eqref{subeq:stage2.pF} and \eqref{subeq:stage2.qF}, respectively, so that \eqref{subeq:stage2.pF} and \eqref{subeq:stage2.qF} become redundant for $(i,j)$ and $(j,i)$ if line $(i,j)$ has failed. These slacks are active only if $\xi_{ij} = 1$, enforced by constraints \eqref{subeq:stage2.obj_unc3}. %
As in the first stage, the absolute values of the real and reactive power balance violation $\tilde{\sigma}_i^{p+} + \tilde{\sigma}_i^{p-}$ and $\tilde{\sigma}_i^{q+} + \tilde{\sigma}_i^{q-}$ for bus $i \in \mathcal{B}$, respectively, are penalized by $g_i$. Note that there is not cost on power generation in the second stage. We set the cost $g_i$ of violating the balance equations to be $\phi \cdot \max_{k \in \mathcal{G}} c_k$ (see Section~\ref{sec:test_instances}). %

%% file: mmcf_model.tex
\section{Two-stage multi-commodity network design model}\label{appendix:mmcf_model}
\re{
Our model is based on the fixed-charge multi-commodity network design problem which has been extensively studied in the literature, such as in~\citet{boland2016proximity} and~\citet{crainic2001bundle}. We are given a directed network with nodes $\mathcal{V}$, arcs $\mathcal{A}$, and a set of commodities $\mathcal{K}$ with known demands $d^k$. For commodity $k \in \mathcal{K}$, let $O_k$ be the origin node and $D_k$ the destination node. Each arc $(i,j)$ has an associated maximum capacity $u_{ij}$, per-unit cost of installing capacity $f_{ij}$ and per-unit cost of flow $c_{ij}$. For each node $i \in \mathcal{V}$, we define the set of neighboring nodes incident to outgoing and incoming arcs as $\mathcal{V}^+(i) = \{j \ | \ (i, j) \in \mathcal{A}\}$ and $\mathcal{V}^-(i) =\{j \ | \ (j, i) \in \mathcal{A}\}$, respectively. 

In the first stage, let $x_{ij}$ be a variable between 0 and 1 which determines the fraction of capacity of arc $(i, j)$ that can be used in the second stage. Note that this differs from a typical fixed-charge multi-commodity network design model, where $x_{ij}$ is a binary variable determining whether arc $(i, j)$ can be used, but becomes the same problem considered in Example \ref{ex:network}. In the second stage, let $y_{ij}^k$ be the amount of flow of commodity $k$ on arc $(i,j)$, and let $\sigma_k$ be the unsatisfied demand of commodity $k$ which we penalize by $g_k$. We consider node failures, and define $\bm{\xi}$ to be a binary vector where $\xi_i = 1$ indicates that node $i \in \mathcal{V}$ has failed. As in Example \ref{ex:network}, if a node fails, all arcs incident to it cannot be used. The model can be as written as follows:
\begin{align*}
    \mathop{\text{minimize}}_x \quad &\sum_{(i,j) \in \mathcal{A}} f_{ij} x_{ij} + 
    \sup_{\mathbb{P} \in \mathcal{P}}
    \mathbb{E}_\mathbb{P}\left[ \mathcal{Q}(\bm{x},\tilde{\bm{\xi}})\right] \\
    \text{subject to} \quad & x_{ij} \in [0, 1], \ \forall (i, j) \in \mathcal{A}.
\end{align*}
where $\mathcal{Q}(\bm{x},\tilde{\bm{\xi}})$ is defined as the optimal objective value of the following linear program:
\begin{subequations}
    \label{eq:mmcf_stage2}
    \begin{align}
        \mathop{\text{minimize}}_{y, \sigma} \quad & \sum_{(i,j) \in \mathcal{A}}\sum_{k \in \mathcal{K}} c_{ij} y_{ij}^k + \sum_{k \in \mathcal{K}}g_k \sigma_k \nonumber\\
        \text{subject to} \quad & \sum_{j \in \mathcal{V}^+(i)} y_{ij}^k - \sum_{j \in \mathcal{V}^-(i)} y_{ji}^k = \span\span \begin{cases}
                                                                                                    d^k - \sigma_k, &\ \text{if} \ i = O_k \\
                                                                                                    -d^k + \sigma_k, &\ \text{if} \ i = D_k \\
                                                                                                    0, &\ \text{otherwise} \\
                                                                                                \end{cases}, \ \forall k \in \mathcal{K}, \ \forall i \in \mathcal{V}, \label{subeq:mmcf_stage2.flow}\\
        &\sum_{k \in \mathcal{K}} y_{ij}^k \leq u_{ij}x_{ij}, && \forall (i, j) \in \mathcal{A}, \label{subeq:mmcf_stage2.capacity}\\
        &y_{ij}^k\leq d^k (1 - \tilde{\xi}_i), &&\forall (i,j) \in \mathcal{A}, \ \forall k \in \mathcal{K}, %
        \label{subeq:mmcf_stage2.contingency1}\\
        &y_{ij}^k\leq d^k (1 - \tilde{\xi}_j), &&\forall (i,j) \in \mathcal{A}, \ \forall k \in \mathcal{K}, %
        \label{subeq:mmcf_stage2.contingency2}\\
        &y_{ij}^k \geq 0 \ &&\forall (i,j) \in \mathcal{A}, \ \forall k \in \mathcal{K}, \nonumber\\
        &\sigma_k \geq 0 &&\forall k \in \mathcal{K}. \nonumber
    \end{align}
\end{subequations}
In the second stage, constraints \eqref{subeq:mmcf_stage2.flow} are typical network flow constraints, with added non-negative variables $\sigma_k$ which model unsatisfied demand. If we have a positive amount $\sigma_k$ of unsatisfied demand for commodity $k$, then an amount of $d^k - \sigma_k$ would leave the origin node and enter the destination node. Note that variables $\sigma_k$ ensure feasibility in the second stage for any $\bm{x}$ and $\bm{\xi}$. For each commodity $k$, unsatisfied demand $\sigma_k$ is penalized by $g_k$. In our experiments, $g_k$ is constant and set to $\phi \cdot \max_{(i,j) \in A} c_{ij}$, where $\phi = 1000$ is a pre-defined non-negative parameter. Constraint \eqref{subeq:mmcf_stage2.capacity} ensures the total flow on any arc $(i, j)$ does not exceed the available capacity, which is determined by the first-stage decision $x_{ij}$. Finally, constraints \eqref{subeq:mmcf_stage2.contingency1} and \eqref{subeq:mmcf_stage2.contingency2} ensure there is no flow through arcs incident to a failed node.
}

%% file: proofs.tex
\section{Proofs}\label{sec:proofs}

\begin{proof}[Proof of Theorem~\ref{thm:finite_sample_guarantee}.]
    First, note that for any $\mathbb{P}, \mathbb{Q} \in \mathcal{M}(\Xi)$, we have
    \begin{equation*}
    d_W(\mathbb{Q}, \mathbb{P}) \leq D d_{TV} (\mathbb{Q}, \mathbb{P}) \leq D \sqrt{d_{KL} (\mathbb{Q}, \mathbb{P})/2},
    \end{equation*}
    where $d_{TV}(\mathbb{Q}, \mathbb{P}) \coloneqq \frac12 \sup_{\bm{\xi} \in \Xi} \abs{\mathbb{P}(\bm{\xi}) - \mathbb{Q}(\bm{\xi})}$ and $d_{KL} (\mathbb{Q}, \mathbb{P}) \coloneqq \sum_{\bm{\xi} \in \Xi} \mathbb{Q}(\bm{\xi}) \log\left(\mathbb{Q}(\bm{\xi}) / \mathbb{P}(\bm{\xi}) \right) $ are the \emph{total variation} and \emph{Kullback-Liebler} distance between $\mathbb{Q}$ and $\mathbb{P}$, respectively.
    The first inequality follows from $d(\bm{\xi}', \bm{\xi}'') \leq D \mathbb{I}[\bm{\xi}' \neq \bm{\xi}'']$, whereas the second inequality is Pinkser's inequality.
    Now define $\Pi \coloneqq \left\{\mathbb{Q} \in \mathcal{M}(\Xi) : d_W(\mathbb{Q}, \mathbb{P}) > \varepsilon_N(\beta) \right\}$, where $\mathbb{P}$ is the true (unknown) distribution, and $\varepsilon_N(\beta)$ satisfies~\eqref{eq:epsilon_finite_sample}.
    Since all $\mathbb{Q} \in \Pi$ satisfy $d_W(\mathbb{Q}, \mathbb{P}) > D \sqrt{(2N)^{-1} \big( \abs{\Xi} \log(N+1) + \log \beta^{-1} \big)}$ by construction, we have:
    \begin{align}
    & d_{KL} (\mathbb{Q}, \mathbb{P}) > N^{-1} \big( \abs{\Xi} \log(N+1) + \log \beta^{-1} \big) \quad \forall \mathbb{Q} \in \Pi, \notag \\
    \iff & \inf_{\mathbb{Q} \in \Pi} d_{KL} (\mathbb{Q}, \mathbb{P})  > N^{-1} \big( \abs{\Xi} \log(N+1) + \log \beta^{-1} \big) \notag \\
    \iff & (N+1)^{\abs{\Xi}} \exp\left[-N \inf_{\mathbb{Q} \in \Pi} d_{KL} (\mathbb{Q}, \mathbb{P}) \right] < \beta \label{eq:finite_sample_temp_1}
    \end{align}
    An application of Sanov's theorem~\citep[inequality~2.1.12]{dembo_zeitouni} gives:
    \begin{align}
    \mathbb{P}^N\left[\hat{\mathbb{P}}_N \in \Pi \right]
    \leq 
    (N+1)^{\abs{\Xi}} \exp\left[-N \inf_{\mathbb{Q} \in \Pi} d_{KL} (\mathbb{Q}, \mathbb{P}) \right] \label{eq:finite_sample_temp_2}
    \end{align}
    Combining inequalities~\eqref{eq:finite_sample_temp_1} and~\eqref{eq:finite_sample_temp_2} along with the definition of $\Pi$, we have:
    \[
    \mathbb{P}^N\left[d_W(\hat{\mathbb{P}}_N, \mathbb{P}) > \varepsilon_N(\beta) \right] < \beta,
    \]
    which is equivalent to the probabilistic guarantee~\eqref{eq:epsilon_choice}.
\end{proof}

\begin{proof}[Proof of Theorem~\ref{thm:convex_reformulation}.]
    Under the stated assumptions of (A1) complete and (A2) sufficiently expensive recourse, strong duality holds between $\mathcal{Q}(\mb{x}, \mb{\xi})$ and its dual $\mathcal{Q}_d(\mb{x}, \mb{\xi})$.
    Along with the fact that $d(\bm{\xi}, \hat{\bm{\xi}}^{(i)}) = \norm{\bm{\xi} - \hat{\bm{\xi}}^{(i)}}$ is induced by a norm, the result from Lemma~\ref{prop:GK} allows us to equivalently reformulate the distributionally robust two-stage problem~\eqref{eq:two_stage_dro} in the form~\eqref{eq:two_stage_dro_reform}, where 
    \begin{equation*}
    \displaystyle Z_i(\mb{x}, \alpha) = 
    \sup_{\bm{\xi} \in \Xi}  \displaystyle \left\{\mathcal{Q}_d(\bm{x}, \alpha) - \alpha \norm{\bm{\xi} - \hat{\bm{\xi}}^{(i)}} \right\}.
    \end{equation*}
    By substituting the expression for $\mathcal{Q}_d(\bm{x}, \bm{\xi})$ from~\eqref{eq:Q_dual} and introducing the epigraphical variable $\tau$, we obtain
    \begin{equation*}
    \displaystyle Z_i(\mb{x}, \alpha) = 
    \mathop{\text{maximize}}_{\bm{\xi} \in \Xi, \bm{\lambda} \in \mathbb{R}^L_{+}, \tau\in\mathbb{R}_+}  \displaystyle \left\{
    \left[\bm{T}(\mb{x})\mb{\xi} + \bm{h}(\mb{x}) \right]^\top \bm{\lambda} - \alpha \tau : 
    \norm{\bm{\xi} - \hat{\bm{\xi}}^{(i)}} \leq \tau, \; \bm{q}(\mb{\xi}) - \bm{W}(\mb{\xi})^\top \bm{\lambda} \in \mathcal{Y}^*  \right\}.
    \end{equation*}
    Next, we \emph{(i)} use the affinity of $\bm{q}$ and $\bm{W}$: $\bm{q}(\mb{\xi}) = \bm{q}_0 + \bm{Q} \mb{\xi}$ and $\bm{W}(\mb{\xi}) = \bm{W}_0 + \sum_{j \in [M]} \xi_j \bm{W}_j$, \emph{(ii)} linearize the products $\bm{\lambda}\bm{\xi}^\top$ by setting them equal to (the new variable) $\bm{\Lambda}$, and \emph{(iii)} use the definition of the norm cone $\mathcal{C}^{M+1} = \left\{(\bm{\xi}, \tau) \in \mathbb{R}^{M} \times \mathbb{R}: \norm{\bm{\xi}} \leq \tau \right\}$ to obtain
    \begin{equation*}
    Z_i(\mb{x}, \alpha) = 
    \mathop{\text{maximize}}_{(\bm{\xi}, \bm{\lambda}, \bm{\Lambda}, \tau) \in \mathcal{Z}_i}  \displaystyle \left\{\inner{\bm{T}(\bm{x})}{\bm{\Lambda}} + \bm{h}(\bm{x})^\top \bm{\lambda} %
    - \alpha \tau \right\},
    \end{equation*}
    where $\mathcal{Z}_i$ is defined in~\eqref{eq:Z_set_definition}.
    The objective function of this maximization problem is linear in its decision variables.
    Therefore, we can equivalently replace the feasible region with the closure of its convex hull to obtain the stated reformulation.
\end{proof}

\begin{proof}[Proof of Theorem~\ref{thm:two_stage_dro_obj}.]
    First, we observe that under Assumption~(A3), the second-stage loss function, $\mathcal{Q}(\bm{x}, \bm{\xi})$, can be equivalently reformulated as follows:
    \begin{equation}\label{eq:q_temp}
    \mathcal{Q}(\bm{x}, \bm{\xi}) =
    \inf_{\bm{y} \in \mathcal{Y}, \bm{z} \in [0, 1]^M}
    \left\{
    \bm{q}(\bm{\xi})^\top \bm{y}: 
    \bm{W}_0 \bm{y} \geq \bm{T}_0 \bm{z} + \bm{h}(\bm{x}), \; (\bm{\mathrm{e}} - 2 \bm{\xi})^\top \bm{z} + \bm{\mathrm{e}}^\top \bm{\xi} \leq 0
    \right\}.
    \end{equation}
    To see this, observe that satisfaction of the last inequality is equivalent to satisfying $\bm{z} = \bm{\xi}$ since
    \begin{align*}
    (\bm{\mathrm{e}} - 2 \bm{\xi})^\top \bm{z} + \bm{\mathrm{e}}^\top \bm{\xi} \leq 0
    &\iff
    \sum_{i\in[M]} \left[\xi_i(1 - z_i) + (1 - \xi_i)z_i \right] \leq 0 \\
    &\overset{(\#)}{\iff}
    \sum_{i\in[M]} \abs{z_i - \xi_i} \leq 0 \\
    &\iff
    \norm{\bm{z} - \bm{\xi}}_1 \leq 0 \\
    &\iff
    \bm{z} = \bm{\xi} ,
    \end{align*}
    where the equivalence $(\#)$ follows from the fact that $\bm{z} \in [0, 1]^M$ and $\bm{\xi} \in \Xi \subseteq \{0, 1\}^M$.
    
    Next, we construct the Lagrangian dual of the problem~\eqref{eq:q_temp} with respect to the last inequality. Strong duality holds since the second-stage problem $\mathcal{Q}(\bm{x}, \bm{\xi})$ is strictly feasible and convex, under Assumption~(A2):
    \begin{equation*}
    \mathcal{Q}(\bm{x}, \bm{\xi}) =
    \sup_{\rho \geq 0} \mathcal{Q}^\rho(\bm{x}, \bm{\xi}).
    \end{equation*}
    As a function of the penalty parameter $\rho$, $\mathcal{Q}^\rho(\bm{x}, \bm{\xi})$ is concave and nondecreasing since $(\bm{\mathrm{e}} - 2 \bm{\xi})^\top \bm{z} + \bm{\mathrm{e}}^\top \bm{\xi} \geq 0$ whenever $\bm{z} \in [0, 1]^M$ and $\bm{\xi} \in \{0, 1\}^M$.
    Therefore, for a fixed choice of $\bm{x} \in \mathcal{X}$ and $\bm{\xi} \in \Xi$, it suffices to choose any value of $\rho$ that is greater than or equal to the optimal Lagrange multiplier of the last constraint in~\eqref{eq:q_temp}.
    The claim then follows from the compactness of $\mathcal{X}$ and $\Xi$.
\end{proof}

The proof of Theorem~\ref{thm:rho_computation} relies on the following technical lemma.
\begin{lemma}\label{lem:equality}
    For each $i\in[N]$, let $f_i :  \mathbb{R}_{+} \mapsto \mathbb{R}$ be a concave and non-decreasing function such that the supremum $\sup_{\rho \geq 0} f_i(\rho)$ is achieved for some finite $\rho$. Then, the following equality holds:
    \begin{align}
    \sum_{i\in[N]} \max_{\rho\geq 0} f_i(\rho) = \max_{\rho\geq 0} \sum_{i\in[N]} f_i(\rho).
    \end{align}
\end{lemma}

\begin{proof}
    The inequality $\geq$ is trivially true, and it implies that the maximization on the right-hand side is attained. 
    Next, we show that the inequality $\leq$ also holds.
    Let $\rho^* \in \mathop{\arg\max}_{\rho \geq 0} \sum_{i \in [N]} f_i(\rho)$.
    If $\rho^* \notin \mathop{\arg\max}_{\rho \geq 0} f_{i'}(\rho)$ for some $i' \in [N]$, then there exists $\hat{\rho} > \rho^*$ such that $f_{i'}(\hat{\rho}) > f_{i'}(\rho^*)$; and it follows from their monotonicity that $f_j(\hat{\rho}) \geq f_j(\rho^*)$ for all $j \in [N] \setminus \{i'\}$. This implies that $\sum_{i \in [N]} f_i(\hat{\rho}) > \sum_{i \in [N]} f_i(\rho^*)$, contradicting that $\rho^*$ is a maximizer of the right-hand side.
\end{proof}

\begin{proof}[Proof of Theorem~\ref{thm:rho_computation}.]
    The proof of Theorem~\ref{thm:two_stage_dro_obj} established that $\mathcal{Q}(\bm{x}, \bm{\xi}) = \max_{\rho \geq 0} \mathcal{Q}^\rho(\bm{x}, \bm{\xi})$. In conjunction with Lemma~\ref{prop:GK}, we can conclude that the distributionally robust two-stage problem~\eqref{eq:two_stage_dro} is equivalent to 
    \begin{subequations}
        \begin{align}
        \displaystyle&\min_{\bm{x} \in \mathcal{X}, \alpha \geq 0} \; \displaystyle c(\bm{x}) + \alpha \varepsilon + \sum_{i \in [N]} \max_{\rho \geq 0} \underbrace{\max_{\bm{\xi} \in \Xi} \frac{1}{N}\left\{ \mathcal{Q}^{\rho}(\bm{x}, \bm{\xi}) - \alpha \norm{\bm{\xi} - \hat{\bm{\xi}}^{(i)}} \right\}}_\text{\normalsize$\coloneqq f_i(\bm{x}, \alpha, \rho)$} \notag \\
        =&\min_{\bm{x} \in \mathcal{X}, \alpha \geq 0} \; \displaystyle c(\bm{x}) + \alpha \varepsilon + \max_{\rho \geq 0} \sum_{i \in [N]} f_i(\bm{x}, \alpha, \rho), \label{eq:thm3:1}
        \end{align}
        where the equality follows by applying Lemma~\ref{lem:equality} to $f_i(\bm{x}, \alpha, \cdot)$, $i \in [N]$: indeed, the mapping $\rho \mapsto \mathcal{Q}^\rho (\bm{x}, \bm{\xi})$ is concave and nondecreasing (see proof of Theorem~\ref{thm:two_stage_dro_obj}), and hence so is $f_i(\bm{x}, \alpha, \cdot)$.
        
        Theorem~\ref{thm:two_stage_dro_obj} also shows that there exists a finite $\bar{\rho} > 0$ such that $\mathcal{Q}^{\rho}(\bm{x}, \bm{\xi}) = \mathcal{Q}(\bm{x}, \bm{\xi})$ for all $\rho \geq \bar{\rho}$ and all $\bm{x}$ and $\bm{\xi}$.
        This implies that $\sum_{i \in [N]} f_i(\bm{x}, \alpha, \cdot)$ is maximized at $\bar{\rho}$, and thus we have
        \begin{align}
        & \eqref{eq:thm3:1} \notag \\
        \leq & \min_{\bm{x} \in \mathcal{X}, \alpha \geq 0} \; \displaystyle c(\bm{x}) + \alpha \varepsilon + \sum_{i \in [N]} f_i(\bm{x}, \alpha, \bar{\rho}) \label{eq:thm3:c} \\
        \leq &\max_{\rho \geq 0} \min_{\bm{x} \in \mathcal{X}, \alpha \geq 0} \; \displaystyle c(\bm{x}) + \alpha \varepsilon + \sum_{i \in [N]} f_i(\bm{x}, \alpha, \rho) \label{eq:thm3:d} \\
        = & \max_{\rho \geq 0} \min_{\bm{x} \in \mathcal{X}, \alpha \geq 0} \; \displaystyle c(\bm{x}) + \alpha \varepsilon + \frac{1}{N} \sum_{i \in [N]} \max_{\bm{\xi} \in \Xi} \left\{ \mathcal{Q}^{\rho}(\bm{x}, \bm{\xi}) - \alpha \norm{\bm{\xi} - \hat{\bm{\xi}}^{(i)}} \right\} . \label{eq:thm3:e}
        \end{align}
    \end{subequations}
    The inequality $\eqref{eq:thm3:c} \leq \eqref{eq:thm3:d}$ follows by treating \eqref{eq:thm3:c} as a function that is evaluated at $\rho = \bar{\rho}$ and \eqref{eq:thm3:d} as maximizing this function. %
    The max-min inequality implies $\eqref{eq:thm3:d} \leq \eqref{eq:thm3:1}$, and therefore, we have $\eqref{eq:thm3:1} = \eqref{eq:thm3:c} = \eqref{eq:thm3:d} = \eqref{eq:thm3:e}$.
    We point out that, unlike the classical minimax theorem, we did not exploit convexity of $\mathcal{X}$. Indeed, we only exploited the fact that each $f_i(\bm{x}, \alpha, \rho)$ is monotone in $\rho$ and the feasible region of $\rho$ is essentially compact because of Theorem~\ref{thm:two_stage_dro_obj}.
    
    We now show that it suffices to choose $\bar{\rho}$ as per the statement of the theorem.
    The key observation is that for any $\varepsilon \geq 0$, the expression inside $\max_{\rho \geq 0}$ in \eqref{eq:thm3:e} is bounded from above by the optimal value of the classical robust optimization problem with any uncertainty set $\Xi^0 \supseteq \Xi$:
    \begin{align*}
    \displaystyle& \max_{\rho \geq 0} \min_{\bm{x} \in \mathcal{X}} \; \displaystyle c(\bm{x}) + \max_{\bm{\xi} \in \Xi^0} \mathcal{Q}^{\rho}(\bm{x}, \bm{\xi})  \\
    \geq& \max_{\rho \geq 0} \min_{\bm{x} \in \mathcal{X}, \alpha \geq 0} \; \displaystyle c(\bm{x}) + \alpha \varepsilon + \frac{1}{N} \sum_{i \in [N]} \max_{\bm{\xi} \in \Xi} \left\{ \mathcal{Q}^{\rho}(\bm{x}, \bm{\xi}) - \alpha \norm{\bm{\xi} - \hat{\bm{\xi}}^{(i)}} \right\} \qquad \forall \varepsilon' \geq \varepsilon \geq 0.
    \end{align*}
    By a similar argument as before, the nondecreasing nature of the mapping $\rho \mapsto \mathcal{Q}^\rho (\bm{x}, \bm{\xi})$ guarantees that any $\rho$ of maximizing the left-hand side (i.e., the classical robust problem) must also maximize the right-hand side (i.e., the distributionally robust problem).
    The proof of Theorem~\ref{thm:two_stage_dro_obj} then shows that it suffices to choose a value that is at least as large as the optimal Lagrange multiplier $\rho^r$.
\end{proof}

\begin{proof}[Proof of Proposition~\ref{lem:disrupt_all_lines}.]
    Observe that $\mathcal{Q}(\bm{x}, \bm{\xi}^r) \geq \mathcal{Q}(\bm{x}, \bm{\xi})$ for all $\bm{\xi} \in \Xi$ and all $\bm{x} \in \mathcal{X}$. Indeed, the objective function of the problem on the left-hand side is greater than that on the right-hand side: $(\bm{q}_0 + \bm{Q}\bm{\xi}^r)^\top \bm{y} \geq (\bm{q}_0 + \bm{Q}\bm{\xi})^\top \bm{y}$ for all $\bm{y} \in \mathcal{Y} \subseteq \mathbb{R}^{N_2}_{+}$. Also, the feasible region of the problem on the left-hand side is a superset of the one on the right: $\bm{W}_0 \bm{y} \geq \bm{T}_0 \bm{\xi}^r + \bm{h}(\bm{x}) \geq  \bm{T}_0 \bm{\xi} + \bm{h}(\bm{x})$. Therefore, $\bm{\xi}^r$ is a worst-case realization of the parameters independent of the first-stage decision $\bm{x} \in \mathcal{X}$.
\end{proof}

\begin{proof}[Proof of Theorem~\ref{thm:exactness_for_zero_radius}.]
Let $v^\star_{\varepsilon, N}$ denote the optimal value of the distributionally robust two-stage problem~\eqref{eq:two_stage_dro} for a given sample $\{\hat{\bm{\xi}}^{(1)}, \ldots, \hat{\bm{\xi}}^{(N)}\}$ of size $N > 0$, and radius $\varepsilon \geq 0$.
Let $v^0_{\varepsilon, N}$ denote the optimal value of the convex hull reformulation~\eqref{eq:two_stage_dro_reform} when the convex hulls $\conv{\mathcal{Z}_i}$ in~\eqref{eq:Z_function_definition}--\eqref{eq:Z_set_definition} are approximated using the continuous relaxation $\mathcal{Z}_i^0$.
For simplicity, denote $\bm{z} = (\bm{\xi}, \bm{\lambda}, \bm{\Lambda}, \tau)$ in~\eqref{eq:Z_function_definition}--\eqref{eq:Z_set_definition}.
Then, observe that:
\begin{align*}
v^0_{\varepsilon_N, N} &= \min_{\bm{x} \in \mathcal{X}} c(\mb{x}) + \inf_{\alpha \geq 0} \alpha \varepsilon_N + \frac{1}{N} \sum_{i = 1}^N \sup_{\bm{z} \in \mathcal{Z}^0_i} \left\{\inner{\bm{T}(\bm{x})}{\bm{\Lambda}} + \bm{h}(\bm{x})^\top \bm{\lambda}
- \alpha \tau \right\} \\
&= \min_{\bm{x} \in \mathcal{X}} c(\mb{x}) + \inf_{\alpha \geq 0} \alpha \varepsilon_N + \sup_{(\bm{z}_1, \ldots, \bm{z}_N) \in \mathcal{Z}^0_1 \times \ldots \times \mathcal{Z}^0_N} \frac{1}{N} \sum_{i = 1}^N \left\{\inner{\bm{T}(\bm{x})}{\bm{\Lambda}_i} + \bm{h}(\bm{x})^\top \bm{\lambda}_i
- \alpha \tau_i \right\} \\
&= \min_{\bm{x} \in \mathcal{X}} c(\mb{x}) + \sup_{(\bm{z}_1, \ldots, \bm{z}_N) \in \mathcal{Z}^0_1 \times \ldots \times \mathcal{Z}^0_N} \frac{1}{N} \sum_{i = 1}^N \left\{\inner{\bm{T}(\bm{x})}{\bm{\Lambda}_i} + \bm{h}(\bm{x})^\top \bm{\lambda}_i\right\} + \inf_{\alpha \geq 0} \alpha \left[\varepsilon_N - \frac{1}{N} \sum_{i = 1}^N \tau_i \right] \\
&= \min_{\bm{x} \in \mathcal{X}} c(\mb{x}) + \sup_{(\bm{z}_1, \ldots, \bm{z}_N) \in \mathcal{Z}^0_1 \times \ldots \times \mathcal{Z}^0_N} \left\{ \frac{1}{N} \sum_{i = 1}^N \left\{\inner{\bm{T}(\bm{x})}{\bm{\Lambda}_i} + \bm{h}(\bm{x})^\top \bm{\lambda}_i \right\} : \sum_{i = 1}^N \tau_i \leq N \varepsilon_N\right\},
\end{align*}
where the equality on the first line follows by definition of $v^0_{\varepsilon, N}$; and
the second equation is obtained by interchanging the order of the summation and the supremum.
The third equality follows from a non-compact variant of Sion's minimax theorem~\cite[Theorem~2]{ha1981noncompact}, which is applicable since \textit{(i)} the objective function is linear in both $(\bm{z}_1, \ldots, \bm{z}_N)$ and $\alpha$, \textit{(ii)} their corresponding feasible regions are convex, and \textit{(iii)} the term $\sup_{(\bm{z}_1, \ldots, \bm{z}_N) \in \mathcal{Z}^0_1 \times \ldots \times \mathcal{Z}^0_N} \inf_{\alpha \geq 0}$ is bounded from below by
$
\sup_{(\bm{z}_1, \ldots, \bm{z}_N) \in \mathcal{K}^0_1 \times \ldots \times \mathcal{K}^0_N} \frac{1}{N} \sum_{i = 1}^N \big\{\inner{\bm{T}(\bm{x})}{\bm{\Lambda}_i} + \bm{h}(\bm{x})^\top \bm{\lambda}_i \big\} + \inf_{\alpha \in \{0\}} \alpha \big[\varepsilon_N - \frac{1}{N} \sum_{i = 1}^N \tau_i \big]
= \frac{1}{N} \sum_{i=1}^N \mathcal{Q}_d(\mb{x}, \hat{\mb{\xi}}^{(i)}),
$
where $\mathcal{K}^0_i \coloneqq \{(\bm{\xi}, \bm{\lambda}, \bm{\Lambda}, \tau) \in \mathcal{Z}^0_i : \bm{\xi} = \hat{\mb{\xi}}^{(i)}, \tau = 0 \}$ is a convex compact subset of $\mathcal{Z}^0_i$.
The last equation is due to $v^0_{\varepsilon_N, N}$ being finite; see Assumptions~\textbf{(A1)} and~\textbf{(A2)} from Section~\ref{sec:reformulation}.
Therefore, we have
\begin{align*}
\limsup_{N \to \infty} v^0_{\varepsilon_N, N} &=
\limsup_{N \to \infty} \left\{
\min_{\bm{x} \in \mathcal{X}} c(\mb{x}) + \sup_{(\bm{z}_1, \ldots, \bm{z}_N) \in \mathcal{Z}^0_1 \times \ldots \times \mathcal{Z}^0_N} \left\{ \frac{1}{N} \sum_{i = 1}^N \left\{\inner{\bm{T}(\bm{x})}{\bm{\Lambda}_i} + \bm{h}(\bm{x})^\top \bm{\lambda}_i \right\} : \sum_{i = 1}^N \tau_i \leq N \varepsilon_N\right\} \right\} \\
&= \limsup_{N \to \infty} \left\{
\min_{\bm{x} \in \mathcal{X}} c(\mb{x}) + \sup_{(\bm{z}_1, \ldots, \bm{z}_N) \in \mathcal{K}^0_1 \times \ldots \times \mathcal{K}^0_N} \frac{1}{N} \sum_{i = 1}^N \left\{\inner{\bm{T}(\bm{x})}{\bm{\Lambda}_i} + \bm{h}(\bm{x})^\top \bm{\lambda}_i \right\} \right\} \\
&= \lim_{N \to \infty} v^0_{0, N} = v^\star_{0, \infty},
\end{align*}
where $v^\star_{0, N}$ is the optimal value of the sample average approximation, and $v^\star_{0, \infty} \coloneqq \lim_{N \to \infty} v^\star_{0, N}$ is the optimal value of the original two-stage stochastic problem, which is guaranteed to exist because of \cite[Theorem~5.4]{shapiro2009lectures} along with Assumptions~\textbf{(A1)} and~\textbf{(A2)} from Section~\ref{sec:reformulation}.
Note that the second equation follows from the assumption that $N\varepsilon_N \to 0$ as $N \to \infty$, which further implies that $\sum_{i=1}^N \tau_i \leq 0$, leading to $\mb{\xi}_i = \hat{\mb{\xi}}^{(i)}$ for all $i \in [N]$.

Moreover, for any $N > 0$ and any $\varepsilon \geq 0$, it holds that
$
v^0_{\varepsilon, N} \geq v^\star_{\varepsilon, N} \geq v^\star_{0, N}.
$
Therefore, we have
\[
\liminf_{N \to \infty} v^0_{\varepsilon, N} \geq \liminf_{N \to \infty} v^\star_{\varepsilon, N} \geq \lim_{N \to \infty}  v^\star_{0, N} = v^\star_{0, \infty}.
\]

Combining these observations, we obtain
$
\lim_{N \to \infty} v^0_{\varepsilon_N, N} = \lim_{N \to \infty} v^\star_{\varepsilon_N, N} = v^\star_{0, \infty},
$
which proves the statement of the theorem.
\end{proof}